\documentclass[11pt]{article}

\usepackage{graphicx,color,amsfonts,amsmath,amssymb,enumerate,diagbox,tabularx}
\usepackage{url,fancyhdr,indentfirst}
\usepackage[colorlinks=true,citecolor=blue]{hyperref}
\usepackage[T1]{fontenc}
\usepackage[utf8]{inputenc}
\usepackage{authblk}
\usepackage{bm}
\usepackage{booktabs}
\usepackage{float}
\usepackage{multirow}
\usepackage{amsthm,natbib,comment}
\usepackage{mathrsfs}
\usepackage{ragged2e}
\usepackage{tikz-qtree}
\usepackage{amsmath,amssymb,amsfonts,amsthm,mathtools}

\newtheorem{theorem}{Theorem}[section]
\newtheorem{lemma}[theorem]{Lemma}
\newtheorem{prop}{Proposition}[section]
\newtheorem{remark}{Remark}[section]
\newtheorem{condition}{Condition}[section]
\newtheorem{corollary}{Corollary}[section]
\newtheorem{exm}{Example}[section]
\newtheorem{ex}{Example}

\newcommand{\R}{\mathbb{R}}
\newcommand{\EE}{\mathbb{E}}
\newcommand{\PP}{\mathbb{P}}

\newcommand{\abs}[1]{\left|#1\right|}

\newcommand{\Rm}[1]{{\rm #1}}

\DeclarePairedDelimiter\ceil{\lceil}{\rceil}
\DeclarePairedDelimiter\floor{\lfloor}{\rfloor}

\theoremstyle{remark}
\newtheorem{definition}[theorem]{Definition}

\title{Indirect Statistical Inference with Guaranteed\\ Necessity and Sufficiency }

\author[a,b,c]{{Zhengjun} {Zhang}}
\author[d]{{Xinyang} {Hu}}
\author[e]{{Chuyang} {Lu}}
\author[a]{{Tianying} {Liu}\thanks{Corresponding author, Email: liutianying23@mails.ucas.ac.cn}}
\affil[a]{School of Economics and Management, University of Chinese Academy of Sciences}
\affil[b]{AMSS Center for Forecasting Science, Chinese Academy of Sciences}
\affil[c]{Department of Statistics, University of Wisconsin, Madison}
\affil[d]{Department of Statistics, Yale University}
\affil[e]{Academy of Mathematics and System Sciences, Chinese Academy of Sciences}

\begin{document}
\maketitle
\begin{abstract}
This paper develops a new framework for indirect statistical inference with guaranteed necessity and sufficiency, applicable to continuous random variables. We prove that when comparing exponentially transformed order statistics from an assumed distribution with those from simulated unit exponential samples, the ranked quotients exhibit distinct asymptotics: the left segment converges to a non-degenerate distribution, while the middle and right segments degenerate to one. This yields a necessary and sufficient condition in probability for two sequences of continuous random variables to follow the same distribution. Building on this, we propose an optimization criterion based on relative errors between ordered samples. The criterion achieves its minimum if and only if the assumed and true distributions coincide, providing a second necessary and sufficient condition in optimization. These dual NS properties, rare in the literature, establish a fundamentally stronger inference framework than existing methods. Unlike classical approaches based on absolute errors (e.g., Kolmogorov–Smirnov), NSE exploits relative errors to ensure faster convergence, requires only mild approximability of the cumulative distribution function, and provides both point and interval estimates. Simulations and real-data applications confirm NSE’s superior performance in preserving distributional assumptions where traditional methods fail.
\end{abstract}

\noindent{\bf keywords:} combinatorial mathematics, indirect inference, relative errors, simulated order statistics.

\section{Introduction}\label{sec_introduction}
In statistical inference, distributional assumptions of continuous random variables play a central role in model building. For unknown parameters, many estimation methods have been proposed, including maximum likelihood estimation (MLE), least squares (OLS), method of moments (MM), generalized method of moments (GMM) \citep{hansen1982large}, generalized estimating equations (GEE) \citep{hardin2002generalized}, and resampling-based methods such as the bootstrap. These procedures all begin with a distributional assumption and construct estimators accordingly. If the assumed model is correct, the resulting estimators are consistent and efficient. In this sense, they can be regarded as sufficient estimations. However, they are not necessary estimations, because they do not guarantee that the fitted distribution coincides with the assumed one. Even when supplemented with test statistics such as the Kolmogorov–Smirnov or Anderson–Darling tests, inference remains tied to “if” conditions under correct specification, rather than necessary and sufficient criteria.

This limitation is not trivial. MLE, MM, and the bootstrap have long been regarded as pillars of modern statistical inference, supporting both theory and practice. Yet they remain fundamentally sufficient-only procedures. Once the proposed Necessary and Sufficient Estimation (NSE) framework is available, these classical methods are no longer strictly indispensable. NSE integrates distributional preservation directly into the estimation process, providing a more rigorous foundation for statistical inference and offering particular value in AI and machine learning, where reliability and interpretability depend on valid distributional assumptions.

Although the literature contains extensive work on goodness-of-fit procedures (see \citep{Aslam2021}, for a recent overview), existing parameter estimation methods rarely guarantee preservation of distributional assumptions after marginal transformations of fitted models. Consider linear and semiparametric regression, where error terms are often assumed to be normally distributed. When data deviate from normality, the power transformation of \citep{box1964analysis}  is frequently applied to stabilize variance and approximate Gaussianity. Yet such transformations often fail to maintain distributional assumptions, especially under nonlinear covariate–response relationships or missing covariates. Contributions from these effects can be absorbed into observed variables and error terms. We demonstrate with real data examples in Section  \ref{sec:realdata}  that power transformations alone are insufficient. In contrast, by minimizing a relative distance defined in this paper, the normality of error terms can always be preserved.

Beyond regression, many statistical applications face even greater challenges. In extreme value theory, MLE of some distributions may be inefficient or fail to exist; see  \citep{matthews1982testing} and \citep{cheng1995non} for examples in threshold and changepoint models, and \cite{atkinson1985plots} for Box–Cox transforms with shifting parameters. Numerical approaches such as maximum simulated likelihood \citep{bhat2001quasi} and approximate likelihood \citep{breslow1993approximate} have been enabled by modern computing, but they do not resolve the fundamental problem of preserving distributional assumptions.

This paper proposes a new inference framework for continuous random variables that avoids reliance on strict regularity conditions and explicit density forms. The central idea is that order statistics from observed data and from simulated samples under the assumed distribution should exhibit similar ranked quotients. We prove a fundamental theorem: when partitioning the quotients into three regions, the left part converges to a non-degenerate distribution, while the middle and right parts degenerate to one. This leads to a necessary and sufficient condition in probability for two sequences of continuous random variables to follow the same distribution.

Building on this probabilistic foundation, we construct an optimization criterion based on relative errors between ordered observed and simulated samples. Minimizing this criterion yields a second necessary and sufficient condition in optimization, since the objective function attains its minimum if and only if the assumed and true distributions coincide. These dual NSE results—one in probability, one in optimization—are, to the best of our knowledge, unique in the literature and represent a significant departure from existing methods. Moreover, relative errors yield faster convergence rates than absolute-error approaches such as OLS or Kolmogorov–Smirnov tests.

The rest of the paper is organized as follows. Section  \ref{sec:theory} introduces the concept of maximum ranked quotients (MRQ) and examines their asymptotic properties. Section \ref{sec:inference} develops point and interval estimation procedures under the NSE framework. Section  \ref{sec:simulation} presents simulation studies demonstrating the performance of NSE, particularly in cases where MLE fails or OLS assumptions are violated. Section \ref{sec:realdata} applies the method to real datasets from different fields, highlighting improvements over classical approaches. Section \ref{sec:conclusion} concludes with a discussion of applications and comparisons with other estimation methods. Proofs and supplementary figures are provided in the Appendix.

\section{Theory of Ranked Quotients}\label{sec:theory}
\subsection{Motivations}\label{motiv}
We use the following semiparametric regression model to illustrate our results.
\begin{equation}\label{Eqn:semi}
  Y_{i,j} = f_1(X_i) + f_2(X_i) + \varepsilon_{i,j},\ i=1,\dots,n;\ j=1,\dots,m
\end{equation}
where the subscript $j$ stands for the $j$th lab, the subscript $i$ stands for the $i$th experiment, $X_i$ is a $p$-dimensional covariate vector at the $i$th experiment, $Y_{i,j}$ is the response at the $i$th experiment from $j$th lab, $\varepsilon_{i,j}$ are error terms following a pre-specified distribution, $f_1(.)$ is a parametric function, and $f_2(.)$ is a nonparametric function. Note that at the $i$th experiment, all labs use the same $X_i$, e.g., the same component materials (conditions). Suppose we only observe $Y_{i,1}$, $i=1,\dots,n$, and we treat $j=2,\dots,m$ as computer experiments.

Suppose the goal is to estimate parameters in $f_1(.)$ and $\varepsilon_{i,j}$, and the form of $f_2(.)$ such that $(\widehat{\varepsilon}_{1,1},\dots,\widehat{\varepsilon}_{n,1})$ is approximately distributed the same as $({\varepsilon_{1,1}},\dots,{\varepsilon_{n,1}})$, and $\sum_{i=1}^{n}\widehat{\varepsilon}_{i,1}^2$ is minimized, where $\widehat{\varepsilon}_{i,1}=Y_{i,1}-\widehat{f}_1(X_i)-\widehat{f}_2(X_i)$ with $\widehat{f}_1(.)$ and $\widehat{f}_2(.)$ being the estimates of $f_1(.)$ and $f_2(.)$ respectively.

Roughly speaking, we note that $(\widehat{\varepsilon}_{1,1},\dots,\widehat{\varepsilon}_{n,1})$ is approximately distributed the same as $({\varepsilon_{1,1}},\dots,{\varepsilon_{n,1}})$ implies $(\widehat{\varepsilon}_{1,1},\dots,\widehat{\varepsilon}_{n,1})$ is approximately distributed the same as $({\varepsilon_{1,j}},\dots,{\varepsilon_{n,j}})$ for all $j$. Suppose $({\varepsilon_{(1),j}},\dots,{\varepsilon_{(n),j}})$ is the order statistics of $({\varepsilon_{1,j}},\dots,{\varepsilon_{n,j}})$. \cite{ZQM2011} proved that $({\varepsilon_{(1),j}},\dots,{\varepsilon_{(n),j}})$ is distributed the same as $({\varepsilon_{(1),k}},\dots,{\varepsilon_{(n),k}})$ for any $j$ and $k$, and then $(\widehat{\varepsilon}_{(1),j},\dots,\widehat{\varepsilon}_{(n),j})$ is distributed the same as $(\widehat{\varepsilon}_{(1),k},\dots,\widehat{\varepsilon}_{(n),k})$ for any $j$ and $k$.

With the pre-specified distribution of $\varepsilon_{i,1}$ and the estimated $(\widehat{\varepsilon}_{(1),j},\dots,\widehat{\varepsilon}_{(n),j})$, we can re-arrange  $\{Y_{i,1}\}$ corresponding to ordered $(\widehat{\varepsilon}_{(1),j},\dots,\widehat{\varepsilon}_{(n),j})$ as $\{Y_{i^*,1}\}$ (not necessarily ordered). As such, $\{Y_{i^*,1}\}$ are thought to be generated from $({\varepsilon_{(1),j}},\dots,{\varepsilon_{(n),j}})$ for any $j$.  A natural idea is to compare the distribution of $(\widehat{\varepsilon}_{(1),j},\dots,\widehat{\varepsilon}_{(n),j})$ with simulated $({\varepsilon_{(1),j}},\dots,{\varepsilon_{(n),j}})$ for all $j$.  Next question is how to measure the distance between $(\widehat{\varepsilon}_{(1),j},\dots,\widehat{\varepsilon}_{(n),j})$ and $({\varepsilon_{(1),j}},\dots,{\varepsilon_{(n),j}})$.

If we started with classical statistical treatments, we would consider the paired absolute distances as
\begin{equation}\label{Eqn: paridistance}
  \max_{1\le i\le n}\mid {\widehat{\varepsilon}_{(i),j}} - {\varepsilon_{(i),j}}\mid,\ j=1,\dots,m.
\end{equation}
Note that when $ \varepsilon_{i,j}$ is normally distributed, the convergence rate (in distribution) of $\max_{1\le i\le n}\widehat{\varepsilon}_{(i),j}$ is at $(2\log(n))^{-1/2}$, which is slower than $1/\sqrt{n}$. For any $j$, $\{\widehat{\varepsilon}_{(i),j} - {\varepsilon_{(i),j}},\ i=1,\dots,n\}$ is a dependence sequence, the conditions, e.g., $\Delta$-conditions, $D$-conditions in the classical extreme value theory, are not applicable. In addition, the approaches developed in quantile processes are not directly applicable.

Motivated by the concepts of quotient correlation and rank transformation in \cite{zhang2008quotient} and \cite{ZQM2011}, to measure the relationship between two samples, the samples may be ranked first and then the quotients between the ranked samples may be investigated. Proposition 4 in \cite{zhang2017random} Page 697 shows that using the relative errors of bivariate normal random variables in the definition of quotient correlation coefficient, the convergence rate (in distribution) of the coefficient is at $1/(2n\sqrt{1-\rho^2}/\pi)$  which is much faster than the rate $(2\log(n))^{-1/2}$ of the maxima of a univariate normal sequence, where $\rho$ is the Pearson linear correlation coefficient between two random variables. We thus, introduce the definition.
\begin{definition}
	If $\{X_i,i=1\ldots,n\}$ and $\{Y_i,i=1\ldots,n\}$ are random samples of two positive random variables $X$ and $Y$ respectively. Let $\Lambda$ be an index set belonging to $\{1,\ldots,n\}$. Then,
	\begin{equation}
		\label{mrqdf}
		q^{(1)}_{\Lambda}=\max_{i\in\Lambda}\frac{X_{(i)}}{Y_{(i)}}\text{ and }q^{(2)}_{\Lambda}=\max_{i\in\Lambda}\frac{Y_{(i)}}{X_{(i)}}
	\end{equation}
	are the maximum ranked quotient (MRQ) on $\Lambda$ with $\{X_{(i)}:X_{(1)}\leq X_{(2)}\leq\cdots\leq X_{(n)}\}$ and $\{Y_{(i)}:Y_{(1)}\leq Y_{(2)}\leq\cdots\leq Y_{(n)}\}$ being order statistics of $\{X_i,i=1,\ldots,n\}$ and $\{Y_i,i=1,\ldots,n\}$ respectively.
\end{definition}

To get a direct impression of MRQ, the empirical cumulative distribution function of $q^{(1)}_{\Lambda}$ for $\Lambda=\{1,2,\ldots,n\}$ and the Q-Q plot are drawn in Figure \ref{mqr} for two independent unit exponential samples $\{X_1,X_2,\ldots,X_n\}$ and $\{Y_1,Y_2,\ldots,Y_n\}$ over sample sizes $n=10,\ 100,\ 1,000,\ 10,000,\ 100,000$. The two plots show that it is very likely that $q^{(1)}_{\Lambda}$ has a non-degenerate limit distribution. Another discovery is that when $n\to\infty$, the probability that $q^{(1)}_{\Lambda}<1$ should converge to 0. These two observations are formally stated and proved in Section \ref{sec: basic}.
\begin{figure}[]
	\begin{minipage}[]{0.49\linewidth}
		\centering
		\includegraphics[width=2.8in]{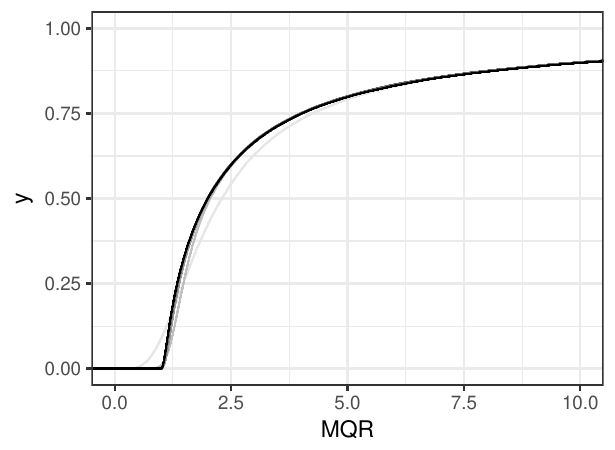}
	\end{minipage}
	\begin{minipage}[]{0.49\linewidth}
		\centering
		\includegraphics[width=2.8in]{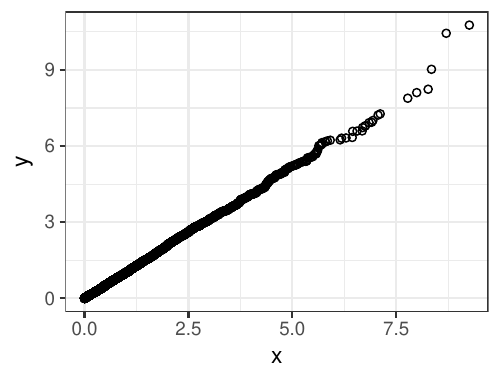}
	\end{minipage}
	\caption{The left panel is the empirical cumulative distribution function of maximum ranked quotient between two independent unit exponential samples. The sample sizes increases from 10 to 100,000 as the line color darkens. The right panel is the Q-Q plot of two independent unit exponential samples with sample size 10,000. }
	\label{mqr}
\end{figure}

It is clear that if two samples are generated from the same distribution, their respective order statistics must have the same distribution. MRQ can certainly be used to measure how close the two order statistics are, i.e., how strong their co-movements are. The extreme case is that the elements in the two samples are identical when $\Lambda=\{1,2,\ldots,n\}$ and  $q^{(1)}_{\Lambda}=q^{(2)}_{\Lambda}=1$ hold. As such, MRQ can characterize the closeness of a model to its distribution assumption based on collected data and data simulated from the assumed distribution. It is worth noting that there are some other measures/tests for distribution assumption, such as the empirical distribution function whose property is given by Glivenko-Cantelli theorem \citep{durrett2010probability}. Section \ref{sec:conclusion} offers more discussions.
\subsection{Basic Properties}\label{sec: basic}
Suppose $X$ and $Y$ are independent identically distributed random variables. Suppose $\{(X_i,Y_i), i=1,\ldots,n\}$ is a bivariate random sample of $(X,Y)$. Let $\{X_{(i)}:X_{(1)}\le X_{(2)}\le\cdots\le X_{(n)}\}$ and $\{Y_{(i)}:Y_{(1)}\le Y_{(2)}\le\cdots\le Y_{(n)}\}$ be order statistics of $\{X_i,i=1,\ldots,n\}$ and $\{Y_i,i=1,\ldots,n\}$ respectively. Then there are $2n$ quotients: $\{X_{(i)}/Y_{(i)},i=1,\ldots,n\}$ and $\{Y_{(i)}/X_{(i)},i=1,\ldots,n\}$, which are called ranked quotients. Define
\begin{equation}
	\label{def:mrq}
	q^{(1)}_n=\max_{1\leq i\leq n}X_{(i)}/Y_{(i)}\quad\text{and }\quad q^{(2)}_n=\max_{1\leq i\leq n}Y_{(i)}/X_{(i)},
\end{equation}
as the maximum ranked quotients between $\{X_i,i=1,\ldots,n\}$ and $\{Y_i,i=1,\ldots,n\}$.

For ease of notation, $q_n$ may be used to denote $q^{(1)}_n$ or $q^{(2)}_n$ if no confusion is caused. It is well known that any absolutely continuous random variable can be transformed into another random variable with the desired distribution. Without loss of generality, suppose now  $\{(X_i,Y_i), i=1,\ldots,n\}$ is a bivariate random sample of $(X,Y)$ with $X_i$ and $Y_i$ having independent unit exponential distributions. Then the following asymptotic conclusion can be drawn:
\begin{prop}
	\label{prop1}
	If $X$ and $Y$ are independent and have unit exponential distribution, and $(X_i,Y_i), i=1,\ldots,n,$ is a random sample from $(X,Y)$, then
	\begin{equation}
		\label{def:prop1}
		\lim\limits_{n\to\infty}P(q^{(1)}_n>1)=\lim\limits_{n\to\infty}P(q^{(2)}_n>1)=1.
	\end{equation}
\end{prop}
A proof of Proposition \ref{prop1} is deferred to the Appendix.

\begin{remark}
	\label{remark1}
	It is worth noting that the reciprocal of a unit exponentially distributed random variable is a unit Fr\'{e}chet random variable. Thus, Proposition \ref{prop1} still holds when $X$ and $Y$ are independent and have unit Fr\'{e}chet distribution. As a result, considering unit exponential (Fr\'{e}chet) distribution in quotient based studies is convenient.
\end{remark}

Under the conditions in Proposition \ref{prop1}, $\{X_{(i)}/Y_{(i)},i=1,\ldots,n\}$ and $\{Y_{(i)}/X_{(i)},i=1,\ldots,n\}$ are two sets of realizations derived from the order statistics of $n$ independent unit exponentially distributed random variables. From Proposition \ref{prop1}, the deviation of $q_n$ from $1$ which can be rewritten as:
\begin{equation}
	\label{def:relative}
	q_n-1=\max_{1\leq i\leq n}\frac{X_{(i)}-Y_{(i)}}{Y_{(i)}},
\end{equation}
is asymptotically positive. In \eqref{def:relative}, $q_n-1$ is actually the biggest relative error of $\{X_{(i)},i=1,\ldots,n\}$ from $\{Y_{(i)},i=1,\ldots,n\}$. It measures the nonlinear similarity of the two realizations of order statistics, which should be quite close to each other since they have the same joint distribution.

From a graphical point of view, $q_n-1$ is the maximum difference between the theoretical slope (i.e., 1) and the slope of the line connecting the origin and points on the Q-Q plot of $\{X_{(i)},i=1,\ldots,n\}$ and $\{Y_{(i)},i=1,\ldots,n\}$. All the points lie between two lines passing through origin with slopes $q^{(1)}_n$ and $1/q^{(2)}_n$ respectively in the first quadrant. It is well known that when the two cumulative distribution functions are identical, the Q-Q plot will be the main diagonal, i.e., the 45-degree line through the origin. Thus, the closer $q^{(1)}_n$ and $q^{(2)}_n$ are to 1, the more concentrated these points are around the main diagonal in the Q-Q plot. This idea is illustrated in Figure \ref{mqr}.
We note that the Q-Q plot is often used to plot the empirical or estimated percentiles against the theoretical percentiles. The visual impressions of Q-Q plot can be misleading or not informative in many applications. The Glivenko-Cantelli approximation results from the empirical distribution or the estimated distribution to the theoretical distribution can not be directly converted to the inverse of distributions. In the literature, Lemma 2 of the Supplementary Materials of \cite{zhang2017random} established the condition for the inverse of estimated distributions to be arbitrarily close to their theoretical counterparts. Back to 1953, \cite{Rnyi1953OnTT} discussed and established the approximation results of a distribution function by its sample related distribution using relative errors of distribution functions. Most recently, \cite{McCullagh21} studies a likelihood analysis of quantile-matching transformations between the observed response values to the quantiles of a given target distribution. \cite{TengZhang21} uses the relative error to introduce relative treatment effects in regression models for medical data, and the resulting AbRelaTEs model is much more advanced than the classical logistic regression as it greatly enhances interpretability and truly being personalized medicine computability.

As detailed in Section \ref{motiv}, this paper considers not only the relative errors of order statistics, but also simulated observations along with observations. Consequently, the existing results dealing with absolute errors cannot be applied. The paper intends to establish a new theoretical foundation and a new statistical inference method (indirect inference) which solves the issues discussed in the abstract and introduction.

\subsection{Limit Behavior of Ranked Quotients}
\label{subsec:limit}
Proposition \ref{prop1} gives a general result about the maximum ranked quotient. In this section, the behaviors of ranked quotients are explored based on different indices,
i.e., the limit behaviors of $\max_{i\in\Lambda}X_{(i)}/Y_{(i)}$ are assessed according to different choices of the index set $\Lambda$. First, the index set which satisfies the following assumption is considered.
\begin{condition}
	\label{cond:1}
	Assume $0<\alpha_n<\beta_n<1$, such that
	\begin{equation}
		\label{condition1}
		\min\{\alpha_n, 1-\beta_n\} / \sqrt{\frac{\log n}{n}} \to \infty.
	\end{equation} Let
	\begin{equation}
		\label{setmid}
		\Lambda_m=\{i\in\mathbb{N}:\alpha_nn<i<\beta_nn\}
	\end{equation}
\end{condition}
\noindent where the subscript $m$ stands for the ``middle" part of the ordered sample, and similarly without any ambiguity, the subscripts $l$ and $r$ stand for the ``left'' part and the ``right'' part of the ordered sample respectively throughout the paper.

Condition \ref{cond:1} is quite mild. A wide range of choices for $\alpha_n$ and $\beta_n$ can satisfy this condition. It is obvious that if $\alpha_n$ and $\beta_n$ are positive constants, they satisfy Condition \ref{cond:1}. The index set defined in Condition \ref{cond:1} is the middle part of the entire sequence. $\alpha_n$ and $\beta_n$ represent the truncated proportion at the two ends of the sample. Though, the proportion of the parts left in the two ends to the whole sequence can be as small as approaching to 0, the number of indices left always approaches infinity as $n$ increases as long as $\alpha_n, \beta_n$ satisfy Condition \ref{cond:1}, i.e.,
$$\alpha_nn\to\infty\text{ and }(1-\beta_n)n\to\infty,\text{ as }n\to \infty.$$
\begin{figure}[]
	\begin{minipage}[]{0.49\linewidth}
		\centering
		\includegraphics[width=2.8in]{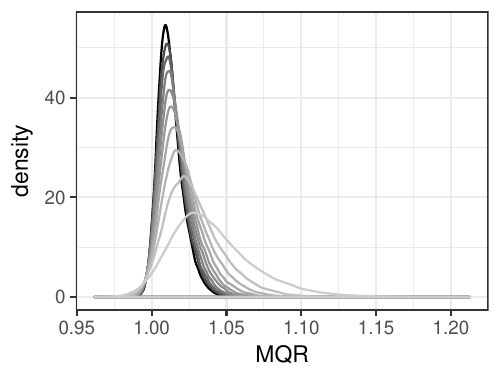}
	\end{minipage}
	\begin{minipage}[]{0.49\linewidth}
		\centering
		\includegraphics[width=2.5in]{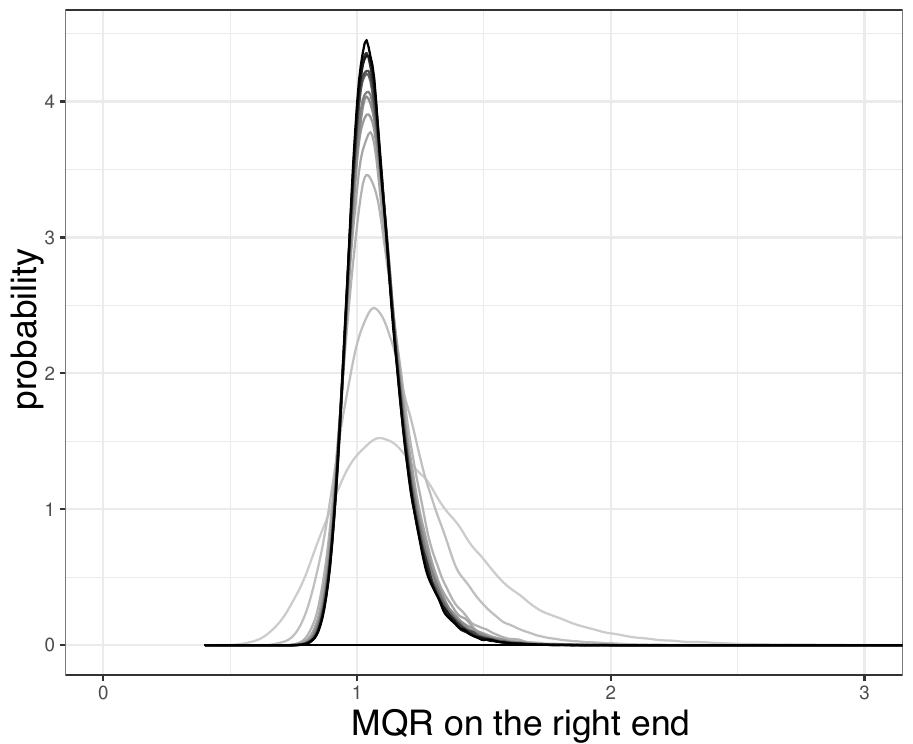}
	\end{minipage}
	\caption{Left panel: Kernel Density of $q_{n,\Lambda_m}$ between two independent unit exponential samples with $\alpha_n=\beta_n=0.1$. The sample sizes increases from 10,000 to 100,000 as the line color darkens. Right panel: Kernel Density of $q_{n,\Lambda_r}$ between two independent unit exponential samples with $k=5$. The sample sizes increases from 100 to 100,000 as the line color darkens. }
	\label{pdfleft}
\end{figure}
Theorem \ref{thm:mid} gives the limit behavior of $\max_{i\in\Lambda}X_{(i)}/Y_{(i)}$ under Condition \ref{cond:1}.

\begin{theorem}
	\label{thm:mid}
	Suppose $\{X_i,i=1,\ldots,n\}$ and $\{Y_i,i=1,\ldots,n\}$ are independent samples from two independent unit exponential random variables $X$ and $Y$ respectively. Let $\{X_{(i)}:X_{(1)}\le X_{(2)}\le\cdots\le X_{(n)}\}$ and $\{Y_{(i)}:Y_{(1)}\le Y_{(2)}\le\cdots\le Y_{(n)}\}$ be order statistics of $\{X_i,i=1,\ldots,n\}$ and $\{Y_i,i=1,\ldots,n\}$ respectively.
	Assume the index set $\Lambda_m$ satisfies Condition \ref{cond:1}. Then
	\begin{equation}
		\label{qmid}
		q_{n,\Lambda_m}\triangleq\max_{i\in\Lambda_m}\{X_{(i)}/Y_{(i)}\} \overset{P}{\to} 1,\text{ as }n\to\infty.
	\end{equation}
\end{theorem}

Theorem \ref{thm:mid} is a law of large numbers for $q_{n,\Lambda_m}$. The convergence rate of $q_{n,\Lambda_m}$ is faster than $1/n$. See the proofs in Appendix for more details.

In Figure \ref{pdfleft} left panel, the density plot of $q_{n,\Lambda_m}$ is drawn by simulation study with $\alpha_n$ and $\beta_n$ being the 10th and 90th percentile respectively. According to the order in which the colors are darkened, the sample size increases from 10,000 to 100,000 by the interval of 10,000. It can be seen that $q_{n,\Lambda_m}$ converges to 1 which is proven in Theorem \ref{thm:mid}.

Theorem \ref{thm:mid} shows that the middle part of the sequence of ranked quotients converge to 1 uniformly in probability. \eqref{qmid} can also be written in the form of the maximum relative error:
\begin{equation}
	\label{qmidequi}
	q_{n,\Lambda_m}-1=\max_{i\in\Lambda_m}\frac{X_{(i)}-Y_{(i)}}{Y_{(i)}} \overset{P}{\to} 0,\text{ as }n\to\infty.
\end{equation}
Combined with Proposition \ref{prop1} which implies $P(q_{n,\Lambda_m}-1>0)\to 1$, Theorem \ref{thm:mid} indicates that the relative errors for $i\in\Lambda_m$ converges to 0 uniformly in probability. It is a strong result since the size of $\Lambda_m$ increases to infinity as $n\to\infty$. It controls the relative distance between $\{X_i\}$ and $\{Y_i\}$ over a wide range. It is consistent with our intuition that as the sample size increases, the data distributes denser on the real line due to that $X_{(i)}$ and $Y_{(i)}$ follow the same distribution for $i=1,\ldots,n.$

From a different viewpoint, the conclusion in Theorem \ref{thm:mid} can be expressed with thresholds instead of index set as follows.
\begin{corollary}
	\label{thm:threshold}
	If $0<u_n<v_n$ satisfy as $n\to\infty$
	\begin{equation}
		\label{eq:uv}
		\min\{F(u_n),1-F(v_n)\}/\sqrt{\frac{\log n}{n}}\to \infty,
	\end{equation}
	where $F(x)=1-e^{-x}$ is the distribution function of unit exponential distribution, then
	\begin{equation}
		q^{(u_n,v_n)}_{n}\triangleq\max_{1\leq i\leq n}\frac{\max\{\min\{v_n,X_{(i)}\},u_n\}}{\max\{\min\{v_n,Y_{(i)}\},u_n\}} \xrightarrow{P} 1.
	\end{equation}
\end{corollary}
In the corollary, threshold values instead of the rank are used to truncate the sample. The condition \eqref{eq:uv} in the corollary is quite similar to Condition \ref{cond:1} only with $\alpha_n$ and $\beta_n$ replaced by $F(u_n)$ and $1-F(v_n)$. $\alpha_n$ and $\beta_n$ are the empirical cumulative probability of the $\alpha_nn$th and $\beta_nn$th elements in the ordered sample. $F(u_n)$ and $1-F(v_n)$ can be viewed as the theoretical cumulative probability of the $\alpha_nn$th and $\beta_nn$th order statistics. Therefore, the equivalence between Theorem \ref{thm:mid} and Corollary \ref{thm:threshold} can be understood from the fact that as $n\to\infty$, the empirical cumulative distribution function uniformly converges to the theoretical distribution function.

The next two theorems deal with the limit behaviors of the two ends of the sequence of ranked quotients.

\begin{theorem}
	\label{thm:right}
	Suppose $\{X_{(i)},i=1,\ldots,n\}$ and $\{Y_{(i)},i=1,\ldots,n\}$ are two samples of the order statistics of n independent unit exponential random variables.
	Then, for any arbitrary positive integer $k$, as $n\to\infty$,
	\begin{equation}
		\label{qright}
		q_{n,\Lambda_r}\triangleq\max_{i\in\Lambda_r}\{X_{(i)}/Y_{(i)}\}\overset{P}{\to}1,
	\end{equation}
	where $$\Lambda_r=\{i\in\mathbb{N}:n-k+1\leq i\leq n\}.$$
\end{theorem}

Theorem \ref{thm:right} is the law of large numbers for $q_{n,\Lambda_r}$. Compared with increasing index set $\Lambda_m$ in Theorem \ref{thm:mid}, the size of index set $\Lambda_r$ is unchanged in Theorem \ref{thm:right}. The fixed number of ranked quotients on the right end of the sequence converge uniformly to 1 in probability. In Figure \ref{pdfleft} right panel, it can be observed that $q_{n,\Lambda_r}$ converges to 1.

For the left end, the following theorem holds.

\begin{theorem}
	\label{thm:left}
	Suppose $\{X_{(i)},i=1,\ldots,n\}$ and $\{Y_{(i)},i=1,\ldots,n\}$ are two samples of the order statistics of n independent unit exponential random variables.
	Then, for any arbitrary positive integer $\ell$, let $\Lambda_l=\{i\in\mathbb{N}: 1\leq i\leq \ell\}$, then $q_{n,\Lambda_l}\triangleq\max_{i\in\Lambda_l}\{X_{(i)}/Y_{(i)}\}$ converges weakly to a random variable with a non-degenerate distribution:
$$ \mathcal{R}_{\ell}(t)=\lim_{n \to \infty} \mathcal{R}_{n,\ell}(t) = \lim_{n \to \infty}P_{n}(\max_{1 \leq i \leq \ell}\{\frac{X_{(i)}}{Y_{(i)}}\}\leq t), t \geq 0, \ell \in \mathbb{Z}^{+}, \ell \leq n$$
and
$$
R_{\ell}(t)=\frac{\sum_{j=0}^{\ell-1}(-1)^j(B_{\ell-1,j}\cdot k^{\ell-j-1})}{k^{2\ell-1}}, k=1+1/t, t\geq 0
$$
where $B_{\ell,j}=\frac{\binom{2\ell+2}{\ell-j}\binom{\ell+j}{\ell}}{\ell+1}$.
\end{theorem}

Theorem \ref{thm:left} describes the limit behavior of the maxima of a fixed number of ranked quotients on the left hand of the sequence. Instead of converging to 1, it converges to a non-degenerate random variable.
This phenomenon may also be observed from Figure \ref{mqr}.
A numerical study generates the probability density function of the distribution in Figure \ref{pdfleftden}. The proof of this theorem is much involved with combinatorial mathematics and twelve lemmas and theorems. The proof is given in Appendix

\begin{figure}[]
	\begin{minipage}[]{0.49\linewidth}
		\centering
		\includegraphics[width=2.8in]{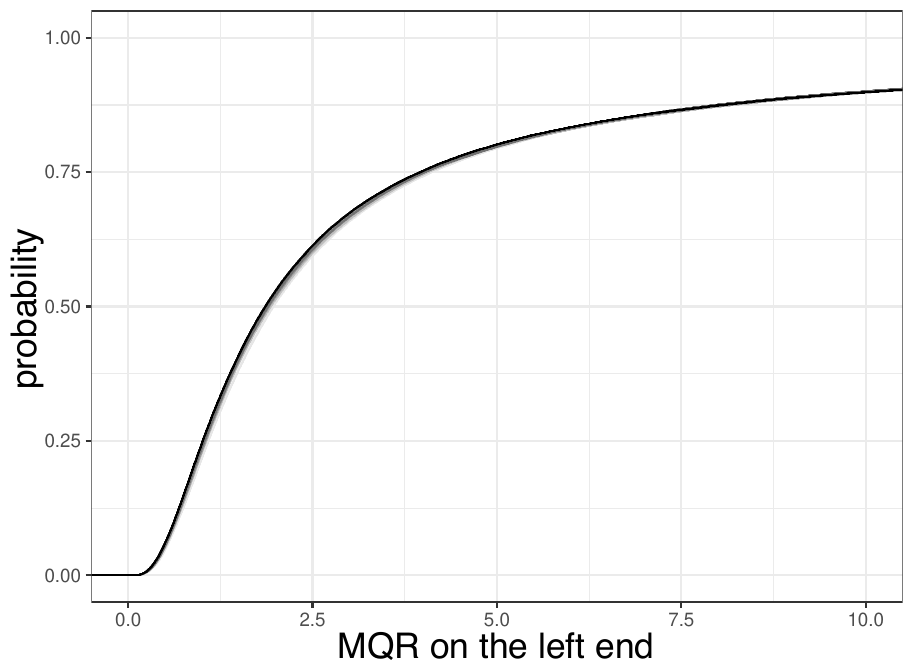}
	\end{minipage}
	\begin{minipage}[]{0.49\linewidth}
		\centering
		\includegraphics[width=2.8in]{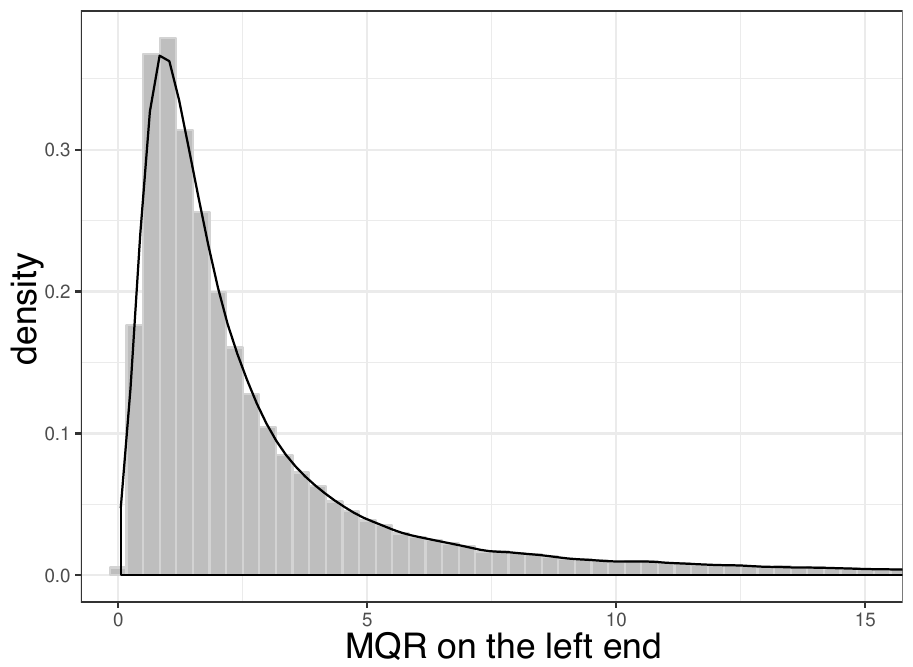}
	\end{minipage}
	\caption{The left panel is the empirical cumulative distribution function of maximum ranked quotient between two independent unit exponential samples on the left end with $\Lambda=\{1,2,3,4,5\}$. The sample sizes are 10, 15, 20, 25, 100, 1000 as the line color darkens. The right plot is the empirical density function of the data with sample size 1000. }
	\label{pdfleftden}
\end{figure}
In the left panel of Figure \ref{pdfleftden}, it is obvious that the distribution of $q_{n,\Lambda_l}$ converges quickly to a non-degenerate random variable as Theorem \ref{thm:left} states. According to Figure \ref{pdfleftden} and the extremely high empirical kurtosis and skewness, the distribution of $q_{n,\Lambda_l}$ is heavy-tailed. This observation results from the property of the function $f(x,y)=x/y$ around the origin. It is much more sensitive to the value $y$ than $x$, thus may generate some very large quotients when $y$ is smaller than $x$ by a small value.

Observing Figure \ref{mqr} again, we can see that Theorem \ref{thm:left} and Figure \ref{pdfleftden} indicate that the relative distance has a heavy-tailed distribution on the left end, while one cannot tell any visual differences of the left-end points distributed on the diagonal in the right panel of Figure \ref{mqr}. This is because the distance observed in the Q-Q plot is the absolute distance, not relative distance. This argument also applies to the right end where the absolute distance is large and the relative distance is proved to converge to 0 in Theorem \ref{thm:right}. This phenomenon suggests that the visual impression of Q-Q plots can be inaccurate.

Theorems \ref{thm:mid}-\ref{thm:left} show the left part, the middle part, and the right part have different limit forms and different convergence rates. As a result, the classical extreme value theory is not applicable to three parts together for them to weakly converge to a well-defined random variable. We note that
examples of maxima that do not have an extreme value limit distribution include Poisson random variables and random variables with an absorbing probability at the right end support points.

Before closing this section, we consider the case where $\{X_{(i)},i=1,\ldots,n\}$ and $\{Y_{(i)},i=1,\ldots,n\}$ are not sampled from unit exponential distribution. The following theorem deals with a deviation of the assumed distribution from the exponential distribution, i.e., we assume an arbitrary continuous smooth distribution with positive support for at least one of the two random variables $X$ and $Y$.
\begin{theorem}
	\label{thm:notexp}
	Suppose $\{X_{(i)},i=1,\ldots,n\}$ and $\{Y_{(i)},i=1,\ldots,n\}$ are the order statistics sampled from two different distributions with monotone increasing, continuous cumulative distribution functions $F$ and $G$. Then the following two limits cannot hold simultaneously
	\begin{equation}
		\label{twoq}
		q^{(1)}_{n,\Lambda_m}\triangleq\max_{i\in\Lambda_m}\{X_{(i)}/Y_{(i)}\}\overset{P}{\to}1\quad\text{and}\quad q^{(2)}_{n,\Lambda_m}\triangleq\max_{i\in\Lambda_m}\{Y_{(i)}/X_{(i)}\}\overset{P}{\to}1,\quad\text{as } n\to\infty
	\end{equation}
	where $$\alpha_n=\sqrt{\frac{\log n\log\log n}{n}},\quad \beta_n=1-\alpha_n$$ in $\Lambda_m$.
\end{theorem}
Theorem \ref{thm:notexp} proves that $\{X_{(i)},i=1,\ldots,n\}$ and $\{Y_{(i)},i=1,\ldots,n\}$ are sampled from the same distribution is the necessary condition for \eqref{twoq}, while may not be sufficient condition since we only prove the converse result for exponential distribution. On the other hand, the following Theorem \ref{cor:iff} provides a necessary and sufficient result when both distributions are unit exponential distribution.
\begin{theorem}
	\label{cor:iff}
	If $\{X_{(i)},i=1,\ldots,n\}$ is an ordered sample from a unit exponential distribution, then
	\eqref{twoq}\text{ holds} if and only if $\{Y_{(i)},i=1,\ldots,n\}$  is an ordered sample from a unit exponential distribution.
\end{theorem}

The proof of Theorem \ref{cor:iff} is straightforward using the results from the established theoretical results in earlier theorems. In Appendix \ref{A5}, we present a special case of that both $X$ and $Y$ are exponentially distributed but with different rate parameters. Then \eqref{twoq} does not hold.

\section{The inference}\label{sec:inference}
According to Theorem \ref{thm:mid}, given a (simulated) random sample of unit exponential random variable $\{X_i,i=1,\ldots,n\}$ and another sample $\{Y_i,1=1,\ldots,n\}$ of $Y$, if $q_{n,\Lambda_m}$ is only slightly bigger than 1, it is very likely that $\{Y_i,1=1,\ldots,n\}$ is generated by the unit exponential distribution ($Exp(1)$). Thus, an inference procedure for estimating the parameters of different models can be proposed. Suppose that the distribution of $Y$ is $P(Y\leq y)=F(y,\theta^*)$ which relies on the parameter vector $\theta^*$. Then, we define
\begin{equation}
	g\left(a,b\right)\coloneqq\max\left\{ a,b,\frac{1}{a},\frac{1}{b}\right\},\qquad a>0,\ b>0.
\end{equation}
To estimate $\theta^*$, we set the estimator to be
\begin{equation}
	\label{loss}
	\widehat{\theta}\triangleq\arg\min_{\theta}g\left(\max_{k\in\Lambda}\frac{X_{(k)}}{G^{-1}(F(Y_{(k)},\theta))},\max_{k\in\Lambda}\frac{G^{-1}(F(Y_{(k)},\theta))}{X_{(k)}}\right),
\end{equation}
where $G(x)=1-e^{-x}$ is the cumulative distribution function of $Exp(1)$. Here, $G^{-1}(F(Y_{(k)},\theta))$ is used to change the scale of $\{Y_i,1=1,\ldots,n\}$ to unit exponential. We have the following proposition.
\begin{prop}\label{estnse}
  $g(a,b)$ is minimized if and only if $a=b=1$.
\end{prop}
The proof is straightforward and is omitted.

 According to Theorem \ref{thm:notexp} and Theorem \ref{cor:iff},
$\widehat{\theta}$ is a consistent estimator of $\theta^*$, which is stated in the following proposition.
\begin{prop}
\label{prop:consistent}
Suppose $\{X_{(i)},i=1,\ldots,n\}$ is an ordered sample from a unit exponential distribution, and $\{Y_{(i)},i=1,\ldots,n\}$  is an ordered sample from $P(Y\leq y)=F(y,\theta^*)$ which relies on the parameter vector $\theta^*$. Then the $\widehat{\theta}$ is a consistent estimator of $\theta^*$.
\end{prop}

We note that the rank-preserving scale regeneration and conditional expectation test for testing independence in \cite{ZQM2011} does not involve parameter estimation. In \cite{ZQM2011}, the observed values are replaced by the simulated marginal specific values with the same rank corresponding to both observed and simulated. In this paper, we are concerned with the parameter estimation and the distribution-preserving properties with the estimated parameter values. Also as mentioned earlier, in \cite{zhang2017random}, parameters are estimated by assumed distribution using available estimation methods, e.g., the maximum likelihood, or other methods. In this paper, the estimation of parameter values is based on (\ref{loss}) directly. Compared with maximum likelihood estimation (MLE), the significant advantage of the new procedure is that it does not rely on the existence of explicit probability density function as long as the cumulative distribution function can be approximated arbitrarily accurately, and the validity of the regularity conditions that are required for the usual asymptotic properties associated with the maximum likelihood estimator. In \eqref{loss}, as long as $F(Y,\theta)$ can be calculated, the method is applicable. As a result, the new method is significantly different from those existing methods in the literature. Given that the estimation is based on a necessary and sufficient result in Theorem \ref{cor:iff}, it is natural to call our new estimation method as necessary and sufficient estimation (NSE).

For a given sample $\{Y_i,i=1,\ldots,n\}$, NSE can be performed many times, and then the estimation which makes \eqref{loss} smallest among all the estimations is used as the final estimation. As such, multiple sequences of simulated order statistics $\{X_{(i)}, i=1,\ldots,n,\}$ can be generated to compare with rescaled $\{Y_i,i=1,\ldots,n\}$, and the parameter values that perform rescaling $Y_i$s and make their relative distance smallest are chosen. Alternatively, each of multiple sequences of simulated order statistics $\{X_{(i)}, i=1,\ldots,n\}$ leads to a set of NSE estimated values for the parameter $\theta^*$, and then a 100(1-$\alpha$)\% confidence set for the parameter can be constructed, i.e., by taking lower and upper $\alpha/2$ percentiles of each estimated component parameter values. This idea is formalized in the following Theorem \ref{procedure} whose proof is deferred to Appendix.

\begin{theorem}
\label{procedure}
Given a sample $\{Y_i,i=1,\cdots,n\}$, we repeat the following procedure for $m$ times: For $j$-th time with $j=1,\dots,m$, we generate an ordered sample $\{X_{(i)}^j, i=1,\ldots,n\}$ from a unit exponential distribution, then calculate the estimate $\widehat{\theta}^{j}$ of the parameter $\theta^*$ by \eqref{loss}. Then for a sufficiently large $m$ or $m$ tending to infinity, a 100(1-$\alpha$)\% approximate confidence set for the parameter $\theta_i, i=1,\cdots,p$ can be constructed by
$$\left[\widehat{\theta}_{i,(\alpha m/2)},\widehat{\theta}_{i,((1-\alpha/2)m)}\right]$$
for $i=1,\cdots,p$.
Here, $\widehat{\theta}_{i,(k)}$ is the order statistics of $\{\widehat{\theta}^j_i, j=1,\cdots,m\}$.
\end{theorem}

\section{Simulation Study} \label{sec:simulation}
In this section, we present several examples to illustrate the efficiency of our NSE methods. In some examples, the existing estimation methods may not be applicable.
\subsection{Regressions with different error distributions}
\label{subsec:reg}
For the regular linear regression model:
\begin{equation}
	\label{linreg}
	y_i=\mu+x_i^T\beta+\sigma*\varepsilon_i,\ i=1,\ldots,n,
\end{equation}
where $\mu\in\R, x_i\in\R^p, \beta\in\R^p$ and $\varepsilon_i$'s are independent identically distributed generations of some distribution to be specified later, ordinary least squares (OLS hereafter) is the dominant method and proven to be optimal in the class of linear unbiased estimators. To apply the new NSE method, let
\begin{equation}
	\label{losslinreg}
	(\widehat{\mu},\widehat{\beta})=\arg\min_{(\mu,\beta)}g\left(\max_{k\in\Lambda}\frac{X_{(k)}}{G^{-1}(F(r_{(k)},\theta))},\max_{k\in\Lambda}\frac{G^{-1}(F(r_{(k)},\theta))}{X_{(k)}}\right)
\end{equation}
which is defined by \eqref{loss} with $\{Y_i,i=1,\ldots,n\}$ replaced by $\{r_i=y_i-\mu-x_i^T\beta,i=1,\ldots,n\}$. Note that the lower case $\{x_i,i=1,\ldots,n\}$ are observed covariate values throughout the paper, i.e., they are different from simulated sequence $\{X_{(i)},i=1,\ldots,n\}$.

The criterion used in OLS is the sum of the squares of the differences between the observed dependent variable in the given dataset and those predicted by the linear function. For the NSE method, the objective function for minimization is the difference between 1 and the maximum ranked quotients between a unit exponential sample and rescaled residuals of the model. According to the limit behaviors of maximum ranked quotients, the criterion adopted by the NSE method is actually the deviation of the empirical distribution of the residuals from the specified distribution ($F(\cdot,\theta)$) used in \eqref{loss}. Thus, the objective function is minimized only when $F(\cdot,\theta)$ is close enough to the empirical distribution of the residuals.

It is clear that if the distribution of $\varepsilon$ is known or an assumption is made about the $\varepsilon$, the new NSE method is applicable. In particular, there are chances that the linear model is incorrect for a real dataset, or some covariates are not observable. In this case, $\varepsilon$ is not independently normally distributed anymore and may suffer from exogeneity and autocorrelation. By applying the NSE method, we can seek an estimation that makes the residuals approximately normally distributed.

The NSE method leads to an NSE based test of normality for model checking. This new test calculates the loss function defined in \eqref{losslinreg} and makes decision by checking whether it falls into the Monte Carlo $95\%$ interval of the loss function with ${r_i,i=1,\cdots,n}$ being normally distributed. In addition to the NSE based test, we also apply Jarque-Bera Test \citep{jarque1987test} and Lilliefors Test \citep{lilliefors1967kolmogorov} to further evaluate the normality of the residuals.

In the simulation, we estimate the linear model \eqref{linreg} in several examples with different underlying true models. The estimated coefficients of the new NSE method are the mean value calculated from estimating the coefficients by using the NSE method multiple times in the same dataset. The influence of randomness in generating $\{X_i, i=1,2,\cdots,n\}$ used in \eqref{loss} can be decreased in this way.

\begin{exm}
	\label{stanlin}
	We consider the linear model (\ref{linreg}) with different distributions of $\{\varepsilon_{i}\}_{i=1}^{n}$ specified in the first columns in Tables \ref{tab:testols} and \ref{tab:testnew}.
	Since the model to be fitted is the standard linear model, the model used in estimation is only correctly specified when the distribution of $\{\varepsilon_{i}\}_{i=1}^{n}$ is $N(0,1)$.
\end{exm}
In the simulation, we let $p=5,\ \sigma=0.5,\ \mu=0$ and $\beta=(1,3,5,3,1)^T$. The sample size is set to be 300 and the replication number is 100. In each replication, $x_{ij},\ j=1,\cdots,5$ are generated independently identically from standard normal distribution $N(0,1)$. OLS and the new NSE method are applied to estimate the parameters $\beta$. The three tests for normality mentioned above are employed to test the normality of the residuals. We list the number of replications which pass these tests in Tables \ref{tab:testols} and \ref{tab:testnew} for each method. We can see that only when the model is correctly specified, the normality assumption of the residuals can be well preserved by OLS. While in other cases, the residuals generated by our NSE method are much closer to normal.

 \begin{table}[]
		\caption{Linear model: the number of OLS estimations passing the test out of 100 tests.}
	\label{tab:testols}
\begin{center}
	\begin{tabular}{|c|c|c|c|}
		\hline
		& NSE based test & Lilliefors test & Jarque-Bera test \\ \hline
		$N(0,1)$ & 100             & 97                & 96                 \\ \hline
		$Exp(1)$ & 0               & 0                 & 0                  \\ \hline
		$t(2)$   & 0               & 0                 & 0                  \\ \hline
		$t(4)$   & 47              & 6                 & 1                  \\ \hline
		$t(8)$   & 91              & 66                & 29                 \\ \hline
	\end{tabular}
\end{center}
\end{table}

\begin{table}[]
	\caption{Linear model: the number of estimations generated by the new NSE method and passing the test out of 100 tests.}
\label{tab:testnew}
\begin{center}
	\begin{tabular}{|c|c|c|c|}
		\hline
		& NSE based test & Lilliefors test & Jarque-Bera test \\ \hline
		$N(0,1)$ & 100                      & 100                        & 97                          \\ \hline
		$Exp(1)$ & 98                       & 69                         & 51                          \\ \hline
		$t(2)$   & 60                       & 53                         & 26                          \\ \hline
		$t(4)$   & 98                       & 80                         & 43                          \\ \hline
		$t(8)$   & 100                      & 95                         & 71                          \\ \hline
	\end{tabular}
\end{center}
\end{table}

\begin{figure}[]
	\begin{minipage}[]{0.49\linewidth}
		\centering
		\includegraphics[width=2.3in]{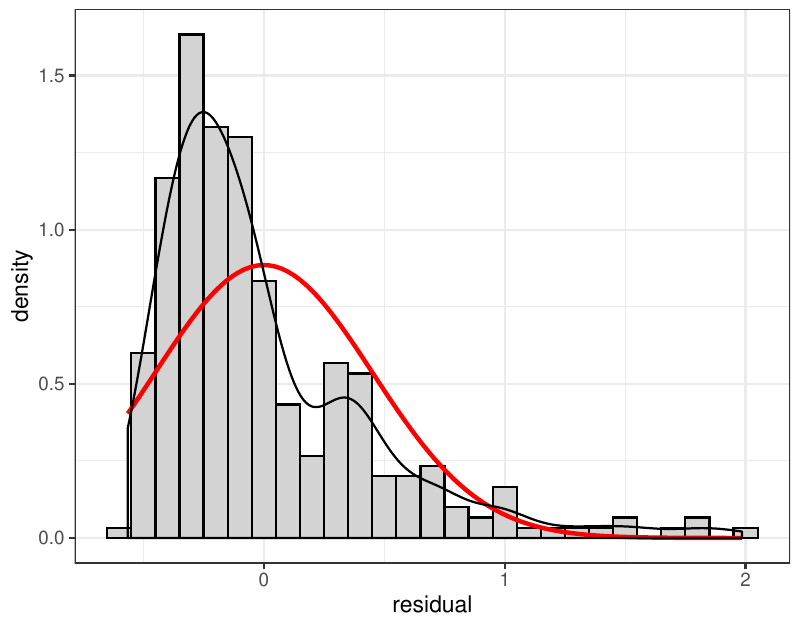}
	\end{minipage}
	\begin{minipage}[]{0.49\linewidth}
		\centering
		\includegraphics[width=2.3in]{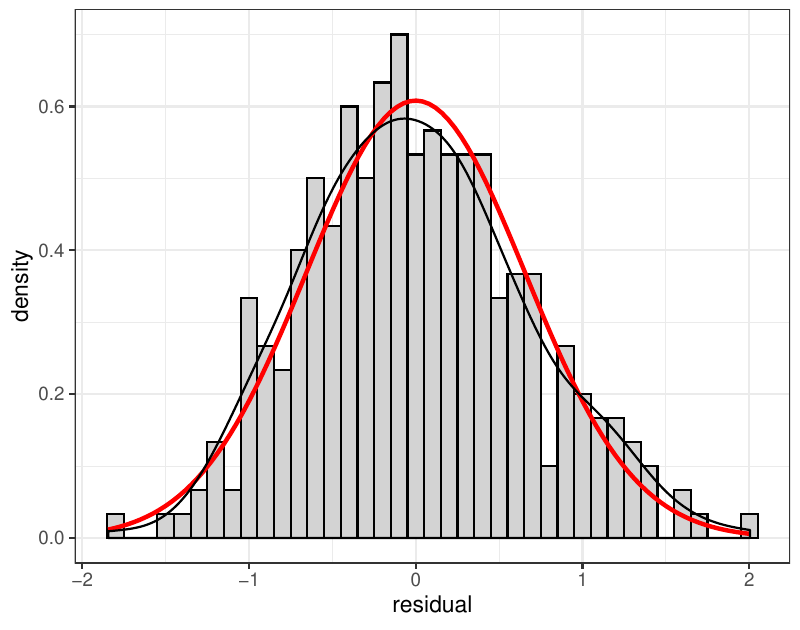}
	\end{minipage}
	\\
	\begin{minipage}[]{0.49\linewidth}
		\centering
		\includegraphics[width=2.3in]{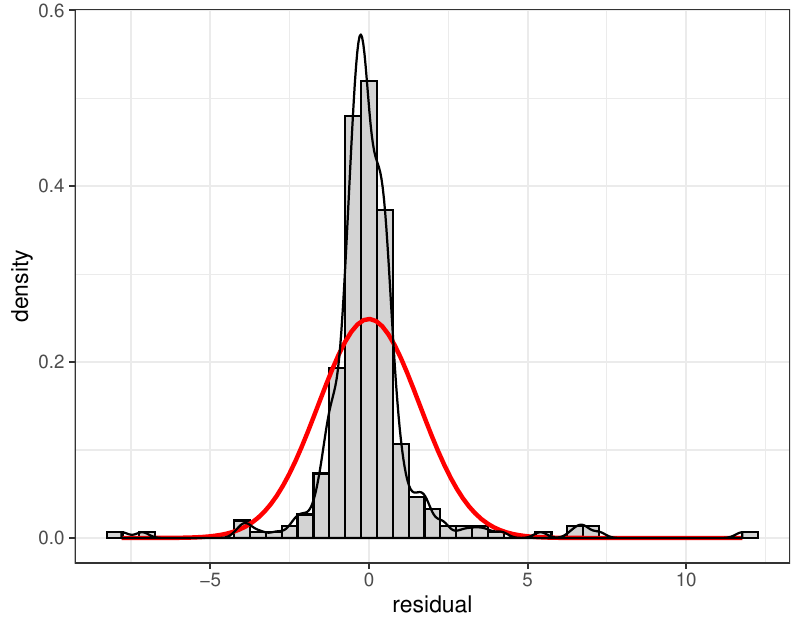}
	\end{minipage}
	\begin{minipage}[]{0.49\linewidth}
		\centering
		\includegraphics[width=2.3in]{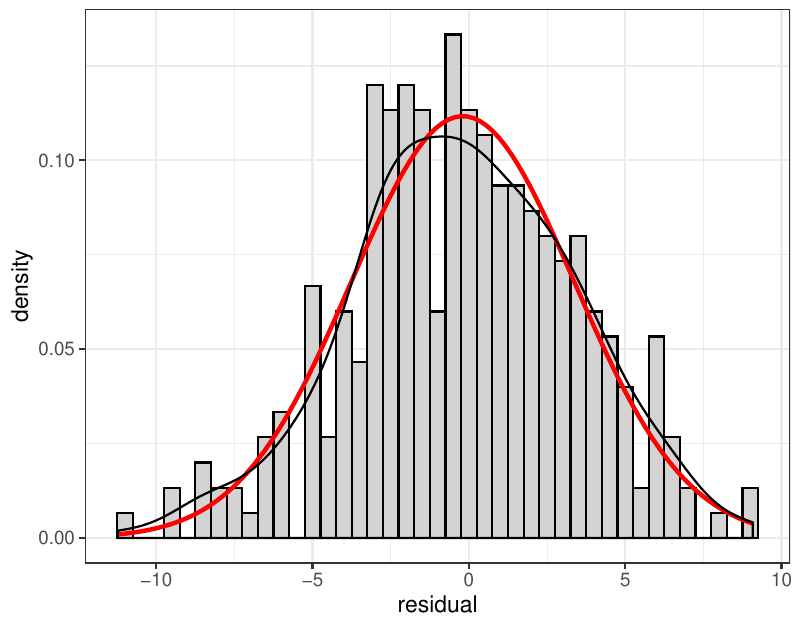}
	\end{minipage}
	\\
	\begin{minipage}[]{0.49\linewidth}
		\centering
		\includegraphics[width=2.3in]{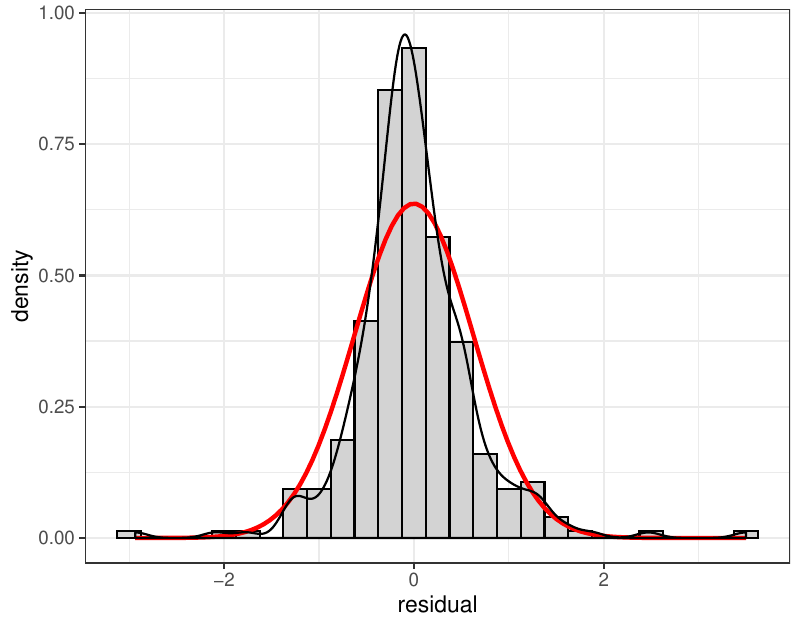}
	\end{minipage}
	\begin{minipage}[]{0.49\linewidth}
		\centering
		\includegraphics[width=2.3in]{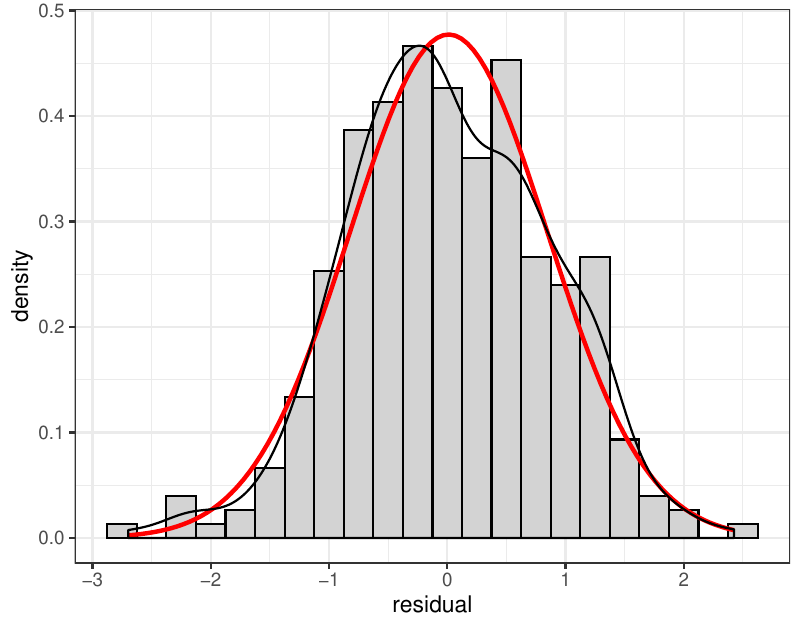}
	\end{minipage}
	\\
	\begin{minipage}[]{0.49\linewidth}
		\centering
		\includegraphics[width=2.3in]{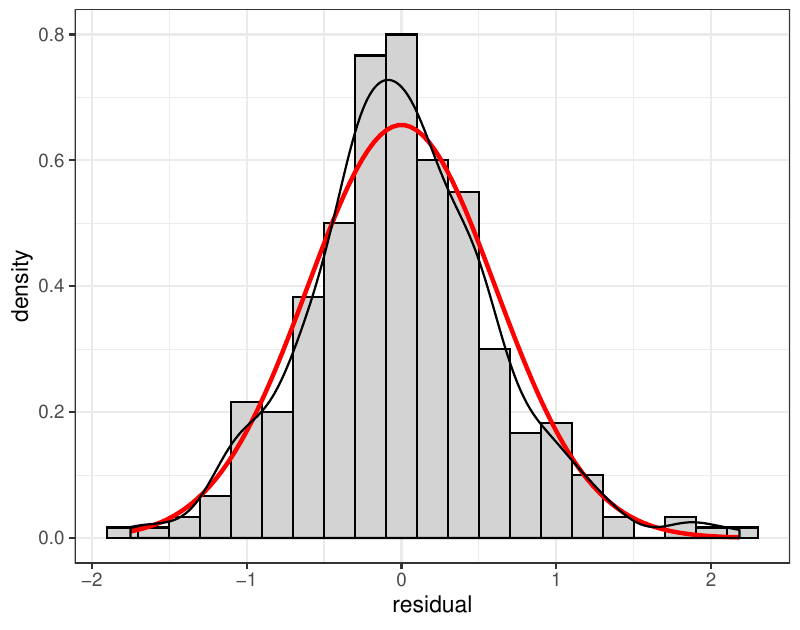}
	\end{minipage}
	\begin{minipage}[]{0.49\linewidth}
		\centering
		\includegraphics[width=2.3in]{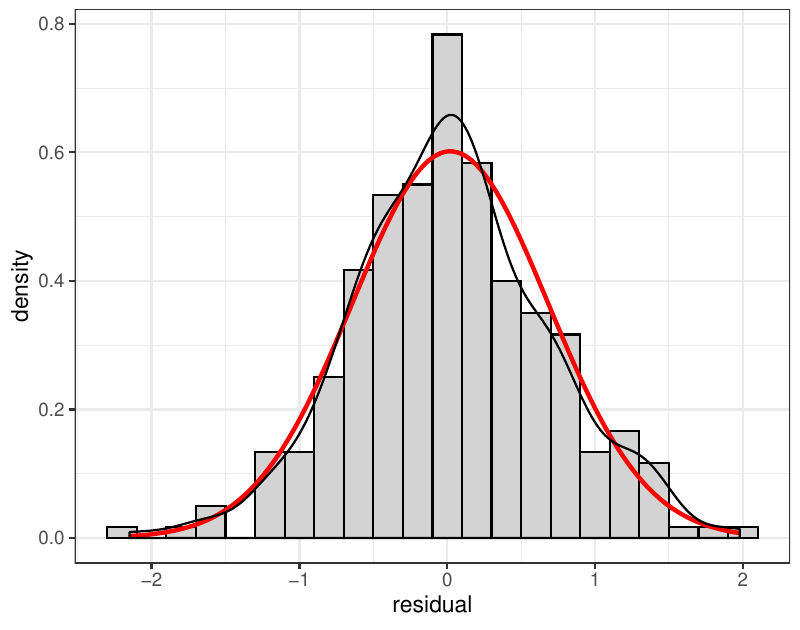}
	\end{minipage}
	\caption{Fitted residual histograms in linear models: Left panel is from OLS method; Right panel is from NSE method. From the top line to the bottom line, the residual distributions in the data generating processes are Exp(1), $t$(2), $t$(4), and $t$(8), respectively.}
	\label{fig:reg1hist}
\end{figure}

The histograms for the residuals generated by OLS and the new NSE method are demonstrated in Figure \ref{fig:reg1hist}. The black line indicates the empirical distribution of the residuals, while the red line represents the normal density curve with the same mean and variance as the residuals.

In Figure \ref{fig:reg1hist}, the residuals of OLS are shown in the left column, and the right column is for the NSE method. The first line is for errors being exponentially distributed, and the following three lines are for $t$ distribution with shape parameters (degrees of freedom $\nu$) being 2, 4, 8, respectively. We do not present the histograms of cases of truly normally distributed residuals here since, as Table \ref{tab:testols} suggests, the residuals of both methods can nearly always pass the three normality tests. Hence, their histograms would be quite similar and fit well with the normal density curve.

Our key observations from Figure \ref{fig:reg1hist} include that the fitness of the residuals of OLS is largely decided by the true distribution of $\{\varepsilon_{i}\}_{i=1}^{n}$. As we know, the t-distribution becomes closer to the normal distribution as $\nu$  increases. Hence, as $\nu$ gets larger, the empirical density curve of the histograms of the residuals gets closer to the normal density curve with the same mean and variance. Another observation is that our new NSE method always fits well with normal distribution no matter the support and how heavy the tail of the distribution is.

Additional regression simulation examples with quadratic terms and missing covariates are presented in Appendix Section \ref{sim:missing}. The missing covariates can be part of the non-parametric function $f_2(X_i)$ defined in (\ref{Eqn:semi}). Nonparametric estimation of $f_2(X_i)$ can be included in (\ref{losslinreg}). However, it is much beyond the scope of the paper and is left as a future project.
\subsection{Extreme Value Distributions and Beyond}
\label{sec:dist}
Regression is about to model the central tendency of the data, which is the most widely studied subject in statistics, machine learning, and data science. The study of extreme values and rare events is as important as central tendency and sometimes can be even more informative in many applied sciences. However, the modeling of extreme values and rare events can be challenging due to two main difficulties: 1) the available data can be sparse; 2) the existing parameter estimation can be inefficient due to that regularity conditions may not be satisfied.

This section presents two examples to illustrate how NSE can be adopted naturally to solve the estimation problems in an extreme value context. Then, in Appendix section, we present two advanced examples using NSE in block-maxima and peak-over-thresholds model fittings.
\subsubsection{Fitting the Generalized Extreme Value Distribution Directly}
\begin{exm}
	In this example, we consider estimating parameters in the generalized extreme value (GEV hereafter) distribution defined by:
	\begin{equation}
		\label{gev}
		F(y;\theta)=\begin{cases}
			\exp\left[-\left(1+\xi\frac{y-\mu}{\sigma}\right)^{-\frac{1}{\xi}}\right], & 1+\xi\frac{y-\mu}{\sigma}>0\text{ and }\xi\neq0, \\
			\exp\left(-e^{-(y-\mu)/\sigma}\right), & \xi=0. \\
		\end{cases}
	\end{equation}
\end{exm}
For the GEV case, the regularity conditions should first be explained. We refer the details of the regularity conditions in Section \ref{sim: extreme}.

In the simulation, we let $\mu=0, \sigma=1$ and the shape parameter $\xi=$-2, -1.5, -1, -0.5, 0, 0.5, 1, 1.5 respectively. We investigate how the estimate performance of $\mu, \sigma, \xi$ by MLE and the new NSE method changes as the shape parameter $\xi$ increases from $-2$ to $1.5$. The sample size $n$ is set to be 500 and the replication number is 100. In each replicate, we generate a GEV sample with specified parameters and sample size. Then estimate all the three parameters by both MLE and NSE method. The boxplot in Figure \ref{fig:gev} is used to display the results. From the figure, we can see that some boxes based on NSE are shorter than their counterpart boxes based on MLE. We know that when $\xi>-0.5$, MLEs attain the optimal asymptotic relative efficiency. The results can be explained here. Note that our new approach is a simulated sample-based indirect inference approach, i.e., the original data are augmented by simulated sequences. As a result, under the correct specified model, the information about the model and its parameters is increased, which results in smaller boxplots based on NSE.

As stated in Section \ref{sim: extreme}, when $\xi<-1$, MLE does not work well anymore, while the performance of the new NSE method is not affected by the change of $\xi$. Based on this observation, NSE estimation can be used as a benchmark estimation method for extreme value analysis.
\begin{figure}[]
	\begin{minipage}[]{0.49\linewidth}
		\centering
		\includegraphics[width=2.5in]{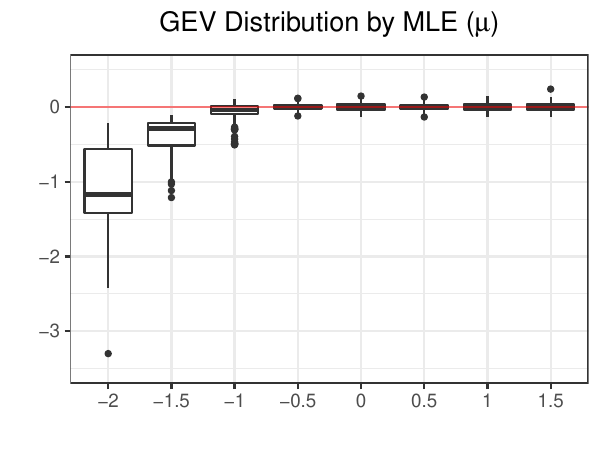}
	\end{minipage}
	\begin{minipage}[]{0.49\linewidth}
		\centering
		\includegraphics[width=2.5in]{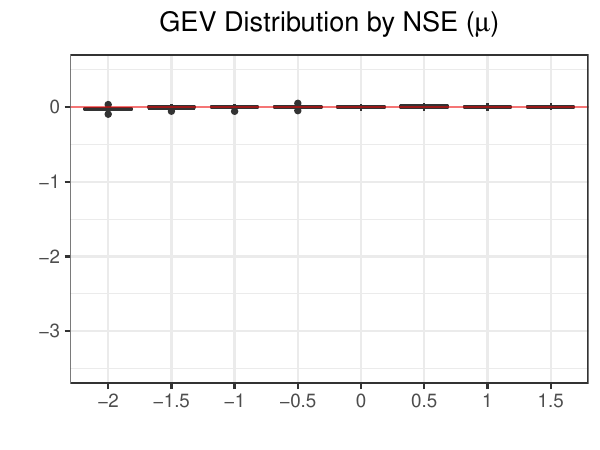}
	\end{minipage}
	\\
	\begin{minipage}[]{0.49\linewidth}
		\centering
		\includegraphics[width=2.5in]{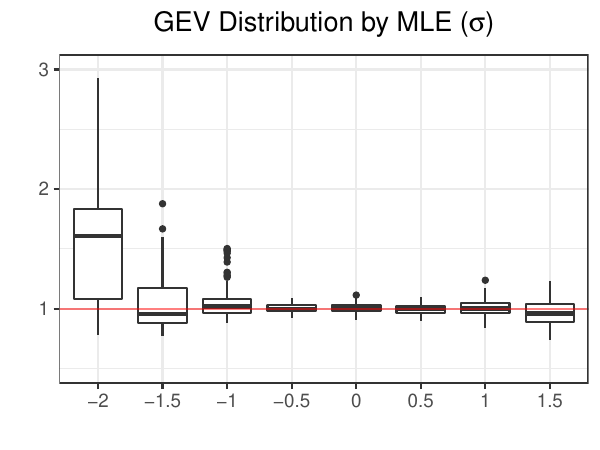}
	\end{minipage}
	\begin{minipage}[]{0.49\linewidth}
		\centering
		\includegraphics[width=2.5in]{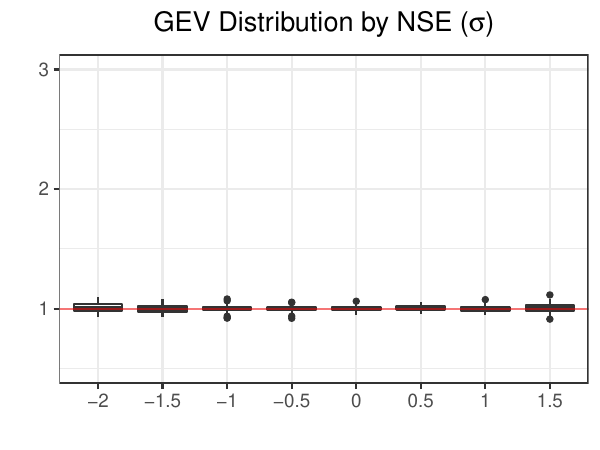}
	\end{minipage}
	\\
	\begin{minipage}[]{0.49\linewidth}
		\centering
		\includegraphics[width=2.5in]{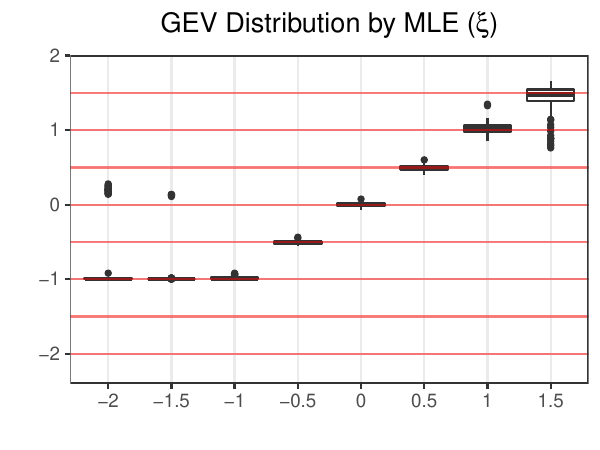}
	\end{minipage}
	\begin{minipage}[]{0.49\linewidth}
		\centering
		\includegraphics[width=2.5in]{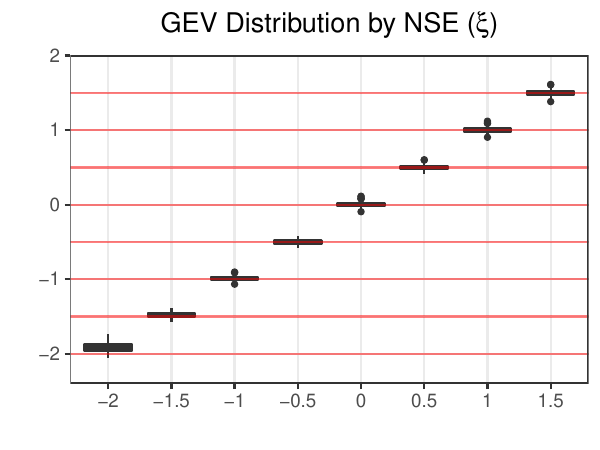}
	\end{minipage}
	\caption{GEV distribution. The left column is the estimation result obtained by MLE. The right column is the estimation result obtained by the new NSE method. The three lines are for $\mu, \sigma$ and $\xi$ respectively. True parameters are shown by colored lines. }
	\label{fig:gev}
\end{figure}

\subsubsection{Fitting positive $\alpha$-stable distribution}
\begin{exm}
	We consider a positive $\alpha$-stable random variable $Z_{\alpha}$ which is defined through its Laplace transform
	\begin{equation}
		\label{alphas}
		\EE[e^{-\lambda Z_{\alpha}}]=e^{-\lambda^{\alpha}},\ \lambda\geq0.
	\end{equation}
\end{exm}
Since the density of $Z_{\alpha}$ is not explicit except in the case $\alpha=1/2$ \citep{simon2014comparing}, MLE is not applicable to this example. For the new NSE estimation procedure, the inversion formula \citep{durrett2010probability} may be used to numerically calculate the distribution function and thus, obtain the estimates of parameters.

In the numerical study, the location parameter $\mu$ and scale parameter $\sigma$ are added, i.e., let $(Y-\mu)/\sigma$ follow the positive $\alpha$-stable distribution. Let $\mu=0, \sigma=1, \alpha=0.375, 0.5, 0.625, 0.75.$ The sample size $n$ is set to be $300$. The sample $\{Z_{\alpha1},\ldots,Z_{\alpha n}\}$ is generated according to the methods provided in \cite{simon2014comparing}. $F(Z_{\alpha},\theta)$ is calculated by the inversion formula. The replication number is set to 1000. In each replication, we generate the positive $\alpha$-stable sample with specified distribution parameters and sample size, and then estimate all the three parameters by the NSE method. We use the boxplot in Figure \ref{alphastable1} to display the estimation results obtained over a range of shape parameters.

\begin{figure}[]
	\begin{minipage}[]{0.32\linewidth}
		\centering
		\includegraphics[width=1.7in]{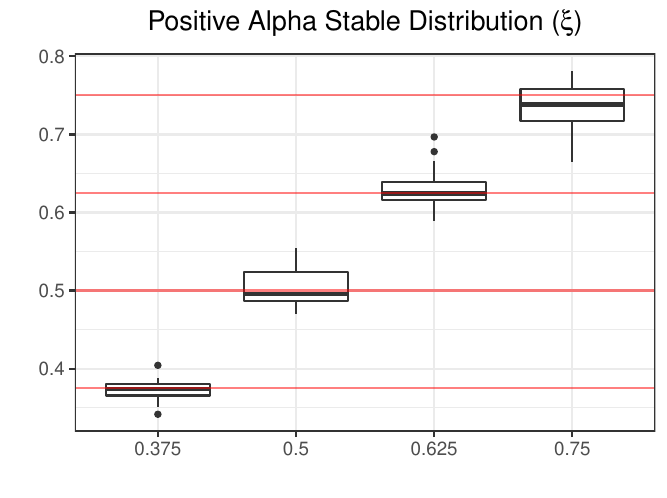}
	\end{minipage}
	\begin{minipage}[]{0.32\linewidth}
		\centering
		\includegraphics[width=1.7in]{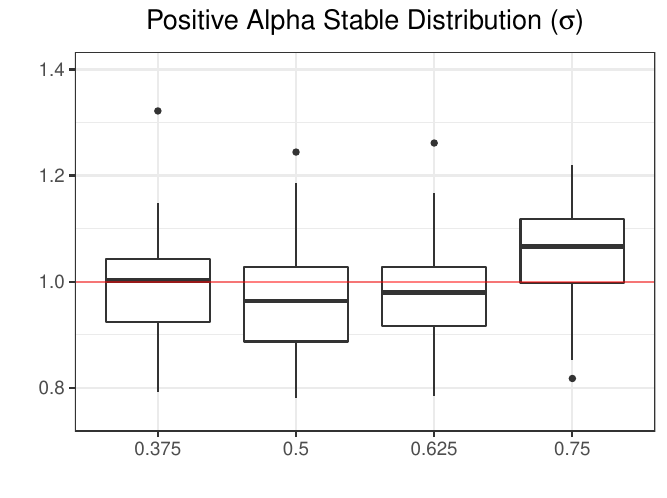}
	\end{minipage}
	\begin{minipage}[]{0.32\linewidth}
		\centering
		\includegraphics[width=1.7in]{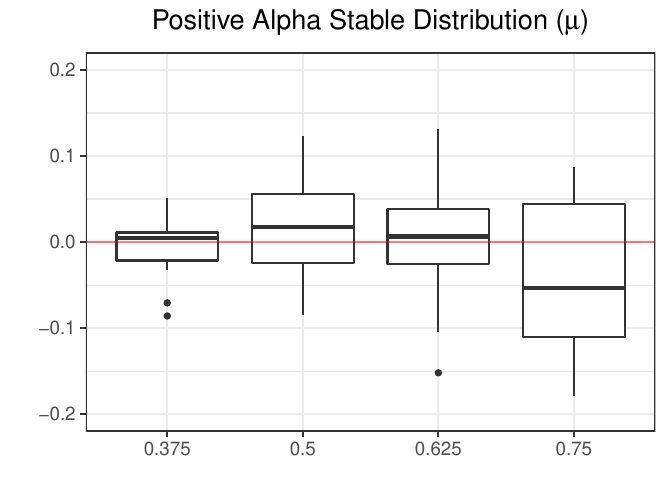}
	\end{minipage}
	\caption{The positive $\alpha$ stable distribution example: The left panel is the estimates of shape parameters $\xi$. The middle panel is the estimates of scale parameters $\sigma$. The right panel is the estimates of location parameter $\mu$. True parameters are shown by colored lines. }
	\label{alphastable1}
\end{figure}


One can see that the NSE estimation method performed excellently with the scale parameter and also performed reasonably well with the location parameter and the shape parameter.

\section{Empirical Study}\label{sec:realdata}
In this section, we present a real dataset in biological science study to illustrate the wide applicability of NSE.

Genome-wide protein–DNA interactions and transcription regulation is widely studied in biology. To pinpoint interaction sites down to the base-pair level, a computational method, Motif Discovery scan (MDscan), examines the ChIP-array-selected sequences and searches for DNA sequence motifs representing the protein-DNA interaction sites is introduced in \cite{liu2002algorithm}. To discover the influence of motifs on gene expression, we consider the following linear model:
\begin{equation}
	\label{mdscan}
	Y_g=\alpha+\sum_{m=1}^{M}\beta_mS_{mg}+\varepsilon_g,
\end{equation}
where $Y_g$ is the $\log_2$-expression value of gene $g$, $S_{mg}$ is defined to determine how well the upstream sequence of a gene $g$ matches a motif $m$ in terms of both degrees of matching and number of sites, and $\varepsilon_g$ is the gene-specific error term. The baseline expression $\alpha$ and the regression coefficient $\beta_m$ will be estimated from the data. A significantly positive or negative $\beta_m$ indicates that upstream sequences containing motif $m$ are correlated with enhanced or inhibited downstream gene expression, respectively \citep{conlon2003integrating, zczz18}.

There are 2155 motifs and 4443 genes in the dataset. Thus, variable screening is needed to be done first to reduce the number of covariates. Two screening procedures are considered here. One is iterative sure independence screening (ISIS hereafter) \citep{fan2008sure} and 36 motifs are selected. The other one is cumulatively-explained variability screening (CVS hereafter) \citep{zhang2018cvs} and 7 variables are selected. To obtain a model between the gene expression and motifs, OLS is first used, but the residuals strongly contradict the normality assumption for both sets of selected variables, which is indicated by both Lilliefors test and Jarque-Bera test. Furthermore, from the histogram of the response in Figure \ref{fig:motif}, it is obvious that the empirical distribution of the response is far away from normality. The large kurtosis (13.02) further implies the heavy tailness of the response. To alleviate the problem of the heavy tail and enhance the performance of the model, the response variable is taken a third root transformation, and the histogram of the transformed variable is displayed in Figure \ref{fig:motif}.

OLS is first used to fit the linear model between the transformed response and the selected variables, but only to find the residuals still reject the normality tests with slightly larger p-values than those obtained from untransformed responses. Then we apply the NSE method to estimate the linear model
\begin{equation*}
	Y^{1/3}_g=\alpha+\sum_{m=1}^{M}\beta_mS_{mg}+\varepsilon_g,
\end{equation*}
by using \eqref{losslinreg}, where $\varepsilon_g$ is assumed to be normal. In this way, the kurtosis (-1.45) is controlled and the interaction terms are introduced into the model without increasing the model complexity. From the histograms of residuals in Figure \ref{fig:motif}, it is apparent that the new NSE method greatly improves the normality of the residuals. For the variables selected by the two screening methods, the number of variables selected by ISIS is about 5 times that of CVS.
\begin{figure}[]
	\begin{minipage}[]{0.49\linewidth}
		\centering
		\includegraphics[width=2.5in]{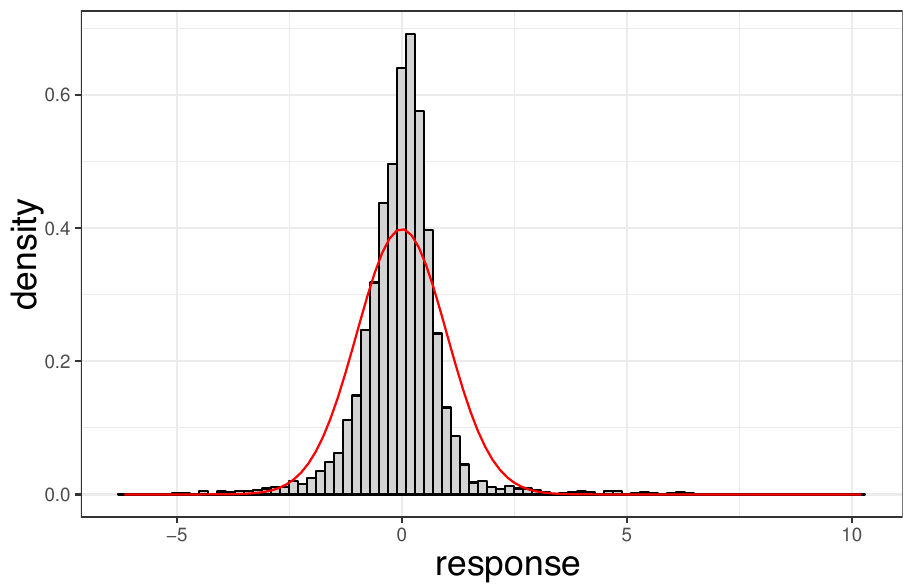}
	\end{minipage}
	\begin{minipage}[]{0.49\linewidth}
		\centering
		\includegraphics[width=2.5in]{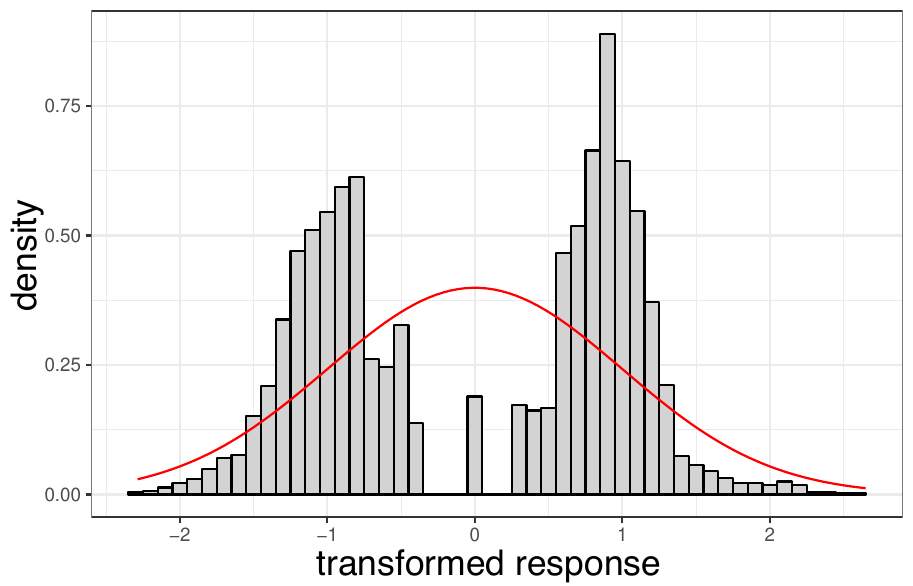}
	\end{minipage}
	\\
	\begin{minipage}[]{0.49\linewidth}
		\centering
		\includegraphics[width=2.5in]{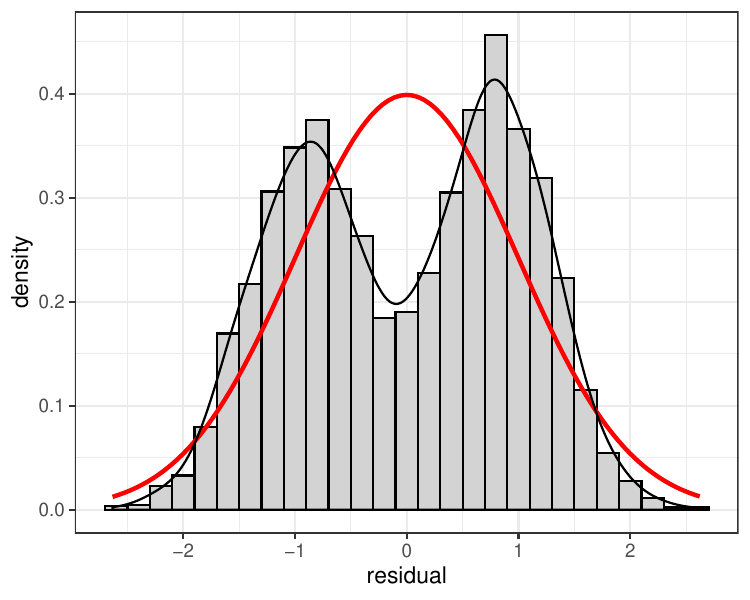}
	\end{minipage}
	\begin{minipage}[]{0.49\linewidth}
		\centering
		\includegraphics[width=2.5in]{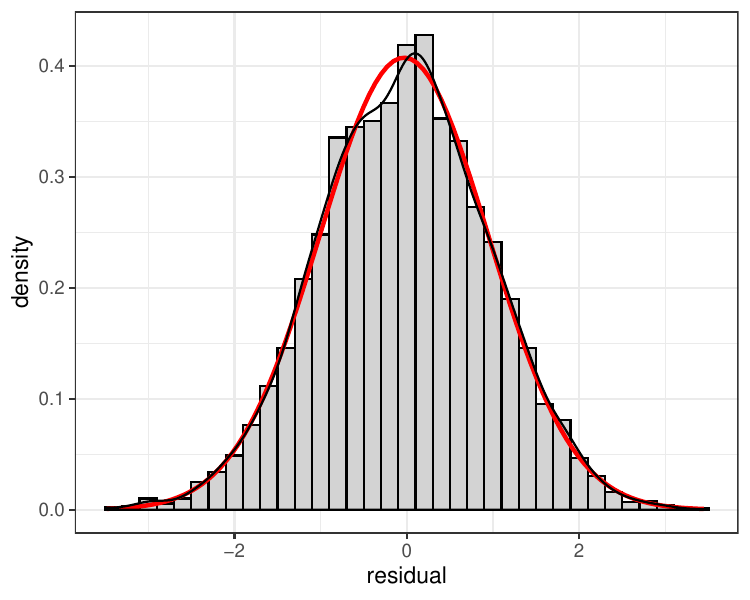}
	\end{minipage}
	\\
	\begin{minipage}[]{0.49\linewidth}
		\centering
		\includegraphics[width=2.5in]{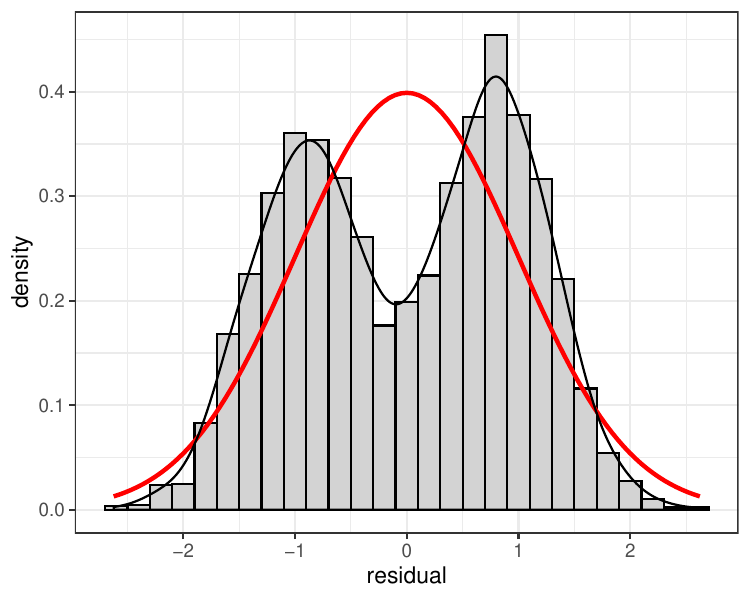}
	\end{minipage}
	\begin{minipage}[]{0.49\linewidth}
		\centering
		\includegraphics[width=2.5in]{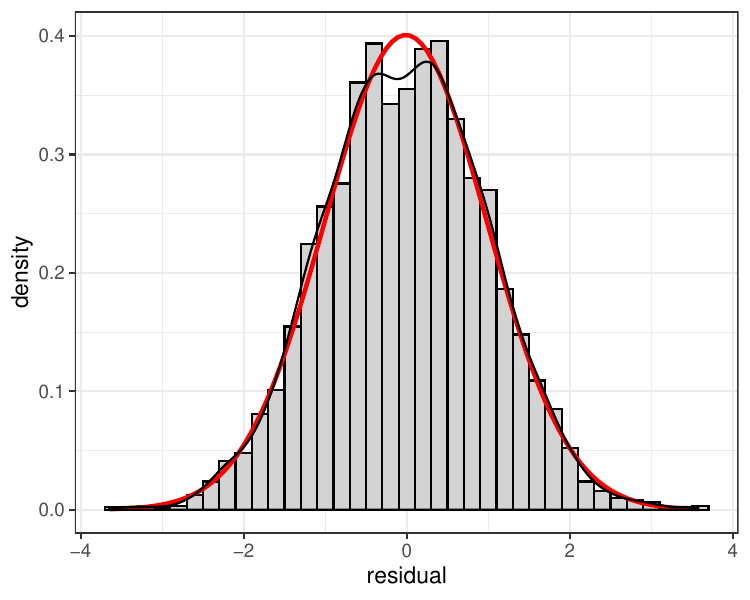}
	\end{minipage}
	\caption{Motif Data: In the first line, the left panel is the histogram of the response and the right plot is the histogram of the transformed response variable.
		In the second line, the left plot is the histogram of the residuals obtained by applying OLS on the transformed variable and the covariates selected by ISIS with p-values to be 0 for normality tests. The right plot is the histogram of the residuals obtained by applying the new NSE method on the transformed variable and the covariates selected by ISIS with p-values to be 0.66 for Lilliefors Test and 0.79 for Jarque-Bera Test.
		In the last line, the left panel is the histogram of the residuals obtained by applying OLS on the transformed variable and the covariates selected by CVS with p-values to be 0 for normality tests. The right panel is the histogram of the residuals obtained by applying the NSE method on the transformed variable and the covariates selected by CVS  with p-values to be 0.11 for Lilliefors Test and 0.07 for Jarque-Bera Test.}
	\label{fig:motif}
\end{figure}
This example dealt with high-dimension variable screening and selection, which is typical in high-dimension inference in the literature. As observed from this example, different screening methods lead to different final predictor sets, and different parameter estimations lead to residuals' normality test results with OLS/MLE having p-values close to 0 while NSE has p-values greater than 0.05. Further statistical inferences, e.g., multiple testing based on the OLS residuals, can be a problem. This data analysis further confirms that NSE can be more informative in statistical practice. 	

\section{Discussion}
\label{sec:conclusion}
In this paper, based on the concept of sample-based quotients proposed by \cite{zhang2008quotient} and developed in \cite{zhang2017random}, a new NSE family of nonlinear statistics is introduced to measure the similarity between two samples. The idea of relative error is also contained in the definition of MRQ, whose deviation from 1 is actually the largest relative error over the index set that MRQ is defined on. This definition is distinct from measures proposed in the existing literature to the best of our knowledge. Besides (\ref{Eqn: paridistance}), another possible statistic that may have similar function as MRQ is
\begin{equation}
	\label{ks}
	\max_{1\leq i\leq n}\abs{\widehat{F}(X_i)-F(X_i,\theta)}
\end{equation}
where $\{X_1,X_2,\ldots,X_n\}$ is a random sample which is assumed to be generated from the distribution $F$ with parameter $\theta$ and $\widehat{F}$ is the empirical distribution. According to Glivenko-Cantelli theorem \citep{durrett2010probability}, \eqref{ks} converges to 0 almost surely as $n\to\infty$. Thus, a possible estimation procedure is to minimize \eqref{ks} over $\theta$. This method can be graphically understood through the probability plot. This is another popular plot in practice. It is actually a probability versus probability plot, or a P-P plot. If the two CDFs are identical, the theoretical P-P plot will be the main diagonal, the 45-degree-line through the origin. The P-P plot is more sensitive to the differences in the middle part of the two distributions (the empirical distribution and the hypothesized distribution), whereas the Q-Q plot is more sensitive to the differences in the tails of the two distributions. A disadvantage of P-P plots is that they mainly discriminate in regions of high probability density. In these regions, the empirical and theoretical cumulative distributions change more rapidly than in regions of low probability density. This property may greatly deteriorate the performance of applying \eqref{ks} to estimation. Some simulation examples also support this observation. In addition, estimators derived from (\ref{Eqn: paridistance}), (\ref{ks}), and their equivalent forms may obey slower convergence rates in distribution, and much more advanced statistical theory and conditions may be required. Most importantly, these methods cannot guarantee the inverse distribution transformed random variates using estimated parameter values to obey the assumed distribution. For example, in the literature, Theorem 4 in \cite{zhang2017random} and Lemma 2 in its supplementary guarantee the inverse distribution transformation based on estimated parameter values to work. On the other hand, the new NSE estimation guarantees the procedure to work and is easy to implement and with faster convergence rates.

A core step in applying the new NSE methodology is to transform the sample to unit exponential scale as Theorem \ref{thm:mid} states. The NSE methodology allows us to fit the sample data to any pre-specified distribution family. The NSE methodology can explain why under some undesirable situations, such as missing data and nonlinear models, the NSE method can adjust the coefficients obtained by OLS to satisfy the normality assumption of residuals. If the normality tests reveal evidence of nonnormality, the standard remedy is to transform the response. Box and Cox introduced a popular family of transformations to approach several types of regression problems so that all the regression assumptions (means linear in the explanatory variables, homogeneous variances, and normal errors) would be satisfied. However, as illustrated in Section \ref{sec:realdata}, the residuals do not pass the normality tests even after transforming the response variable in biological data. Thus, the combination of the Box-Cox transformation and the NSE method is of practical benefit.

In the application and optimization, another noteworthy point is the choice of the initial value. Since the objective function \eqref{loss} is not convex, it is hard to obtain the optimal solution limited to the optimization algorithm. Usually, the optimization result is the local minima of \eqref{loss}. However, with the help of normality tests, we are able to select and present the typical result in the figures. In a future project, we will explore the method of POTS \citep{GSZ2018}  in searching global minimum points.

It should be noted that there are two gaps in the whole sequence that our derived theorems do not cover since the size of the index set is fixed in Theorem \ref{thm:left} and Theorem \ref{thm:right}. They only make up a small proportion which may approach 0 as the sample size increases. By observing Figure \ref{mqr}, the behaviors of these two parts are consistent with their neighboring parts. The theoretical property of these two parts may be our future research topic.

In summary, this paper presents a new NSE statistical inference approach: simulated distribution-based learning for statistical inferences. The method can be widely applicable in any problem concerned with the estimation of distribution parameters, especially the widely-used linear model. One potential application of NSE is in Bayesian analysis, in which various distribution assumptions are made. In real data, the linear model is usually not exactly the same as the true underlying model because of either the incorrectly specified function or the unobserved covariates. As we can see in Sections \ref{sec:simulation} and \ref{sec:realdata} and Appendix A, when the true model is not far away from the linear model, the NSE method absorbs part of the influence of the missing covariates and the incorrectly specified function into the linear part by preserving the normality assumption on the residuals, leading to competitive performance compared with MLE. Otherwise, when the true model deviates far away from the linear model, using the linear model with normality assumption on the residuals to depict the dataset will be inappropriate and may lead to an unreliable further inference. Therefore, the role of the distribution assumption is attached with proper importance in simulated distribution-based learning while ignored in the commonly-used MLE and other estimation methods.

{\bf Acknowledgments} The authors thank Bingyan Wang for her unselfish assistance and work in this project.
\bibliographystyle{apalike} 
\bibliography{main}       


\newpage
\begin{center}
  {{\huge\bf Supplementary file:}\\ {\LARGE\bf Necessary and Sufficient Estimation for Continuous\\ Distribution Preserved Inferences}}
  ~~\\
  ~~\\
  {\large Zhengjun Zhang$^{a,b,c}$, Xinyang Hu$^d$, Chuyang Lu$^e$, and Tianying Liu$^a$ }
\\
  ~~\\
  ~~\\

$^a$School of Economics and Management, University of Chinese Academy of Sciences\\
$^b$AMSS Center for Forecasting Science, Chinese Academy of Sciences\\
$^c$Department of Statistics, University of Wisconsin, Madison\\
$^d$Department of Statistics, Yale University\\
$^e$Academy of Mathematics and System Sciences, Chinese Academy of Sciences
\end{center}
~~\\
\begin{appendix}
\section{Appendix}
\subsection{Additional Simulation Studies}
\subsubsection{Regressions with nonlinear and missing predictors}\label{sim:missing}
We consider the incorrectly specified model.

\begin{exm}We consider the true model
	\begin{equation}
		\label{linreg1}
		y_i=x_i^T\beta_1+\mu+\beta_2(\sum_{j=1}^{q}z_{ij})^2+\varepsilon_i, i=1,\ldots,n,
	\end{equation}
	where $\mu\in\R, x_i\in\R^p, \beta_1\in\R^p, z_i\in\R^q, \beta_2\in\R^q$, and $\varepsilon_i$'s are independently identically generated from normal distribution $N(0,\sigma^2)$. Since the true model is a quadratic model, the linear model used in estimation is incorrectly specified. Furthermore, $x_i, i=1,\ldots,n$, are observable, while $z_i, i=1,\ldots,n$ are missing data.
\end{exm}
In the simulation, let $p=5,q=2$ and $\beta_1=(0.3,0.5,0.5,0.3,0.3)^T$ and $\beta_2=0.1$. The sample size is set to be 500 and the replication number is 500. In each replication, $\varepsilon_i$ are generated independently identically from normal distribution $N(0,0.3^2)$; $x_{i3},x_{i4}$ and $x_{i5}$ are generated independently identically from the standard exponential  distribution respectively; $(x_{i1},z_{i1})$ and $(x_{i2},z_{i2})$ are generated from standard bivariate normal distribution with correlation $\rho=0.6$ and $\rho=0.8$ respectively. OLS and the new method are applied to estimate the parameters in $\beta_1$. The boxplots for estimation results from all the replications are shown in Figure \ref{fig:2}.

\begin{figure}[]
	\begin{minipage}[]{0.4\linewidth}
		\centering
		\includegraphics[width=2.5in]{figures/reg7newq1.pdf}
	\end{minipage}
	\begin{minipage}[]{0.59\linewidth}
		\centering
		\includegraphics[width=2.88in]{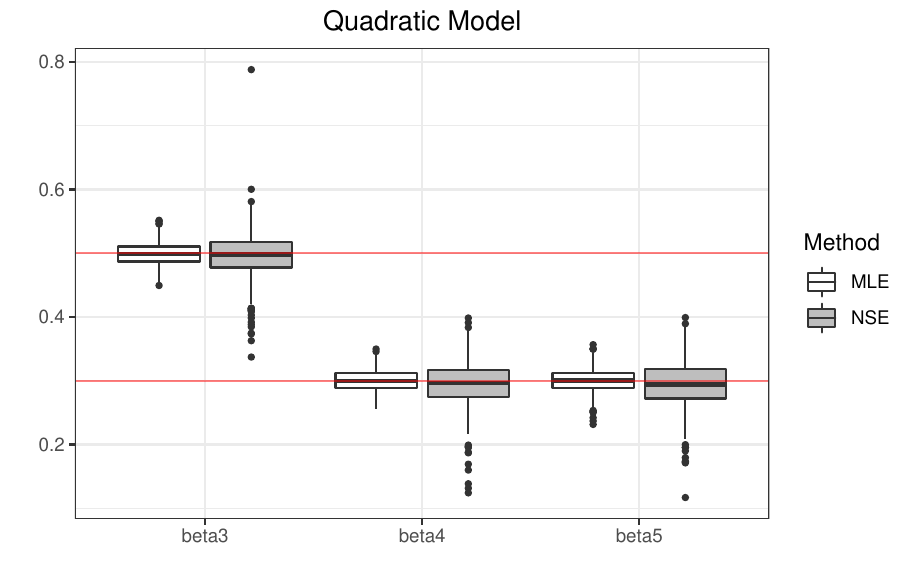}
	\end{minipage}
	\caption{Quadratic model (\ref{linreg1}): the left panel show the estimates of OLS and the right is for the new method. True parameters are shown by red lines. The parameter names corresponding to the box plots are marked on the horizontal axis.}
	\label{fig:2}
\end{figure}

\begin{figure}[]
	\begin{minipage}[]{0.49\linewidth}
		\centering
		\includegraphics[width=2.5in]{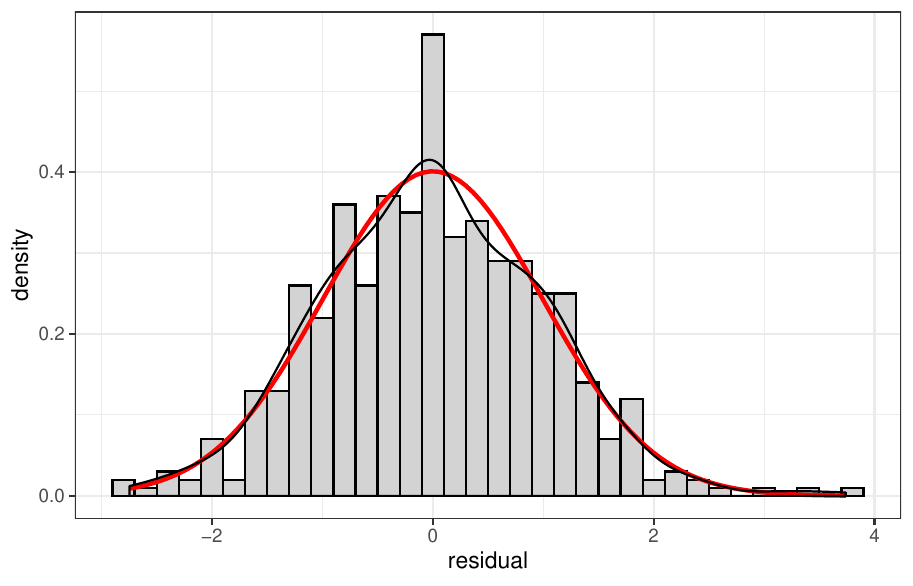}
	\end{minipage}
	\begin{minipage}[]{0.49\linewidth}
		\centering
		\includegraphics[width=2.5in]{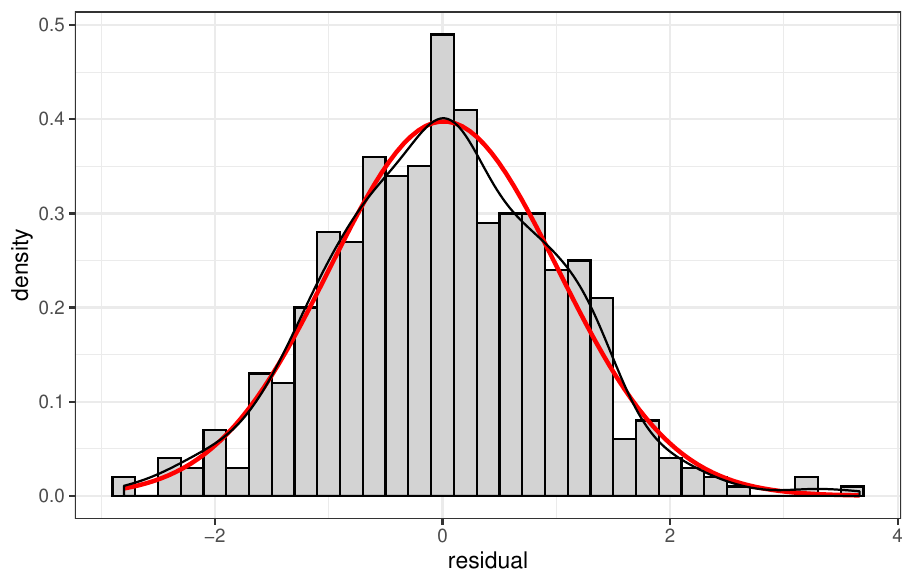}
	\end{minipage}
	\caption{Quadratic Model (\ref{linreg1}): The left panel is the histogram of the residuals generated by OLS with p-values to be 0.03 for Lilliefors Test and 0.20 for Jarque-Bera Test. The right panel is the histogram of the residuals generated by the new method with p-values to be 0.92 for Lilliefors Test and 0.37 for Jarque-Bera Test. The black line is the fitted density function. The red line is the density function of the normal distribution with the same mean and variance as the residuals. }
	\label{fig:quares}
\end{figure}

The histograms for the residuals generated by OLS and the new NSE method are demonstrated in Figure \ref{fig:quares}. The black line indicates the fitted density function, and the red line is the density function of the normal distribution with the same mean and variance as the residuals. The OLS residual normality hypothesis is rejected by Lilliefors Test with a p-value being 0.03, while the NSE residual normality hypothesis is not rejected by both tests. Figure \ref{fig:quares} directly shows the new method makes the empirical distribution of the residuals closer to the normal distribution.

The next example is about the missing covariates.

\begin{exm} We consider a linear model with missing covariates
	\begin{equation}
		\label{linreg2}
		y_i=x_i^T\beta_1+\mu-z_i\beta_2+\varepsilon_i, i=1,\ldots,n,
	\end{equation}
	where $\mu\in\R, x_i\in\R^p, \beta_1\in\R^p, z_i\in\R^q, \beta_2\in\R^q$ and $\varepsilon_i$'s are independently identically generated from normal distribution $N(0,\sigma^2)$. Furthermore, it is assumed that $x_i, i=1,\ldots,n$ are observable, while $z_i, i=1,\ldots,n$ are missing data.
\end{exm}
In the simulation, we let $p=5, \sigma=0.3$ and $q=1$. Let $\beta_1=(0.3,0.5,0.5,0.3,0.3)^T$ and $\beta_2=0.1$. The sample size is set to be 500 and the replication time is 500. Denote by $x_i=(x_{i1},\ldots,x_{i5})^T$. In each replication, $x_{i2},x_{i3}$, $x_{i4}$ and $x_{i5}$ are generated independently identically from standard exponential distribution $Exp(1)$ respectively. $(x_{i1},z_{i1})$ are generated from standard bivariate normal distribution with correlation $\rho_1=0.6$. OLS and the new method are applied to estimate the parameters $\beta_1$.


In Figure \ref{fig:miss}, the estimates of OLS and the new method are quite similar, except that the box of the new method is larger. Due to the strong correlation between the observable and missing data, the estimation of the coefficients associated to $x_{i1}$ are away from the true value.

The histograms in Figure \ref{fig:missres} illustrate how the new method adjusts the coefficients to make the residuals satisfy the normality assumption.

\begin{figure}[]
	\begin{minipage}[]{0.4\linewidth}
		\centering
		\includegraphics[width=2.45in]{figures/reg8miss1.pdf}
	\end{minipage}
	\begin{minipage}[]{0.59\linewidth}
		\centering
		\includegraphics[width=2.99in]{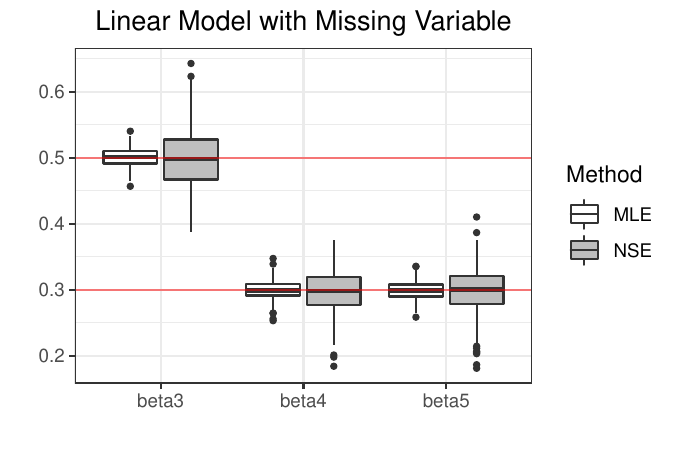}
	\end{minipage}
	\caption{Linear model (\ref{linreg2}) with missing variable: the left panel shows the estimates of two methods and the right is for the new method. Red lines indicate the true parameters. The parameter names corresponding to the box plots are marked on the horizontal axis.}
	\label{fig:miss}
\end{figure}

\begin{figure}[]
	\begin{minipage}[]{0.49\linewidth}
		\centering
		\includegraphics[width=2.5in]{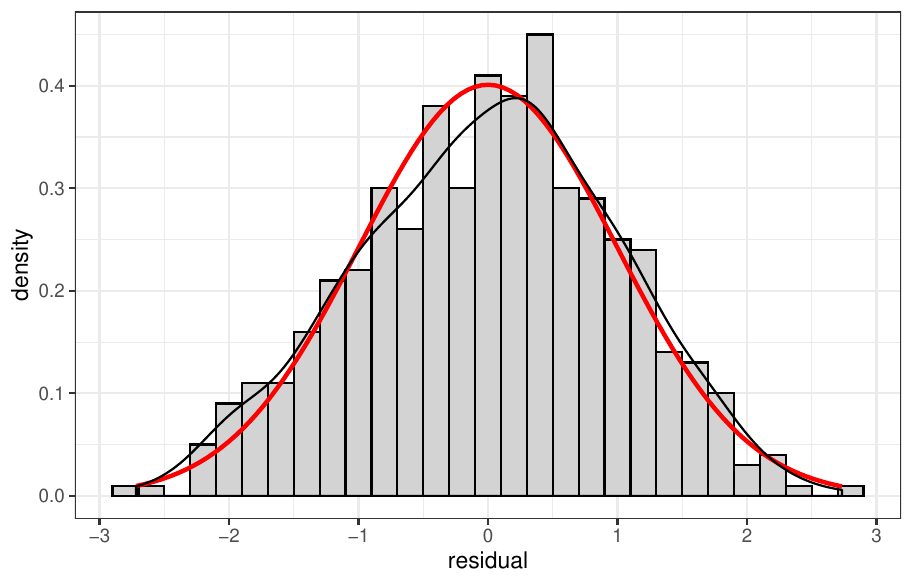}
	\end{minipage}
	\begin{minipage}[]{0.49\linewidth}
		\centering
		\includegraphics[width=2.5in]{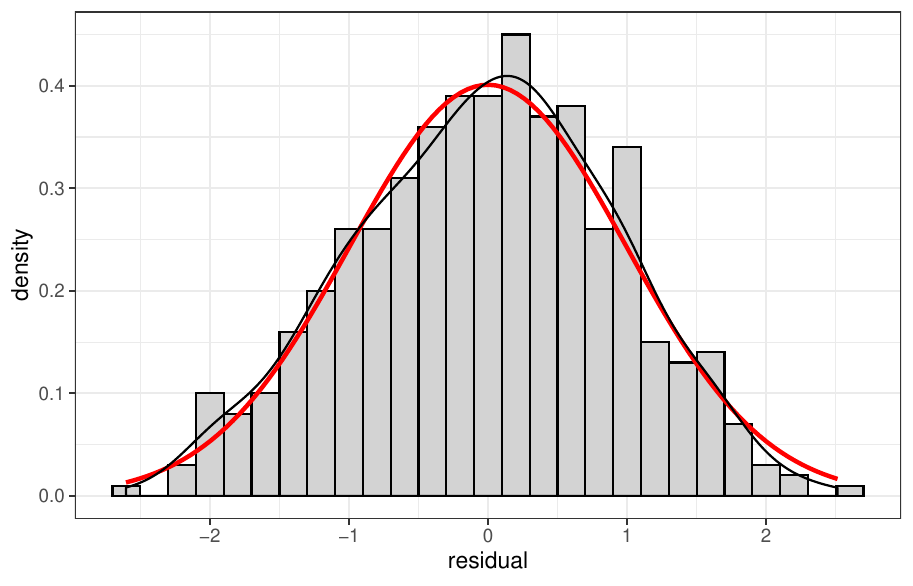}
	\end{minipage}
	\caption{Linear model (\ref{linreg2}) with missing variable: The left panel is the histogram of the residuals generated by using OLS with p-values to be 0.21 for Lilliefors Test and 0.11 for Jarque-Bera Test. The right panel is the histogram of the residuals generated by the new method  with p-values to be 0.48 for Lilliefors Test and 0.13 for Jarque-Bera Test. The black line is the fitted density function. The red line is the density function of the normal distribution with the same mean and variance as the residuals. }
	\label{fig:missres}
\end{figure}
\subsubsection{Extreme value distribution: advanced fitting}\label{sim: extreme}
We start to give a brief review of regularity conditions in extreme value inference. Let $\theta$ be a vector of $p$ parameters and $\theta^*$ be the unknown true value. Then the theory is said to be regular if, in an open neighborhood of $\theta^*$, the log-likelihood function is closely approximated, in probability, by a concave quadratic function whose maximum point converges in some efficient sense to the true parameter value as the sample size increases. Conditions ensuring the convergence are called regular conditions. One possible failure of these conditions involves model behavior at a boundary. It occurs when $\theta^*$ is an interior point of the feasible region, but the observations impose direct restrictions on the value that $\theta$ can take \citep{smith1985maximum}.

As an example, the three types of extreme value distribution must form a smooth family since they arise as limiting distributions for the normalized maxima of independent and identically distributed samples from a continuum of possible original distributions. This family can be parameterized by the GEV distribution. The cases $\xi<0$, $\xi\to0$, and $\xi>0$ give the Weibull, Gumbel, and Fr\'echet distributions respectively, e.g., see \citep{David2003Order}. Furthermore, it is notable that if the support of a probability density depends on an unknown parameter, then the classical regularity conditions for maximum likelihood estimation are not satisfied. In some cases, the maximum likelihood estimators exist and have the same asymptotic properties as in regular cases. In other cases, they may exist but not be asymptotically efficient or normally distributed, and in still other cases, they may not exist at all, at least not as solutions of the likelihood equations \citep{smith1985maximum}.

Obviously, the GEV model does not satisfy the regularity conditions because the end-points of the GEV distribution are functions of the parameter values: $\mu-\sigma/\xi$ is an upper end-point of the distribution when $\xi<0$, and a lower end-point when $\xi>0$. Thus, the violation of the usual regularity conditions means that the standard asymptotic likelihood results are not automatically applicable \citep{coles2001introduction}.

\cite{smith1985maximum} studied this problem in detail and obtained the results: when $\xi>-0.5$, maximum likelihood estimators are regular, in the sense of having the usual asymptotic properties; when $-1 < \xi < -0.5$, maximum likelihood estimators are generally obtainable, but do not have the standard asymptotic properties; when $\xi<-1$, maximum likelihood estimators are unlikely to be obtainable. On the other hand, the new NSE method is free of these constraints.

\subsubsection{Fitting GEV Using Block Maxima}
To model the extremes of a series of independent observations $X_1,X_2,\ldots$, data are blocked into sequences of observations of length $m$ for some large value of $m$, which generates a sequence of block maxima, $M_{n,1}, \ldots, M_{n,k}$, say, to which the GEV distribution can be fitted \citep{coles2001introduction}. This application blocks the data into blocks of equal length and fits the GEV to the set of block maxima. But in implementing this model for any particular dataset, the choice of block size can be critical. The choice amounts to a trade-off between bias and variance: blocks that are too small mean that approximation by the limit model is likely to be poor, leading to bias in estimation and extrapolation; large blocks generate few block maxima, leading to large estimation variance. Pragmatic considerations often lead to the adoption of blocks of length one year, in which case $m$ is the number of observations in a year and the block maxima are annual maxima.

In the simulation, we generate the sequence $\{X_1,X_2,\ldots,X_n\}$ from standard normal distribution $N(0,1)$, unit exponential distribution $Exp(1)$, unit Fr\'echet distribution, and standard uniform distribution $U(0,1)$, respectively. We then obtain the sequence of block maxima $\{M_{n,i}, i=1,2,\ldots,k\}$ by
$$M_{n,i}=\max_{(i-1)m<j\leq im}X_j.$$
Hence the $n=m\times k$ observations are divided into $k$ blocks of size $m$.
Let $m=100$ and $k=500$. The estimation procedure is performed 100 times by implementing MLE and the new NSE method. The performance is illustrated in Figure \ref{block}.
\begin{figure}[]
	\begin{minipage}[]{0.49\linewidth}
		\centering
		\includegraphics[width=2.5in]{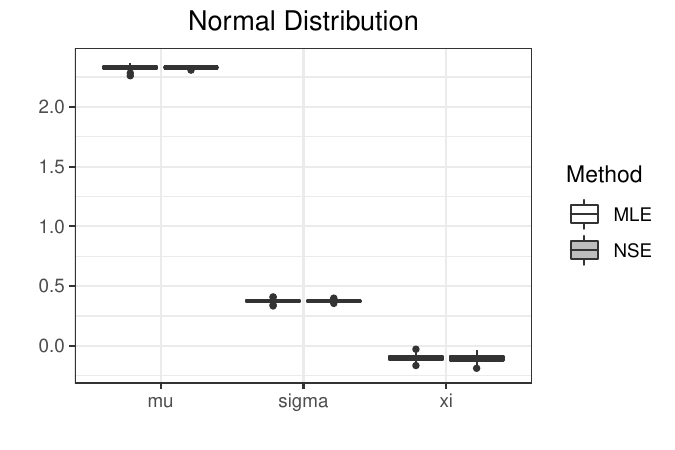}
	\end{minipage}
	\begin{minipage}[]{0.49\linewidth}
		\centering
		\includegraphics[width=2.5in]{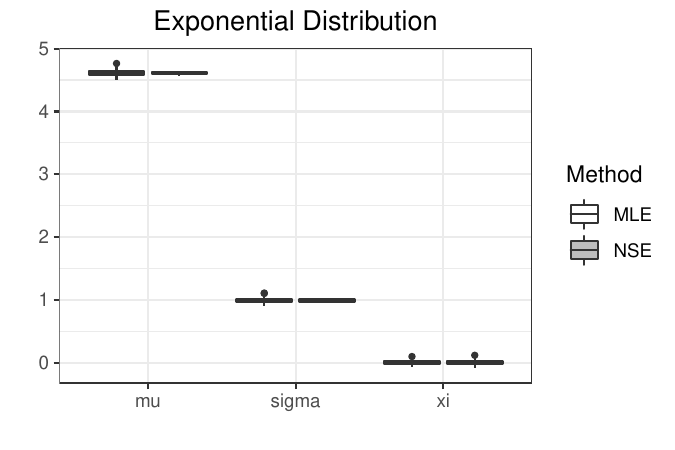}
	\end{minipage}
	\\
	\begin{minipage}[]{0.49\linewidth}
		\centering
		\includegraphics[width=2.5in]{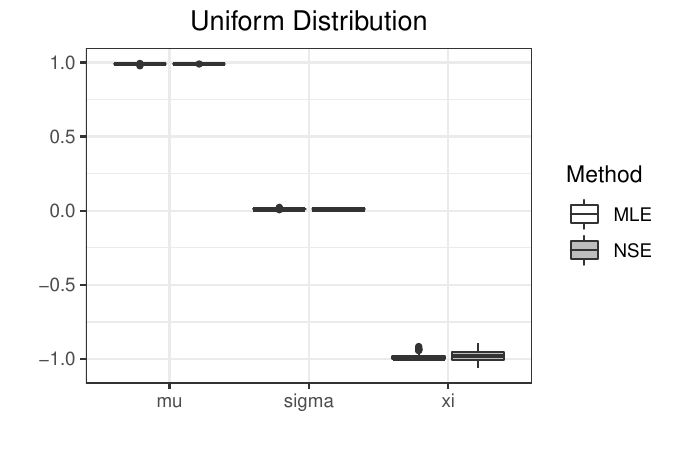}
	\end{minipage}
	\begin{minipage}[]{0.49\linewidth}
		\centering
		\includegraphics[width=2.5in]{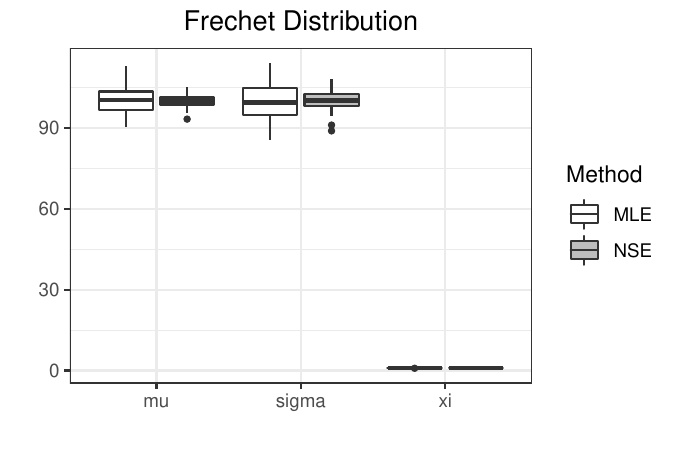}
	\end{minipage}
	\caption{Fitting GEV using block maxima. On the first line, the left plot is of normal distribution, and the right plot is of exponential distribution. On the second line, the left plot is of uniform distribution, and the right plot is of Fr\'echet distribution.}
	\label{block}
\end{figure}

From Figure \ref{block}, one can see that the new NSE method is robust and comparable to MLE.
\subsubsection{Peak Over Threshold Method}
Let $X_1,X_2,\ldots$ be a sequence of independent and identically distributed random variables, having marginal distribution function $F$. It is natural to regard extreme events as those of the $X_i$ that exceed some high threshold $u$. Denoting an arbitrary term in the $X_i$ sequence by $X$. Then the random variable $X-u$ represents the value of exceedances over threshold $u$, provided that this threshold has been exceeded. The distribution function of this random variable is called excess distribution function (denoted by $F_u$), and it follows that a description of the stochastic behavior of extreme events is given by the conditional probability
\begin{equation}
	\label{pot}
	F_u(x)=P(X-u\leq x|X>u)=\frac{F(x+u)-F(u)}{1-F(u)}.
\end{equation}
If the parent distribution $F$ were known, the distribution of threshold exceedances in \eqref{pot} would also be known. Since, in practical applications, this is not the case, approximations that are broadly applicable for high values of the threshold are sought. This parallels the use of the GEV as an approximation to the distribution of maxima of long sequences when the parent population is unknown.

To be specific, let $M_n=\max\{X_1,\ldots,X_n\}$. If for large $n$,
$$\PP\{M_n\leq y\}\approx G(y)$$
holds, where $G(y)$ is GEV distribution defined by \eqref{gev}, then for large enough $u$, the distribution function of $X-u$, conditional on $X>u$, is approximately
\begin{equation}
	\label{gpd}
	H(y)=1-\left(1+\frac{\xi y}{\widetilde{\sigma}}\right)^{-1/\xi}
\end{equation}
defined on $\{y:y>0\text{ and }(1+\xi y/\widetilde{\sigma})>0\}$ where $\widetilde{\sigma}=\sigma+\xi(u-\mu)$ \citep{coles2001introduction}. \eqref{gpd} is called the generalized Pareto distribution (GPD hereafter) family. GPD can be used for the modeling of tails of distributions, i.e., for data exceeding certain threshold (peaks over threshold) as the threshold tends to infinity (\cite{pickands1975statistical}; \cite{davison1984modelling}; \cite{smith1984threshold}). \cite{smith1985maximum} gives analogous results to those for the generalized extreme value distribution: for $\xi>-0.5$ the information matrix is finite and the classical asymptotic theory of maximum likelihood estimators is applicable, while for $\xi\leq-0.5$ the problem is non-regular and special procedures are needed.

It is worth noting that choosing a different, but still large, block size $m$ would affect the values of the GEV parameters, but not those of the corresponding generalized Pareto distribution of threshold excesses: $\xi$ is invariant to block size, while the calculation of $\widetilde{\sigma}$ is unperturbed by the changes in $\mu$ and $\sigma$ which are self-compensating.

In the experiment, set $n=10,000$ and the threshold is set to $95$th percentile of the sample. The original sequence $\{X_1,X_2,\ldots,X_n\}$ is generated from unit normal distribution $N(0,1)$, unit exponential distribution $Exp(1)$, standard Fr\'echet distribution, and standard uniform distribution $U(0,1)$ respectively, i.e., the same as those in block maxima method. The simulation is performed 100 times. Both MLE \citep{hosking1987parameter} and the new NSE method are implemented to estimate $\xi$ and $\widetilde{\sigma}$ in \eqref{gpd}. In Figure \ref{POT}, we display the estimates of $\xi$ and $\widetilde{\sigma}$ to compare the performance.
\begin{figure}[]
	\begin{minipage}[]{0.49\linewidth}
		\centering
		\includegraphics[width=2.6in]{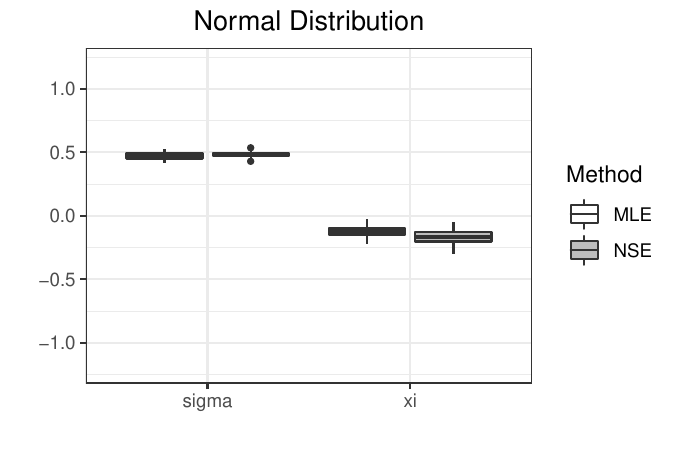}
	\end{minipage}
	\begin{minipage}[]{0.49\linewidth}
		\centering
		\includegraphics[width=2.6in]{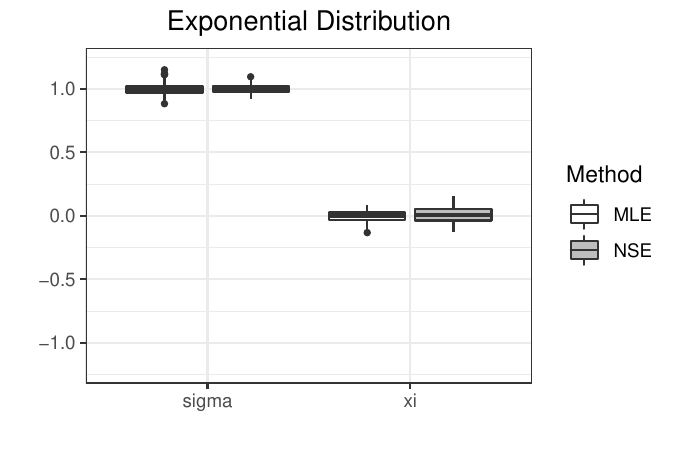}
	\end{minipage}
	\\
	\begin{minipage}[]{0.49\linewidth}
		\centering
		\includegraphics[width=2.6in]{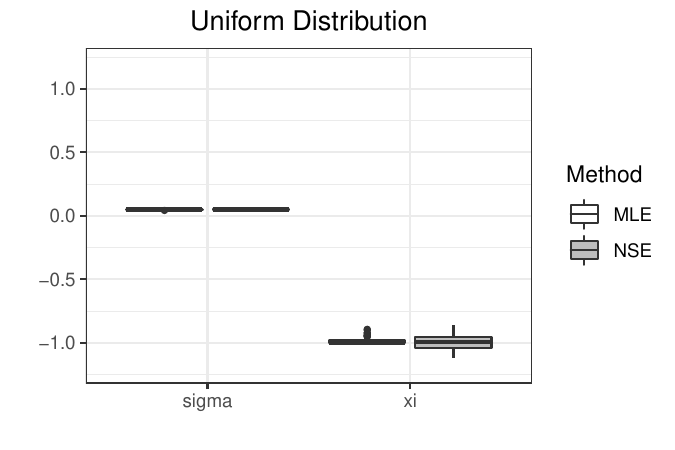}
	\end{minipage}
	\begin{minipage}[]{0.49\linewidth}
		\centering
		\includegraphics[width=2.6in]{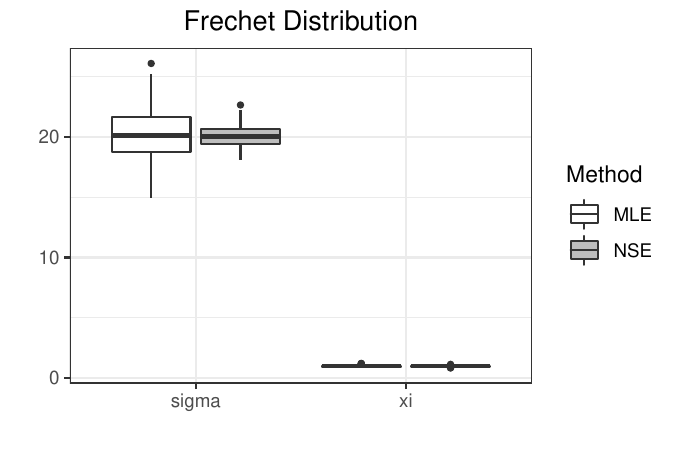}
	\end{minipage}
	\caption{Peaks Over Threshold. On the first line, the left plot is of normal distribution and the right plot is of exponential distribution. On the second line, the left plot is of uniform distribution and the right plot is of Fr\'echet distribution.}\label{POT}
\end{figure}

From Figures \ref{fig:gev}-\ref{POT}, one can conclude that the new NSE method performs reasonably well or better compared with MLE.

\subsection{Empirical Study: Concrete Compressive Strength}
\label{sec:concrete}
In this section, we study the concrete compressive strength as a function of seven ingredients and age based on the dataset from UCI Machine Learning Repository \citep{yeh1998modeling}. Apart from the component types, the properties of concrete are influenced by the mixing proportions. Although technical references consist of experimental data describing thousands of different mixes, no one has yet made a composite of this information. Moreover, a mix is almost never described with all of the important details indicated; thus, a strength prediction from the available data is a highly uncertain task. In this approach, the compressive strength of concrete is a function of the eight input features.

We use the concrete compressive strength as the response variable and the other eight features as variables to estimate the linear regression model \eqref{linreg}. Before applying these methods, all the variables are normalized first since the ranges differ a lot. The sample size is 1030, and the replication number in the new NSE method is set to be 500. OLS is applied first, and it shows that the intercept, the coefficients before Coarse Aggregate, and Fine Aggregate are insignificant. In principle, when all the values of the covariates are zero, it is easy to argue that the concrete compressive strength is zero. Thus, it is reasonable to fit the model through the origin, i.e., let the intercept be zero. In this case, OLS shows that all the coefficients are significant. Then, we use the NSE method to fit the linear model without origin.

The histograms in Figure \ref{fig:concres} give a direct demonstration of how the new NSE method adjusts the coefficients to make the residuals satisfy the normality assumption. Obviously, the density function corresponding to the NSE method is closer to the normal density than that of OLS. Although OLS estimated parameter values lead to significance in normality tests, the overall residual normality tests failed with p-values to be 0.004 for Lilliefors Test and 0.04 for Jarque-Bera Test. In contrast, NSE resulted residuals have normality test p-values to be 0.91 for Lilliefors Test and 0.64 for Jarque-Bera Test.

\begin{figure}[H]
	\begin{minipage}[]{0.49\linewidth}
		\centering
		\includegraphics[width=2.5in]{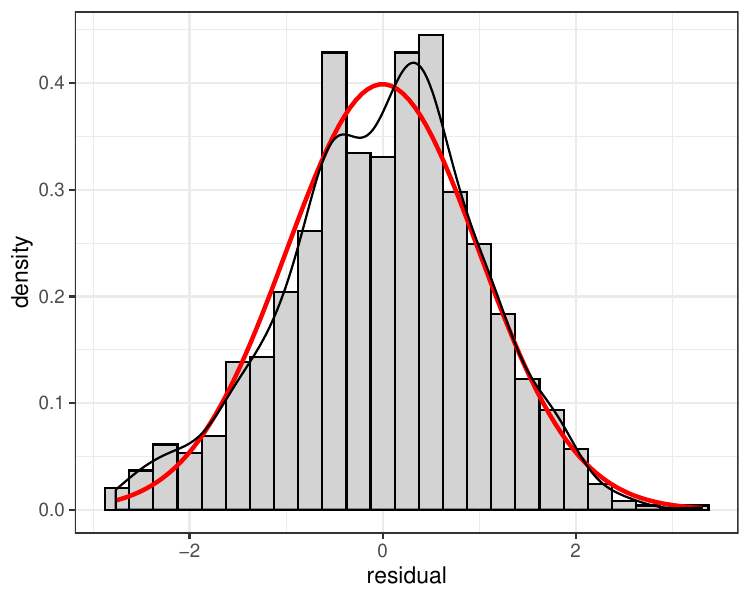}
	\end{minipage}
	\begin{minipage}[]{0.49\linewidth}
		\centering
		\includegraphics[width=2.5in]{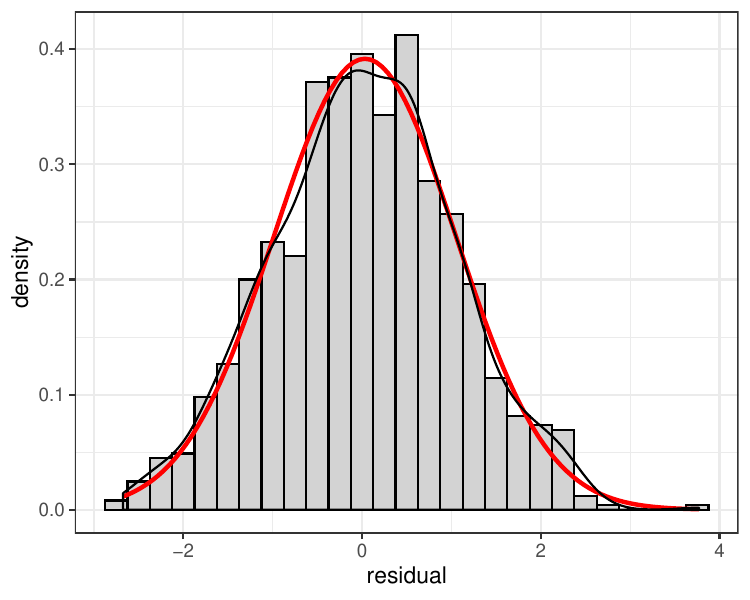}
	\end{minipage}
	\caption{Concrete Data: The left panel is the histogram of the residuals generated by OLS with p-values to be 0.004 for Lilliefors Test and 0.04 for Jarque-Bera Test. The right panel is the histogram of the residuals generated by the new NSE method with p-values to be 0.91 for Lilliefors Test and 0.64 for Jarque-Bera Test. The black line is the fitted density function. The red line is the density function of the normal distribution with the same mean and variance as the residuals. }
	\label{fig:concres}
\end{figure}

\begin{figure}[]
	\begin{minipage}[]{0.49\linewidth}
		\centering
		\includegraphics[width=2.5in]{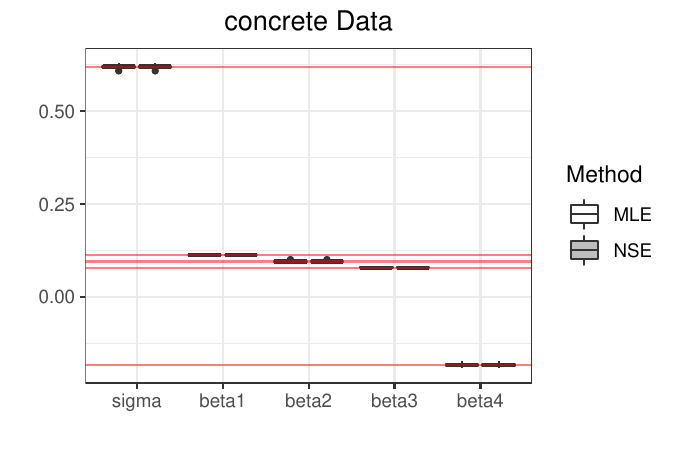}
	\end{minipage}
	\begin{minipage}[]{0.49\linewidth}
		\centering
		\includegraphics[width=2.5in]{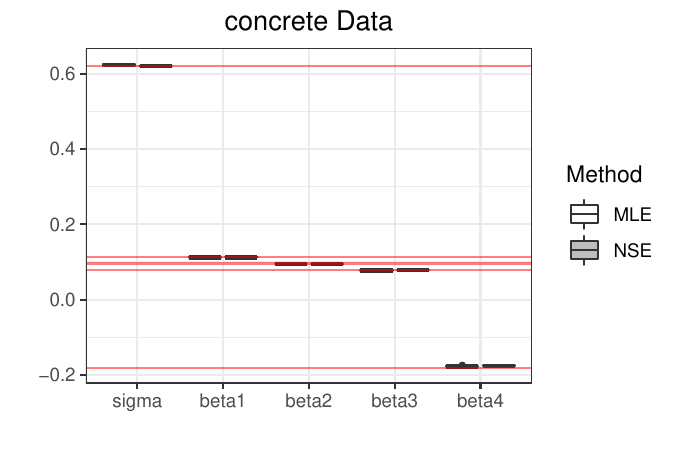}
	\end{minipage}
	\\
	\begin{minipage}[]{0.49\linewidth}
		\centering
		\includegraphics[width=2.5in]{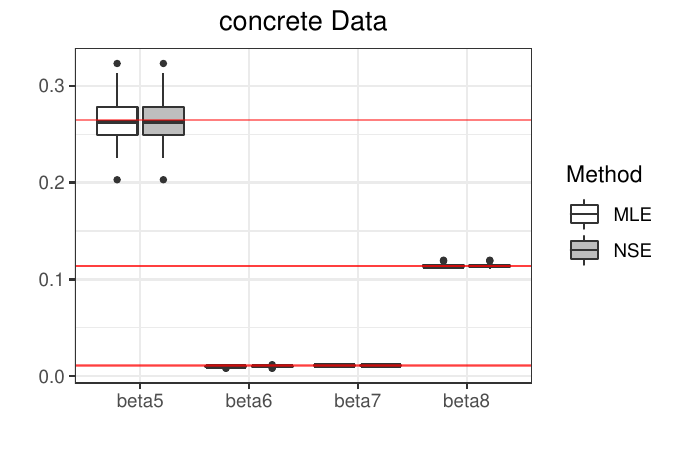}
	\end{minipage}
	\begin{minipage}[]{0.49\linewidth}
		\centering
		\includegraphics[width=2.5in]{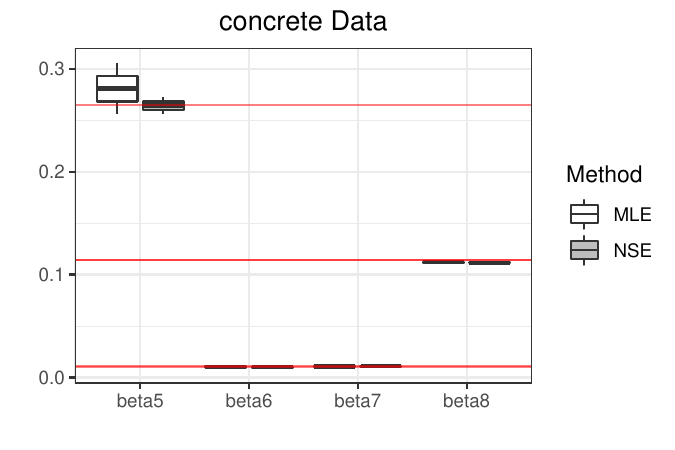}
	\end{minipage}
	\caption{Concrete Data: the left panels show the estimates of two methods before applying normality tests. The right panels show the estimates of the replications which pass the additional normality tests. Red lines indicate the estimates of MLE on the whole dataset for comparison. These are the estimates of coefficients in the model. The parameter names corresponding to the box plots are marked on the horizontal axis.}
	\label{fig:3}
\end{figure}

Recall that MLE is based on minimizing squared absolute errors, while NSE is based on minimizing relative errors. After models are fitted, their inferences can be very different. For this example, the costs of different materials can be different, and the ratios of product quality to cost can be different. Examples and applications like this one in diversified areas are numerous. It is safe to say that NSE is certainly an alternative approach for regression analysis. In some scenarios, e.g., preserving distribution assumption, it can give better estimations and lead to better decision-making.

\subsection{Proof of Proposition \ref{prop1}}
In this section, we prove the behavior of $q_n^{(1)}$ and $q_n^{(2)}$ at $1$ as displayed in the left panel of Figure \ref{mqr}. Denote
\begin{equation}
	F(t,n)=\PP(q_{n}^{(1)}\leq t).
\end{equation}
\begin{lemma}
	\label{lemma}
	If $X_1,\ldots,X_n$ are i.i.d samples from a distribution with probability density function $f(x)$, then the joint density function of $(X_{(i)}:X_{(1)},\ldots,X_{(n)})$ is
	$$ f(x_1,\ldots,x_n)=\begin{cases}
		n! \prod_{i=1}^{n} f(x_n) & \Rm{if} \ x_1<\ldots<x_n, \\
		0 & \Rm{otherwise}.
	\end{cases}$$
\end{lemma}

Then we can derive $F(t,n)$. Denote $x=(x_1,\ldots,x_n)$ and $y=(y_1,\ldots,y_n)$. By Lemma \ref{lemma},
\begin{equation*}
	F(t,n)=\int_{0<y_1<\ldots<y_n} n! e^{-\sum_{i=1}^{n}y_i}\Big(\int_D n! e^{-\sum_{i=1}^{n}x_i}dx\Big)dy,
\end{equation*}
where $$D=\{x\in\R^n: 0<x_1<ty_1,x_1<x_2<ty_2,\ldots,x_{n-1}<x_n<ty_n\}.$$

For $1=i_1<i_2<\ldots<i_k\leq n$, define $$j_1=i_2-i_1, \ j_2=i_3-i_2,\ \ldots, \ j_{k-1}=i_k-i_{k-1}, j_k=n+1-i_k.$$  Let $g(y_1,\ldots,y_n)=\int_D e^{-\sum_{i=1}^{n}x_i}dx$. By integrating $x_n$, then $x_{n-1}$, \ldots, finally $x_1$ successively we have
\begin{equation*}
	g(y_1,\ldots,y_n)=\sum \frac{(-1)^k}{j_1!\cdots j_k!} \Big[e^{-j_1 ty_{i_1}-\cdots-j_k ty_{i_k}}-e^{-j_2 t y_{i_2}-\cdots-j_k ty_{i_k}}\Big],
\end{equation*}
where the summation is over any possible $1=i_1<i_2<\ldots<i_k\leq n$.

\begin{lemma}
	\label{int}
	For $c_1,\ldots,c_n>0$, we have
	\begin{equation*}
		\int_{0<y_1<\cdots<y_n} e^{-c_1 y_1-\cdots-c_n y_n} dy=\frac{1}{c_n (c_n+c_{n-1})\cdots (c_n+\cdots+c_1)}.
	\end{equation*}
\end{lemma}

For simplicity, define $[c_1,\ldots,c_n]$ as
\begin{equation*}
	[c_1,\ldots,c_n]=\frac{1}{c_n (c_n+c_{n-1})\cdots (c_n+\cdots+c_1)}.
\end{equation*}
Then we have
\begin{equation*}
	\begin{aligned}
		F(t,n)=(n!)^2 \sum \frac{(-1)^k}{j_1!\cdots j_k!}\{[tj_1+1,&1,\ldots,1,tj_2+1,1,\ldots,1,tj_k+1,1,\ldots,1] \\
		-[1,&1,\ldots,1,tj_2+1,1,\ldots,1,tj_k+1,1,\ldots,1]\},
	\end{aligned}
\end{equation*}
where the summation is over any possible $\sum_{i=1}^{k} j_i=n$, $j_i\geq 1$. For convenience we further let
$$(i_1,\ldots,i_k)_1=\frac{(-1)^k}{j_1!\cdots j_k!}[tj_1+1,1,\ldots,1,tj_2+1,1,\ldots,1,tj_k+1,1,\ldots,1],$$
$$(i_1,\ldots,i_k)_2=\frac{(-1)^k}{j_1!\cdots j_k!}[1,1,\ldots,1,tj_2+1,1,\ldots,1,tj_k+1,1,\ldots,1].$$
So we have
\begin{equation*}
	F(t,n)=(n!)^2 \sum \{(i_1,\ldots,i_k)_1-(i_1,\ldots,i_k)_2\},
\end{equation*}

\begin{theorem}
	\label{thm_point0}
	$$\PP(q_n^{(1)}\leq 1)=F(1,n)=\frac{1}{n+1}.$$
\end{theorem}
\begin{proof}[Proof of Theorem \ref{thm_point0}]
	Let $t=1$. Notice the fact that for any $k>1$, $i_k<n$,
	$$(i_1,\ldots,i_k)_1=-(i_1,\ldots,i_k, n)_1,$$
	$$(i_1,\ldots,i_k)_2=-(i_1,\ldots,i_k, n)_2.$$
	So we have
	\begin{equation*}
		\begin{aligned}
			F(1,n)&=(n!)^2 \{-(1)_2-(1,n)_2\}\\
			&=(n!)^2\{\frac{1}{n!}\times\frac{1}{n!}-\frac{1}{(n-1)!}\times\frac{1}{(n+1)!}\}\\
			&=1-\frac{n}{n+1}=\frac{1}{n+1}.
		\end{aligned}
	\end{equation*}
\end{proof}
Theorem \ref{thm_point0} indicates that
$$\lim_{n\to\infty} \PP(q_n^{(1)}>1)=1,$$ and similarly, we can prove that $$\lim_{n\to\infty} \PP(q_n^{(2)}>1)=1.$$

\subsection{Proof of Theorem \ref{thm:mid}}
Before proving the theorem, we need some necessary preparations.

Inspired by the asymptotic normality of central order statistics, let $r$ be an integer satisfying $1\le r\le n$, and define
\begin{equation}
	a_{r,n}=\frac{\sqrt{r}}{\sqrt{(n+1)(n+1-r)}}, \quad b_{r,n}=\log\frac{n}{n-r+1}.
\end{equation}
Denote $\rho_{r,n}(t)=a_{r,n}t+b_{r,n}$. Then
\begin{equation}
	\label{eq_lower}
	\begin{aligned}
		\PP(a_{r,n}^{-1}(X_{(r)}-b_{r,n}) \leq t)&=\PP(X_{(r)} \leq \rho_{r,n}(t))\\&=\PP(\sum_{i=1}^{n} \Rm{1}_{(-\infty,\rho_{r,n}(t))}(X_i)\geq r),
	\end{aligned}
\end{equation}
\begin{equation}
	\label{eq_upper}
	\begin{aligned}
		\PP(a_{r,n}^{-1}(X_{(r)}-b_{r,n}) \geq t)&=\PP(X_{(r)} \geq \rho_{r,n}(t))\\&=\PP(\sum_{i=1}^{n} \Rm{1}_{(-\infty,\rho_{r,n}(t))}(X_i)\leq r).
	\end{aligned}
\end{equation}
Let $B_i=\Rm{1}_{(-\infty,\rho_{r,n}(t))}(X_i)$. Notice that $B_1,\ldots,B_n$ are i.i.d Bernoulli r.v.'s with parameter $p=1-\exp(-\rho_{r,n}(t))$.

\begin{equation*}
	\sum_{i=1}^{n} \Rm{1}_{(-\infty,\rho_{r,n}(t))}(X_i)=\sum_{i=1}^{n} B_i \sim \Rm{Binomial}(n,p).
\end{equation*}
By Hoeffding's inequality, we get
\begin{equation}
	\label{eq_tail}
	\PP(\sum_{i=1}^{n} \Rm{1}_{(-\infty,\rho_{r,n}(t))}(X_i)\geq r) \leq \exp(-2\frac{(np-r)^2}{n}), \quad r\geq np,
\end{equation}
\begin{equation}
	\label{eq_head}
	\PP(\sum_{i=1}^{n} \Rm{1}_{(-\infty,\rho_{r,n}(t))}(X_i)\leq r) \leq \exp(-2\frac{(np-r)^2}{n}), \quad r\leq np.
\end{equation}

Next, we begin to prove the theorem. The proof of Theorem \ref{thm:mid} contains two main parts. In the first part, we fix $\alpha_n$ and choose $\beta_n$ according to Condition \ref{cond:1}. In the second part, we fix $\beta_n$ and choose $\alpha_n$ according to Condition \ref{cond:1}. In these two parts, we prove the result respectively. Combining the two results in two parts, we prove the final result.

(i). Let $$\alpha_n=\alpha, \quad \beta_n=1-\sqrt{\frac{\log n \log\log n}{n}},$$ where $\alpha>0$ is a fixed number chosen arbitrary. We aim to prove
$$\widetilde{Z}_n^{(1)}\overset{\Rm{d}}{=}\max\{X_{(r)}/Y_{(r)}: \alpha_n n< r < \beta_n n\} \overset{P}{\to} 1.$$
For any $\varepsilon>0$, choose $\delta>0$ such that
\begin{equation}
	\delta<\frac{\varepsilon}{2+\varepsilon}\log\frac{1}{1-\alpha}.
\end{equation}
Since $$b_{r,n}=\log\frac{n}{n-r+1}\in (\log{\frac{1}{1-\alpha_n}}, \log{\frac{1}{1-\beta_n}}),$$
for any $0<\tau<\delta$, the inequality
\begin{equation}
	\label{eq_delta}
	\frac{b_{r,n}+\tau}{b_{r,n}-\tau}\leq1+\varepsilon
\end{equation}
always holds for any integer $n$ and $\alpha_n n < r < \beta_n n$.

Let $A_{r,n}=\{X_{(r)}/Y_{(r)}\leq 1+\varepsilon\}$. Then
\begin{equation*}
	\{\widetilde{Z}_n^{(1)}\leq 1+\varepsilon\} = \mathop\bigcap_{\alpha_n n \leq r \leq \beta_n n} A_{r,n},
\end{equation*}
\begin{equation*}
	\PP(\widetilde{Z}_n^{(1)}> 1+\varepsilon)=\PP(\mathop\bigcup_{\alpha_n n \leq r \leq \beta_n n} A_{r,n}^c)\leq \sum_{\alpha_n n \leq r \leq \beta_n n} \PP(A_{r,n}^c).
\end{equation*}
From \eqref{eq_delta}, we have $A_{r,n} \supset \{X_{(r)}<b_{r,n}+\delta\} \cap\{Y_{(r)}>b_{r,n}-\delta\}$. So
\begin{equation}
	\label{eq_Pdecom1}
	\begin{aligned}
		\PP(A_{r,n}^c) &\leq \PP(\{X_{(r)}\geq b_{r,n}+\delta\}\cup \{Y_{(r)}\leq b_{r,n}-\delta\})\\&\leq\PP(X_{(r)}\geq b_{r,n}+\delta)+\PP(Y_{(r)}\leq b_{r,n}-\delta)\\
		&=\PP(a_{r,n}^{-1}(X_{(r)}-b_{r,n})\geq a_{r,n}^{-1}\delta)+\PP(a_{r,n}^{-1}(Y_{(r)}-b_{r,n})\leq -a_{r,n}^{-1} \delta).
	\end{aligned}
\end{equation}

By \eqref{eq_upper} and \eqref{eq_head} we can bound $\PP(a_{r,n}^{-1}(X_{(r)}-b_{r,n})\geq a_{r,n}^{-1}\delta)$. We have
\begin{equation*}
	t=a_{r,n}^{-1}\delta,\quad p=1-\exp(-\rho_{r,n}(t))=1-\frac{n-r+1}{n} e^{-\delta}.
\end{equation*}
Apparently \begin{equation*}
	\begin{aligned}
		np-r&=(n-r)(1-e^{-\delta})-e^{-\delta}\geq (1-\beta_n)n(1-e^{-\delta})-e^{-\delta}>0
	\end{aligned}
\end{equation*}
for large enough $n$. So the condition of Hoeffding's inequality \eqref{eq_head} holds. Then we have
\begin{equation*}
	\begin{aligned}
		\PP(a_{r,n}^{-1}(X_{(r)}-b_{r,n})\geq a_{r,n}^{-1}\delta_n) & \leq \exp(-2\frac{(np-r)^2}{n})\\
		&\leq\exp(-2n[(1-\beta_n)(1-e^{-\delta})-\frac{1}{n}e^{-\delta}]^2) \\
		&=\exp(-2n[\sqrt{\frac{\log n \log\log n}{n}}(1-e^{-\delta})-\frac{1}{n}e^{-\delta}]^2) \\
		&=\exp(-2 \log n\log\log n(1-e^{-\delta})^2+o(1)) \\
		&\leq 2\exp(-2\log n)=\frac{2}{n^2}
	\end{aligned}
\end{equation*}
for sufficiently large $n$.

By \eqref{eq_lower} and \eqref{eq_tail} we can bound $\PP(a_{r,n}^{-1}(Y_{(r)}-b_{r,n})\leq -a_{r,n}^{-1} \delta)$. We have
\begin{equation*}
	t=-a_{r,n}^{-1}\delta,\quad p=1-\exp(-\rho_{r,n}(t))=1-\frac{n-r+1}{n} e^{\delta}.
\end{equation*}
Apparently, since $e^\delta>1$, $$np-r=(n-r)(1-e^{\delta})-e^{\delta}\leq (1-\beta_n) n (1-e^\delta)-e^{\delta}<-e^\delta<0,$$ the condition of Hoeffding's inequality \eqref{eq_tail} holds. Then we have
\begin{equation*}
	\begin{aligned}
		\PP(a_{r,n}^{-1}(Y_{(r)}-b_{r,n})\leq -a_{r,n}^{-1} \delta) & \leq \exp(-2\frac{(np-r)^2}{n})\\
		&=\exp(-2n[(1-\beta_n)(1-e^{\delta})-\frac{1}{n}e^\delta]^2) \\
		&=\exp(-2n[(1-e^\delta)\sqrt{\frac{\log n \log\log n}{n}}-\frac{1}{n} e^\delta]^2)\\
		&=\exp(-2 \log n\log\log n(1-e^{-\delta})^2+o(1)) \\
		&\leq 2\exp(-2\log n)=\frac{2}{n^2}
	\end{aligned}
\end{equation*}
for sufficiently large $n$.

Recall \eqref{eq_Pdecom1}, for sufficiently large $n$, we have
\begin{equation*}
	\begin{aligned}
		\PP(A_{r,n}^c) &\leq \frac{4}{n^2}
	\end{aligned}
\end{equation*}
holds for every $\alpha_n n < r < \beta_n n$. This implies
\begin{equation*}
	\PP(\widetilde{Z}_n^{(1)}> 1+\varepsilon)\leq \sum_{\alpha_n n \leq r \leq \beta_n n} \PP(A_{r,n}^c) <\frac{4n}{n^2} \to 0
\end{equation*}
as $n\to\infty$.

(ii). Let $$\alpha_n=\sqrt{\frac{\log n \log\log n}{n}}, \quad \beta_n=\beta,$$ where $\beta>0$ is a fixed number chosen arbitrary. We aim to prove
$$\widetilde{Z}_n^{(2)}\overset{\Rm{d}}{=}\max\{X_{(r)}/Y_{(r)}: \alpha_n n< r < \beta_n n\} \overset{P}{\to} 1.$$
For any $\varepsilon>0$, this time we choose $$\delta_n=\sqrt{\frac{\log n \log\log\log n}{n}}.$$
Notice that $\log(1-x)^{-1}>x$, so
\begin{equation}
	\delta_n<\log\frac{1}{1-\alpha_n}, \quad  \delta_n<\frac{\varepsilon}{2+\varepsilon}\log\frac{1}{1-\alpha_n}
\end{equation}
still hold for sufficiently large $n$.
Since $$b_{r,n}=\log\frac{n}{n-r+1}\in (\log{\frac{1}{1-\alpha_n}}, \log{\frac{1}{1-\beta_n}}),$$
for any $0<\tau<\delta_n$, the inequality
\begin{equation}
	\label{eq_deltan}
	\frac{b_{r,n}+\tau}{b_{r,n}-\tau}\leq1+\varepsilon
\end{equation}
always holds for any integer $n$ and $\alpha_n n < r < \beta_n n$.

Let $A_{r,n}=\{X_{(r)}/Y_{(r)}\leq 1+\varepsilon\}$. Then
\begin{equation*}
	\{\widetilde{Z}_n^{(1)}\leq 1+\varepsilon\} = \mathop\bigcap_{\alpha_n n \leq r \leq \beta_n n} A_{r,n},
\end{equation*}
\begin{equation*}
	\PP(\widetilde{Z}_n^{(2)}> 1+\varepsilon)=\PP(\mathop\bigcup_{\alpha_n n \leq r \leq \beta_n n} A_{r,n}^c)\leq \sum_{\alpha_n n \leq r \leq \beta_n n} \PP(A_{r,n}^c).
\end{equation*}
From \eqref{eq_deltan}, we have $A_{r,n} \supset \{X_{(r)}<b_{r,n}+\delta_n\} \cap\{Y_{(r)}>b_{r,n}-\delta_n\}$. So
\begin{equation}
	\label{eq_Pdecom2}
	\begin{aligned}
		\PP(A_{r,n}^c) &\leq \PP(\{X_{(r)}\geq b_{r,n}+\delta_n\}\cup \{Y_{(r)}\leq b_{r,n}-\delta_n\})\\&\leq\PP(X_{(r)}\geq b_{r,n}+\delta_n)+\PP(Y_{(r)}\leq b_{r,n}-\delta_n)\\
		&=\PP(a_{r,n}^{-1}(X_{(r)}-b_{r,n})\geq a_{r,n}^{-1}\delta_n)+\PP(a_{r,n}^{-1}(Y_{(r)}-b_{r,n})\leq -a_{r,n}^{-1} \delta_n).
	\end{aligned}
\end{equation}
By \eqref{eq_upper} and \eqref{eq_head} we can bound $\PP(a_{r,n}^{-1}(X_{(r)}-b_{r,n})\geq a_{r,n}^{-1}\delta_n)$. We have
\begin{equation*}
	t=a_{r,n}^{-1}\delta_n,\quad p=1-\exp(-\rho_{r,n}(t))=1-\frac{n-r+1}{n} e^{-\delta_n}.
\end{equation*}
Apparently \begin{equation*}
	\begin{aligned}
		np-r&=(n-r)(1-e^{-\delta})-e^{-\delta}\geq (1-\beta)n(1-e^{-\delta_n})-e^{-\delta_n}>0
	\end{aligned}
\end{equation*}
for large enough $n$. So the condition of Hoeffding's inequality \eqref{eq_head} holds. Then we have
\begin{equation*}
	\begin{aligned}
		\PP(a_{r,n}^{-1}(X_{(r)}-b_{r,n})\geq a_{r,n}^{-1}\delta_n) & \leq \exp(-2\frac{(np-r)^2}{n})\\
		&\leq\exp(-2n[(1-\beta)(1-e^{-\delta_n})-\frac{1}{n}e^{-\delta_n}]^2) \\
		&=\exp(-2(1-\beta)^2 n(\delta_n-\frac{1}{2}\delta_n^2+\frac{1}{6}\delta_n^3)^2+o(1))\\
		&=\exp(-2(1-\beta)^2 n \delta_n^2+o(1)) \\
		&=\exp(-2(1-\beta)^2 \log n \log\log\log n+o(1)) \\
		&\leq 2\exp(-2\log n)=\frac{2}{n^2}
	\end{aligned}
\end{equation*}
for sufficiently large $n$.

By \eqref{eq_lower} and \eqref{eq_tail} we can bound $\PP(a_{r,n}^{-1}(Y_{(r)}-b_{r,n})\leq -a_{r,n}^{-1} \delta_n)$. We have
\begin{equation*}
	t=-a_{r,n}^{-1}\delta_n,\quad p=1-\exp(-\rho_{r,n}(t))=1-\frac{n-r+1}{n} e^{\delta_n}.
\end{equation*}
Apparently, since $e^\delta_n>1$, $$np-r=(n-r)(1-e^{\delta_n})-e^{\delta_n}\leq (1-\beta) n (1-e^\delta_n)-e^{\delta_n}<-e^{\delta_n}<0,$$ the condition of Hoeffding's inequality \eqref{eq_tail} holds. Then we have
\begin{equation*}
	\begin{aligned}
		\PP(a_{r,n}^{-1}(Y_{(r)}-b_{r,n})\leq -a_{r,n}^{-1} \delta_n) & \leq \exp(-2\frac{(np-r)^2}{n})\\
		&=\exp(-2n[(1-\beta)(1-e^{\delta_n})-\frac{1}{n}e^{\delta_n}]^2) \\
		&=\exp(-2(1-\beta)^2 n(\delta_n+\frac{1}{2}\delta_n^2+\frac{1}{6}\delta_n^3)^2+o(1))\\
		&=\exp(-2(1-\beta)^2 n \delta_n^2+o(1)) \\
		&=\exp(-2(1-\beta)^2 \log n \log\log\log n+o(1)) \\
		&\leq 2\exp(-2\log n)=\frac{2}{n^2}
	\end{aligned}
\end{equation*}
for sufficiently large $n$.

Recall \eqref{eq_Pdecom2}, for sufficiently large $n$, we have
\begin{equation*}
	\begin{aligned}
		\PP(A_{r,n}^c) &\leq \frac{4}{n^2}
	\end{aligned}
\end{equation*}
holds for every $\alpha_n n < r < \beta_n n$. This implies
\begin{equation*}
	\PP(\widetilde{Z}_n^{(2)}> 1+\varepsilon)\leq \sum_{\alpha_n n \leq r \leq \beta_n n} \PP(A_{r,n}^c) <\frac{4n}{n^2} \to 0
\end{equation*}
as $n\to\infty$.

(iii). For any $\varepsilon>0$, since $\widetilde{Z}_n = \max\{\widetilde{Z}_n^{(1)}, \widetilde{Z}_n^{(2)}\}$ and
\begin{equation}
	\PP(\widetilde{Z}_n^{(1)}> 1+\varepsilon) \to 0, \quad \PP(\widetilde{Z}_n^{(2)}> 1+\varepsilon) \to 0
\end{equation}
as $n\to\infty$, we have
\begin{equation*}
	\PP(\widetilde{Z}_n> 1+\varepsilon) \to 0
\end{equation*}
as $n\to\infty$. Changing the role of $X$ and $Y$, we have
\begin{equation*}
	\PP(\widetilde{Z}_n< 1-\varepsilon) \to 0
\end{equation*}
as $n\to\infty$. So we conclude that
$$ \widetilde{Z}_n \overset{P}{\to} 1.$$

\subsection{Proof of Corollary \ref{thm:threshold}}
Based on Theorem \ref{thm:mid}, we only need to prove that there exist $\alpha_n$ and $\beta_n$ satisfy Condition \ref{cond:1} such that
\begin{equation*}
	\PP(X_{\alpha_n n:n}\leq u_n< v_n\leq X_{\beta_n n:n},Y_{\alpha_n n:n}\leq u_n< v_n\leq Y_{\beta_n n:n})\to 1.
\end{equation*}
We can rewrite it in an equivalent form
\begin{equation}
	\label{eq:equi}
	\PP(\alpha_n\leq \widehat{F}^X_n(u_n)<\widehat{F}_n^X(v_n)\leq\beta_n,\alpha_n\leq \widehat{F}^Y_n(u_n)<\widehat{F}_n^Y(v_n)\leq\beta_n)\to 1,
\end{equation}
where $F_n^X(t)$ and $F_n^Y(t)$ are empirical distribution functions of $X_1,\ldots,X_n$ and $Y_1,\ldots,Y_n$ respectively. Since
\begin{equation}
	\label{eq:sq1}
	\begin{aligned}
		&\PP(\alpha_n\leq \widehat{F}^X_n(u_n)<\widehat{F}_n^X(v_n)\leq\beta_n,\alpha_n\leq \widehat{F}^Y_n(u_n)<\widehat{F}_n^Y(v_n)\leq\beta_n) \\
		&=\PP(\alpha_n\leq \widehat{F}^X_n(u_n)<\widehat{F}_n^X(v_n)\leq\beta_n)^2, \\
	\end{aligned}
\end{equation}
it suffices to prove that
\begin{equation}
	\label{eq:suff}
	\PP(\alpha_n\leq \widehat{F}^X_n(u_n)<\widehat{F}_n^X(v_n)\leq\beta_n)\to 1.
\end{equation}
We set
$$\alpha_n=\frac{F(u_n)}{2},\ \beta_n=\frac{1+F(v_n)}{2}.$$
It is obvious $\{\alpha_n,\beta_n\}$ satisfies \eqref{condition1} in Condition \ref{cond:1}.
If
$$\abs{\widehat{F}^X_n(u_n)-F(u_n)}<\frac{F(u_n)}{2},\ \abs{\widehat{F}^X_n(v_n)-F(v_n)}<\frac{1-F(v_n)}{2},$$
then $\alpha_n\leq \widehat{F}^X_n(u_n)<\widehat{F}_n^X(v_n)\leq\beta_n$ holds. Thus, it suffices to prove that
\begin{equation}
	\label{eq:finalsuff}
	\PP\left(\sup_t\abs{\widehat{F}^X_n(t)-F(t)}<\min\{\frac{F(u_n)}{2},\ \frac{1-F(v_n)}{2}\}\right)\to 1.
\end{equation}
We have
\begin{equation}
	\label{eq:com1}
	\sqrt{n}\sup_t\abs{\widehat{F}^X_n(t)-F(t)}\xrightarrow{d}\max_{0\leq s\leq 1}\abs{B_s-sB_1},
\end{equation}
where $B_t$ is a Brownian motion starting at 0 (\cite{durrett2010probability}). $\max_{0\leq t\leq 1}\abs{B_t-tB_1}$ have a non-degenerate distribution over $\mathbb{R}^+$.
Combining
\begin{equation}
	\label{eq:com2}
	\min\{F(u_n),1-F(v_n)\}/\sqrt{\frac{\log n}{n}}\to \infty	
\end{equation}
and \eqref{eq:com1}, we have \eqref{eq:finalsuff} holds.

\subsection{Proof of Theorem \ref{thm:right} }
Assume $X_1,\ldots,X_n$ are i.i.d samples from distribution $F$ (with density $f$). If there exist constant $a_n>0$ and $b_n$ such that $$F^n(a_n x+b_n) \to G(x)$$
at all continuity points of $G$, where $G$ is the distribution function of one of the three types: Fr\'echet, Weibull and Gumbel, we say $F\in \mathcal{D}(G)$.

The following lemma (see page 307 of \cite{David2003Order}) serves the key to our proof.
\begin{lemma}
	When $F \in \mathcal{D}(G)$, assume $$n a_n f(a_n x+b_n) [F(a_n x+b_n)]^{n-1} \to g(x).$$ The $k$-dimensional vector
	\begin{equation}
		\label{eq_vector}
		(\frac{X_{n:n}-b_n}{a_n}, \cdots, \frac{X_{n-k+1:n}-b_n}{a_n})
	\end{equation}
	converges in distribution to the vector $(W_1,\ldots,W_k)$, whose joint pdf is
	$$g(w_1,\ldots,w_k)=G(w_k)\prod_{i=1}^{k} \frac{g(w_i)}{G(w_i)}, \quad w_1>\cdots>w_k.$$
	Additionally, the pdf of the random vector in \eqref{eq_vector} is
	$$g_n (x_1,\ldots,x_n)=[F(a_n x_k+b_n)]^{n-k} \prod_{i=1}^{k} (n-i+1)a_n f(a_n x_i+b_n).$$
\end{lemma}

In our proof of Theorem \ref{thm:right}, we need to apply the lemma to exponential distribution $F(x)=1-\exp(-x)$. In this case, $a_n=1$ and $b_n=\ln n$, and
\begin{equation}
	\label{eq_limit_exp}
	g(w_1,\ldots,w_k)=\exp(-\exp(-w_k)-\sum_{i=1}^{k} w_i),
\end{equation}
\begin{equation}
	\label{eq_n_exp}
	g_n(x_1,\ldots,x_k)=F(x_k+\ln n)^{n-k}\prod_{i=1}^{k} \frac{n-i+1}{n} \exp(-x_i).
\end{equation}

In our proof of Theorem \ref{thm:left}, we need to apply the lemma to Fr\'echet distribution $F(x)=\exp(-x^{-1})$. In this case, $a_n=n$ and $b_n=0$, and
\begin{equation}
	\label{eq_limit_frechet}
	g(w_1,\ldots,w_k)=\exp(-w_k^{-1}) \prod_{i=1}^{k} w_i^{-2},
\end{equation}
\begin{equation}
	\label{eq_n_frechet}
	g_n(x_1,\ldots,x_k)=F(nx_k)^{n-k} \prod_{i=1}^{k} \frac{n-i+1}{n}  x_i^{-2} \exp(-\frac{1}{nx_i}).
\end{equation}

\begin{proof}[\bf Proof of Theorem \ref{thm:right}]
	Denote $$W_1=X_{n:n}-\ln n, \ldots, W_k=X_{n-k+1:n}-\ln n,$$
	$$Z_1=Y_{n:n}-\ln n, \ldots, Z_k=Y_{n-k+1:n}-\ln n.$$
	For any $t\in\R$,
	\begin{equation*}
		\begin{aligned}
			\PP(R_n <t)&=\PP(\frac{X_{n:n}}{Y_{n:n}}<t, \ldots, \frac{X_{n-k+1:n}}{Y_{n-k+1:n}}<t) \\
			&=\PP(\frac{W_1+\ln n}{Z_1+\ln n}<t, \ldots, \frac{W_k+\ln n}{Z_k+\ln n}<t) \\
			&=\PP(W_1<tZ_1+(t-1)\ln n,\ldots, W_k<tZ_k+(t-1)\ln n)\\
			&=\int_{D(n,t)} g_n(w_1,\ldots,w_k) g_n(z_1,\ldots,z_k) dw dz,
		\end{aligned}
	\end{equation*}
	where
	\begin{equation}
		\label{eq_set_exp}
		\begin{aligned}
			D(n,t)=\{(w,z):z_1>\cdots>z_k, \
			w_2<&w_1<tz_1+(t-1)\ln n, \\
			& \cdots \\
			w_k<&w_{k-1}<tz_{k-1}+(t-1)\ln n,\\
			-\ln n <&w_k<tz_k+(t-1)\ln n \}.
		\end{aligned}
	\end{equation}
	From \eqref{eq_n_exp}, we can observe that $g_n \nearrow g$ as $n\to\infty$.
	
	When $t>1$, the set $$D(n,t) \overset{n\to\infty}{\longrightarrow} \{(w,z): z_1>\cdots>z_k, \ w_1>\cdots>w_k\}\overset{\Delta}{=}D.$$
	Notice that $g\in L^1$, by the dominated convergence theorem, we get
	\begin{equation*}
		\begin{aligned}
			\PP(R_n <t)&=\int_{D(n,t)} g_n(w_1,\ldots,w_k) g_n(z_1,\ldots,z_k) dw dz \\
			&\to \int_D g(w_1,\ldots,w_k) g(z_1,\ldots,z_k) dw dz =1.
		\end{aligned}
	\end{equation*}
	
	When $t<1$, the set $$D(n,t) \overset{n\to\infty}{\longrightarrow} \varnothing.$$
	Notice that $g\in L^1$, by the dominated convergence theorem, we get
	\begin{equation*}
		\begin{aligned}
			\PP(R_n <t)&=\int_{D(n,t)} g_n(w_1,\ldots,w_k) g_n(z_1,\ldots,z_k) dw dz \\
			&\to \int_\varnothing g(w_1,\ldots,w_k) g(z_1,\ldots,z_k) dw dz =0.
		\end{aligned}
	\end{equation*}
	
	Thus we conclude that $R_n \overset{P}{\to} 1$.
\end{proof}
\subsection{Proof of Theorem \ref{thm:left}}

\begin{proof}[\bf The existence of non-degenerate distribution:]
We first prove that $q_{n,\Lambda_l}\triangleq\max_{i\in\Lambda_l}\{X_{(i)}/Y_{(i)}\}$ converges weakly to a random variable with a non-degenerate distribution
	Let $X'_1, \ldots, X'_n$ and $Y'_1,\ldots,Y'_n$ be i.i.d sample from unit Fr\'echet distribution. Then $$(\frac{X'_{n:n}}{Y'_{n:n}},\ldots,\frac{X'_{1:n}}{Y'_{1:n}})\overset{d}{=} (\frac{X_{1:n}}{Y_{1:n}},\ldots,\frac{X_{n:n}}{Y_{n:n}}).$$
	Denote $$W_1=nX'_{n:n}, \ldots, W_k=nX'_{n-k+1:n},$$
	$$Z_1=nY'_{n:n}, \ldots, Z_k=nY'_{n-k+1:n}.$$
	For any $t\in\R$,
	\begin{equation*}
		\begin{aligned}
			\PP(L_n <t)&=\PP(\frac{X_{1:n}}{Y_{1:n}}<t, \ldots, \frac{X_{k:n}}{Y_{k:n}}<t) \\
			&=\PP(\frac{X'_{n:n}}{Y'_{n:n}}<t, \ldots, \frac{X'_{n-k+1:n}}{Y'_{n-k+1:n}}<t) \\
			&=\PP(\frac{W_1}{Z_1}<t, \ldots, \frac{W_k}{Z_k}<t) \\
			&=\PP(W_1<tZ_1,\ldots, W_k<tZ_k)\\
			&=\int_{D(t)} g_n(w_1,\ldots,w_k) g_n(z_1,\ldots,z_k) dw dz,
		\end{aligned}
	\end{equation*}
	where
	\begin{equation}
		\label{eq_set_frechet}
		\begin{aligned}
			D(t)=\{(w,z):z_1>\cdots>z_k, \
			w_2<&w_1<tz_1, \\
			& \cdots \\
			w_k<&w_{k-1}<tz_{k-1},\\
			0 <&w_k<tz_k \}.
		\end{aligned}
	\end{equation}
	From \eqref{eq_n_frechet}, we can observe that $g_n \nearrow g$ as $n\to\infty$.
	Notice that $g\in L^1$, by the dominated convergence theorem, we get
	\begin{equation*}
		\begin{aligned}
			\PP(L_n <t)&=\int_{D(t)} g_n(w_1,\ldots,w_k) g_n(z_1,\ldots,z_k) dw dz \\
			&\to \int_{D(t)} g(w_1,\ldots,w_k) g(z_1,\ldots,z_k) dw dz
		\end{aligned}
	\end{equation*}
	
	Thus we conclude that $L_n$ converges weakly to a non-degenerate distribution.
\end{proof}

Given that the reciprocal of a unit exponential random variable is a unit Fr\'{e}chet random variable, then these two events
$\{\max_{n-\ell+1 \leq i \leq n}\{\frac{X_{(i)}}{Y_{(i)}}\}\leq t)\} $ and $\{\max_{1 \leq i \leq \ell}\{\frac{Y_{(i)}}{X_{(i)}}\}\leq t)\} $ have the same probability. In the subsequent proofs in this subsection, we assume both $X_i$ and $Y_i$ are unit Fr\'{e}chet random variables.

We restate Theorem \ref{thm:left} as
\begin{theorem}[Main theorem: formula of $R_m(t)$]
We have known

$$
g(w_1,...,w_m) = \exp(-w_m^{-1})\Pi_{i=1}^mw_i^{-2}
$$

$$
D(t) = \{(w,z):z_1>...>z_m,w_2<w_1<tz_1,...,w_m<w_{m-1}<tz_{m-1},0<w_m<tz_m\}
$$

Define $R_m(t)=\int _{D(t)}g(w_1,...,w_m)g(z_1,...,z_m)dwdz$,then we have
$$
R_{m}(t)=\frac{\sum_{j=0}^{m-1}(-1)^j(B_{m-1,j}\cdot k^{m-j-1})}{k^{2m-1}}, k=1+1/t, t\geq 0
$$

where $B_{m,j}=\frac{\binom{2m+2}{m-j}\binom{m+j}{m}}{m+1}$.
\end{theorem}
Subsequently, we show the proof of this main theorem.
\begin{theorem}[Recursive formula for $f_{n}$]
$$f_{n}(t,x_{1}, x_{2},\cdots,x_{n}) = \sum_{i=1}^{n}\binom{n}{i}(-1)^{i+1}e^{-\frac{i}{t(x_{n-i+1})}}f_{n-i}(t,x_{1}, x_{2}...x_{n-1})$$
\end{theorem}
This formula is derived inductively by the following expansion:
$$
f_{n}(t,x_{1},\cdots,x_{n})=\int_{0}^{tx_{1}}\cdots\int_{y_{i}}^{tx_{i+1}}\cdots\int_{y_{n-1}}^{tx_{n}}f_{Y_{(1)},Y_{(2)},\cdots,Y_{(n)}} (y_1,y_2,\cdots,y_n)\mathrm{d}y_{n}\cdots\mathrm{d}y_{1}.
$$
\begin{ex}[Illustration of $f_{n}$]
Here are some examples of the recursive function $f_{n}(x)$ defined above.
\begin{align*}
&f_{0}(t)=1 \\
&f_{1}(t,x_{1})=e^{-\frac{1}{tx_{1}}}\\
&f_{2}(t,x_{1},x_{2})=2e^{-(\frac{1}{tx_{1}}+\frac{1}{tx_{2}})}-e^{-\frac{2}{tx_{1}}}\\
&f_{3}(t,x_{1},x_{2},x_{3})=(6e^{-(\frac{1}{tx_{1}}+\frac{1}{tx_{2}}+\frac{1}{tx_{3}})}-3e^{-(\frac{2}{tx_{1}}+\frac{1}{tx_{3}})})-3e^{-(\frac{1}{tx_{1}}+\frac{2}{tx_{2}})}+e^{-\frac{3}{tx_{1}}}\\
&f_{4}(t,x_{1},x_{2},x_{3},x_{4})=(24e^{-(\frac{1}{tx_{1}}+\frac{1}{tx_{2}}+\frac{1}{tx_{3}}+\frac{1}{tx_{4}})}-12e^{-(\frac{2}{tx_{1}}+\frac{1}{tx_{3}}+\frac{1}{tx_{4}})}-12e^{-(\frac{1}{tx_{1}}+\frac{2}{tx_{2}}+\frac{1}{tx_{4}})}+4e^{-(\frac{3}{tx_{1}}+\frac{1}{tx_{4}})})\\
&-(12e^{-(\frac{1}{tx_{1}}+\frac{1}{tx_{2}}+\frac{2}{tx_{3}})}-6e^{-\frac{2}{tx_{1}}+\frac{2}{tx_{3}}})+4e^{-(\frac{1}{tx_{1}}+\frac{3}{tx_{2}})}-e^{-\frac{4}{tx_{1}}}
\end{align*}
\end{ex}

\begin{lemma}[Integration formula]
$$
\int_{0}^{\infty}\int_{0}^{x_{n}}\cdots\int_{0}^{x_{2}} x_{1}^{-2}e^{-\frac{k_{1}}{x_{1}}}\cdots x_{n}^{-2}e^{-\frac{k_{n}}{x_{n}}}dx_{1}\cdots dx_{n}
=\frac{1}{k_{1}}\frac{1}{k_{1}+k_{2}}\cdots \frac{1}{k_{1}+k_{2}\cdots +k_{n}}
$$
\end{lemma}

\begin{definition}[The right tail asymptotic distribution]
\[
\mathcal{R}_{n,m}(t)
= P_{n}\!\left( \max_{\,n-m+1\le i \le n} \frac{Y_{(i)}}{X_{(i)}} \le t \right).
\]

Recalling the fundamental formulation, note that for $i\le n-m$ we have the upper bound $Y_{(i)}\le t\,x_{n-m+1}$; hence we set the corresponding $x_i$’s equal to $x_{n-m+1}$. We can therefore rewrite the expression as
\begin{align*}
\mathcal{R}_{n,m}(t)
&= \int_{0< x_{n-m+1}\le x_{n-m+2}\le \cdots \le x_{n}}
r_{n,m}\!\big(t, x_{n-m+1},\ldots,x_{n}\big)\,
P\!\left( X_{(i)}=x_{i}\ \forall\, n-m+1\le i\le n \right)
\,\mathrm{d}x_{n-m+1}\cdots \mathrm{d}x_n\\[4pt]
&= \int_{0< x_{n-m+1}\le x_{n-m+2}\le \cdots \le x_{n}}
r_{n,m}\!\big(t, x_{n-m+1},\ldots,x_{n}\big)\,
\frac{n!}{(n-m)!}
e^{-\frac{n-m}{x_{n-m+1}}}
\prod_{j=n-m+1}^{n} x_{j}^{-2}e^{-\frac{1}{x_{j}}}
\mathrm{d}x_{n-m+1}\cdots \mathrm{d}x_n,
\end{align*}
where
\begin{align*}
&r_{n,m}\!\big(t, x_{n-m+1},\ldots,x_{n}\big)\\
&= P\!\Big( Y_{(1)}\le t x_{n-m+1},\,\ldots,\,
Y_{(n-m+1)}\le t x_{n-m+1},\,\ldots,\,
Y_{(n)}\le t x_{n}
\,\Big|\,
X_{(n-m+1)}=x_{n-m+1},\ldots,X_{(n)}=x_{n}\Big)\\[4pt]
&= \int_{0}^{t x_{n-m+1}}
\int_{y_1}^{t x_{n-m+1}}
\cdots
\int_{y_{n-m-1}}^{t x_{n-m+1}}
\int_{y_{n-m}}^{t x_{n-m+1}}
\int_{y_{n-m+1}}^{t x_{n-m+2}}
\cdots
\int_{y_{n-1}}^{t x_{n}}
f_{Y_{(1)},\ldots,Y_{(n)}}(y_1,\ldots,y_n)\,
\mathrm{d}y_n\cdots \mathrm{d}y_1.
\end{align*}
\end{definition}

\begin{lemma}[Recursive formula for $r_{n,k}$]\label{r_{n,k}}
For $k<n$,
$$
r_{n,k}(t)=\left(\sum_{i=1}^{k-1}\binom{n}{i}(-1)^{i+1}r_{n-i,k-i}e^{-\frac{i}{tx_{n-i+1}}}\right)+(-1)^{k+1}\binom{n-1}{k-1}e^{-\frac{n}{tx_{n-k+1}}}.
$$
Otherwise, $r_{n,n}(t)=f_{n}(t)$
\end{lemma}

\begin{lemma}[Recursive formula for $\mathcal{R}_{n,m}$]
For $m<n$,
$$
\mathcal{R}_{n,m}(t)=\sum_{j=1}^{m-1}(-1)^{j+1}\binom{n}{j}\mathcal{R}_{n-j,m-j}\prod_{i=0}^{j-1}\frac{n-i}{nk-i}+(-1)^{m+1}\binom{n-1}{m-1}\prod_{q=0}^{m-1}\frac{n-q}{nk-q},k=1+\frac{1}{t},
$$
with base case
$\mathcal{R}_{n,1}(t)=\frac{1}{k}$.
\end{lemma}

This recursive formula for $\mathcal{R}_{n,m}$ follows directly from the formulation of $r_{n,m}$ in Lemma \ref{r_{n,k}}. The important trick is to substitute $t$ with variable $k=1+1/t$. Note that $\lim_{n\to \infty}\mathcal{R}_{n,m}(t)\in [0,1]$, and $\frac{n!}{(n-m)!}=\mathcal{O}(n^m)$. Next, we will only consider the part of the $\mathcal{R}_{n,m}(t)$ excluding $\frac{n!}{(n-m)!}$, i.e. we will study this function:
$$
R_{n,m}(t)=\mathcal{R}_{n,m}(t)/\frac{n!}{(n-m)!}
$$
\begin{theorem}[Recursive formula for ${R}_{n,m}$]
 For $m<n$,
$$
R_{n,m}(t)=\left(\sum_{j=1}^{m-1}(-1)^{j+1}\frac{R_{n-j,m-j}(t)}{j!}\prod_{i=0}^{j-1}\frac{n-i}{nk-i}\right)+(-1)^{m+1}\frac{1}{(m-1)!n}\prod_{q=0}^{m-1}\frac{n-q}{nk-q},k=1+1/t,
$$
with base case $R_{n,1}(t)=\frac{1}{nk}$.
\end{theorem}

Consequently, $R_{n,m}(t)=\mathcal{O}(n^{-m})$.
In the next theorem, we perform an alternative expression for
$\lim_{n\to\infty} R_{n,m}(t)$, which will be used to derive the asymptotic distribution.

Note that for finite real number $b$,
$$
\lim_{n\to \infty}R_{n,m}(t)=\lim_{n\to \infty}R_{n+b,m}(t)
$$
By replacing $n$ with $n+m-1$, we instead study functions $Q_{n,m}(t)=R_{n+m-1,m}(t).$ For simplicity, we omit the $(t)$ when writing.
\begin{theorem}\label{Q}
By induction we can rewrite $Q_{n,m}$ as follows.
\begin{align*}
   Q_{n,m}&=\frac{1}{k}\left(Q_{n,m-1}-\frac{1}{2k}g_{1,m-1}\cdot\left(Q_{n,m-2}-\frac{1}{3k}g_{2,m-1}\cdot(\cdots \frac{1}{(m-2)k}g_{m-3,m-1}\cdot\left(Q_{n,2}\right.\right.\right.\\
   &\left.\left.\left.-\frac{1}{(m-1)k}g_{m-2,m-1}\frac{1}{k}f_{m-1}\right)\right)\right),
\end{align*}
where
$$
f_a=\frac{1}{n}-\frac{n}{(n-a/k+a)(n+a)}
$$
$$
g_{a-b,a}=\frac{n+b}{n-\frac{a-b}{k}+a}
$$
$$Q_{n,2}=\frac{f_1}{k^2}$$
\end{theorem}

\begin{ex}
\begin{align*}
&Q_{n,3}=\frac{1}{k}Q_{n,2}-\frac{1}{2k^3}g_{1,2}\cdot f_2=\frac{1}{k^3}f_1-\frac{1}{2k^3}g_{1,2}\cdot f_2  \\
&Q_{n,4}=\frac{1}{k}Q_{n,3}-\frac{1}{2k^4}g_{1,3}\cdot(f_1-\frac{1}{3}g_{2,3}\cdot f_3)\\
&Q_{n,5}=\frac{1}{k}Q_{n,4}-\frac{1}{2k^2}g_{1,4}\cdot(Q_{n,3}-\frac{1}{3k^3}g_{2,4}\cdot (f_1-\frac{1}{4}g_{3,4}f_4))
\end{align*}
\end{ex}
In the next part, we will use Laurent Expansion to find the limit distribution $\lim_{n\to \infty}R_{n,m}(t)$.
\begin{lemma}[Laurent Expansion of $f_a$ and $g_{a-b,a}$]\label{3.12}
$$s=1-\frac{1}{k}$$
$$
f_1(t)=\frac{1}{n}-\frac{n}{(n-1/k+1)(n+1)}
$$

\begin{center}
\begin{tabular}{ c c  }
 $1$ & $a_0=0$\\
 $\frac{1}{n}$ & $a_1=0$ \\
 $\frac{1}{n^2}$ & $a_2=1+s$ \\
 $\frac{1}{n^3}$ & $a_3=-(1+s+s^2)$ \\
 ...&...\\
 $\frac{1}{n^j}$ & $a_j=(-1)^j(1+s...s^{j-1})$ \\
\end{tabular}
\end{center}
$\newline$
$$
f_a(t)=\frac{1}{n}-\frac{n}{(n-a/k+a)(n+a)}
$$
\begin{center}
\begin{tabular}{ c c  }
 $1$ & $0$\\
 $\frac{1}{n}$ & $0$ \\
 $\frac{1}{n^2}$ & $a\cdot a_2$ \\
 $\frac{1}{n^3}$ & $a^2 \cdot a_3$ \\
 ...&...\\
 $\frac{1}{n^j}$ & $a^{j-1}\cdot a_j$ \\
\end{tabular}
\end{center}

$\newline$
$$
g_{a-b,a}(t)=\frac{n+b}{n-\frac{a-b}{k}+a}
$$

\begin{center}
\begin{tabular}{ c c  }
 $1$ & $1$\\
 $\frac{1}{n}$ & $-s(a-b)$ \\
 $\frac{1}{n^2}$ & $\frac{s}{k}(a-b)(ak-a+b)$ \\
 $\frac{1}{n^3}$ & $\frac{-s}{k^2}(a-b)(ak-a+b)^2$ \\
 ...&...\\
 $\frac{1}{n^j}$ & $\frac{(-1)^js}{k^{j-1}}(a-b)(ak-a+b)^{j-1}$ \\
\end{tabular}
\end{center}
Note that products of $f_a$ and $g_{a-b,a}$ are well-defined because their Laurent Expansion have coefficient $0$ for all $n^i,i \geq 1$.
\end{lemma}

\begin{definition}\label{T}
We define $T_{n,m}$ as follows:
\begin{align*}
T_{n,m}(a) \;=\;& \frac{1}{k}\Bigg(Q_{n,m-1}-\frac{1}{(a-(m-3))k}\, g_{a-(m-2),a}\cdot\Big(Q_{n,m-2} -\frac{1}{(a-(m-4))k}\, g_{a-(m-3),a}\cdot\big(\,\\\cdots& \frac{1}{(a-1)k}\, g_{a-2,a}\cdot
\big(Q_{n,2}-\frac{1}{ak}\, g_{a-1,a}\,\frac{1}{k} f_{a}\big)\big)\Bigg).
\end{align*}
The quantity $T_{n,m}(a)$ can be regarded as a generalized extension of $Q_{n,m}$, and it reduces to $Q_{n,m}$ when $a=m-1$, i.e.,
\[
T_{n,m}(m-1)=Q_{n,m}.
\]
We introduce $T_{n,m}(a)$ because this generalized form naturally appears in the recursive formulation of $Q_{n,m'}$ for $m'\geq m+1$. For instance,
\[
Q_{n,m+1}=\frac{1}{k}Q_{n,m}-\frac{1}{2k}g_{1,m}\,T_{n,m}(m),
\]
\[
Q_{n,m+2}=\frac{1}{k}Q_{n,m+1}-\frac{1}{2k}g_{1,m+1}\,T_{n,m+1}(m+1).
\]
\end{definition}

\begin{ex}
Here we illustrate the difference between $T_{n,m}(a)$ and $Q_{n,m}$ for small $m$.
$$
Q_{n,2}=\frac{1}{k^2}f_1
$$
$$
T_{n,2}(a)=\frac{1}{k^2}f_a
$$
$$
Q_{n,3}=\frac{f_1}{k^3}-\frac{1}{2k^3}g_{1,2}\cdot f_2
$$
$$
T_{n,3}(a)=\frac{f_1}{k^3}-\frac{1}{ak^3}g_{a-1,a}\cdot f_a
$$


$$
Q_{n,4}=\frac{1}{k}Q_{n,3}-\frac{1}{2k^4}g_{1,3}\cdot(f_1-\frac{1}{3}g_{2,3}\cdot f_3)
$$
$$
T_{n,4}(a)=\frac{1}{k}Q_{n,3}-\frac{1}{(a-1)k^4}g_{a-2,a}\cdot(f_1-\frac{1}{a}g_{a-1,a}\cdot f_a)
$$
\end{ex}

\begin{lemma}\label{3.15}
Let $q,h,p$ be positive integers, and $k$ be a real constant, define the following function:
$$
f_{q,h,p}(a)=a^q(ak-a+1)^h(ak-a+b)^p
$$
For simplicity, here we will just abbreviate $f_{q,h,p}(a)$ as $f(a)$. Set
$$
g_0(a) = \frac{f(a)-f(c)}{a-c}
$$
$$
g_l = \frac{g_{l-1}(a)-g_{l-1}(c+l)}{a-(c+l)},l\ge1
$$
Then we have
$$
g_{q+h+p-1}(a)=(k-1)^{h+p}
$$
\end{lemma}
\begin{proof}
\begin{align*}
f(a)-f(c)&=
    a^q(ak-a+1)^h(ak-a+b)^p-c^q(ck-c+1)^h(ck-c+b)^p\\&=\sum_{i=0}^h\sum_{j=0}^p\binom{p}{j}\binom{h}{i}a^{i+j+q}(k-1)^{i+j}b^{p-j}-\sum_{i=0}^h\sum_{j=0}^p\binom{p}{j}\binom{h}{i}c^{i+j+q}(k-1)^{i+j}b^{p-j}\nonumber\\
    &=\sum_{i=0}^h\sum_{j=0}^p\binom{p}{j}\binom{h}{i}(k-1)^{i+j}b^{p-j}(a^{i+j+q}-c^{i+j+q})\nonumber\\
    &=(a-c)\sum_{i=0}^h\sum_{j=0}^p\binom{p}{j}\binom{h}{i}(k-1)^{i+j}b^{p-j}\left(\sum_{m=0}^{i+j+q-1}a^{i+j+q-1-m}c^m\right)
\end{align*}
For simplicity we write $M=i+j+q$,
\begin{align*}
    &\cfrac{f(a)-f(c)}{a-c}-\cfrac{f(c+1)-f(c)}{c+1-c}\\
    &\quad=\sum_{i=0}^h\sum_{j=0}^p\binom{p}{j}\binom{h}{i}(k-1)^{i+j}b^{p-j}\left(\sum_{m=0}^{i+j+q-1}a^{i+j+q-1-m}c^m-(c+1)^{i+j+q-1-m}c^m\right)\\
    &\quad=(a-(c+1))\sum_{i=0}^h\sum_{j=0}^p\binom{p}{j}\binom{h}{i}(k-1)^{i+j}b^{p-j}\left(\sum_{m=0}^{M-1}c^m\sum_{m'=0}^{M-2-m}a^{M-2-m-m'}(c+1)^{m'}\right)
\end{align*}

As we can see, after each subtraction we can factor out a new element $a-(c+i)$, which decreases the original power $a^{q+h+p}$ by $1$. By induction, it follows that this operation stops after factoring out exactly $q+h+p$ factors of the form $a-(c+i)$, where $i=0,\cdots, q+h+p-1$. Note that $f(a)$ has degree $q+h+p$, therefore, any terms involving variable $a$ with lower degree will be sent to $0$ through our operation. We are left out with the coefficient in front of $a^{q+h+p}$ in the original function $f(a)$, i.e., $(k-1)^{h+p}$.
\end{proof}
\begin{lemma}\label{3.16}
According to Lemma \ref{3.15}, let $H_b(a)=\frac{1-a^b}{1-a}$. $b$ is a positive integer. We then have
$$
g_{b-2}(a) = 1
$$
$$
\frac{1-a^b}{1-a}=a^{b-1}+a^{b-2}+...1,b\geq 1$$
\end{lemma}
\begin{proof}
First note that
\begin{align*}
 H_b(a)-H_b(2)&=\frac{1-a^b}{1-a}-\frac{1-2^b}{1-2}\\&=a^{b-1}-2^{b-1}+a^{b-2}-2^{b-2}+...1-1\\
 &=(a-2)\left(\sum_{i=1}^{b-1}2^{i-1}a^{b-1-i}+\sum_{i=1}^{b-2}2^{i-1}a^{b-2-i}+...+1\right)\\
 &=(a-2)\sum_{i=1}^{b-1}(2^i-1)a^{b-1-i}
\end{align*}
Then the proof is quite similar to the proof of the previous lemma. The key point is that we have a function consisting of powers of $a^i$, which means that $a^i-c^i$ can be broken down to the sum of terms of smaller degree. And we repeat this process, we can get the degree down to $0$ and only left with the coefficient of the term of the highest degree in the original function.
\end{proof}

\begin{definition}
$C(n,k)$ denotes the element in the $n$th row and $x$th column in the Catalan's triangle, and $C(n,k)=\frac{(n+k)!(n-k+1)}{k!(n+1)!}$.
For any integer $m\ge1$, set $C_1(n,x)=C(n,x)$ for $0\le x\le n$ and $C_2(n,x)=C(n+1,x)$ for $0\le x\le n+1$, and define inductively
$$
C_{m}(n,x)=\begin{cases}
\displaystyle \binom{n+x}{x} & \text{if }0\leq x<m,\\[6pt]
\displaystyle \binom{n+x}{x} - \binom{n+x}{x-m} & \text{if }m\leq x\leq n+m-1,\\[6pt]
\displaystyle 0 & \text{if }n+m-1<x.
\end{cases}
$$

\end{definition}

\begin{lemma}[Properties of Catalan’s Trapezoid]\label{3.18}
    $$C_{b-m+2}(m-3,a)+C_{b-m+1}(m-3,a-1)+\cdots+C_2(m-3,a-b+m)=C_{b-m+1}(m-2,a),a\ge b-m$$
    $$C_{b-m+2}(m-3,a)+C_{b-m+1}(m-3,a-1)+\cdots+C_{b-a-m+2}(m-3,0)=C_{b-m+1}(m-2,a),a<b-m$$
\end{lemma}
\begin{proof}
    For $a\ge b-m$,
    \begin{align*}
      &C_{b-m+1}(m-2,a)-\left(C_{b-m+2}(m-3,a)+C_{b-m+1}(m-3,a-1)+\cdots+C_2(m-3,a-b+m)\right)\\
      &=\binom{a+m-2}{a}-\binom{a+m-2}{a-b+m-1}-\left(\binom{a+m-3}{a}-\binom{a+m-3}{a-b+m-2}+\binom{a+m-4}{a-1}-\right.\\
      &\quad \left.\binom{a+m-4}{a-b+m-2}+\cdots+\binom{a-b+2m-3}{a-b+m}-\binom{a-b+2m-3}{a-b+m-2}\right)\\
      &=\left(\binom{a+m-2}{a}-\binom{a+m-3}{a}\right)-\left(\binom{a+m-2}{a-b+m-1}-\binom{a+m-3}{a-b+m-2}\right)-\left(\binom{a+m-4}{a-1}-\right.\\
      &\quad \binom{a+m-4}{a-b+m-2}\left.+\cdots+\binom{a-b+2m-3}{a-b+m}-\binom{a-b+2m-3}{a-b+m-2}\right)\\
      &=\binom{a+m-3}{a-1}-\binom{a+m-3}{a-b+m-1}-\left(\binom{a+m-4}{a-1}-\binom{a+m-4}{a-b+m-2}+\cdots+\binom{a-b+2m-3}{a-b+m}\right.\\
      &\quad\left.-\binom{a-b+2m-3}{a-b+m-2}\right)\\
      &=\binom{a+m-2-(b-m+1)}{a-(b-m+1)}-\binom{a+m-2-(b-m+1)}{a-b+m-1}\\
      &=0
    \end{align*}
    For $a<b-m$, similarly, it follows that
    \begin{align*}
    &C_{b-m+1}(m-2,a)-(C_{b-m+2}(m-3,a)+C_{b-m+1}(m-3,a-1)+\cdots+C_{b-a-m+2}(m-3,0))\\
    &=\binom{m-2+a-a}{a-a}-\binom{m-3}{0}\\
    &=0
    \end{align*}
\end{proof}

Notation: we use $\floor*{\frac{1}{n^x}}f$ to denote the coefficient of $\frac{1}{n^x}$ of the Laurent expansion of $f$.
\begin{theorem}\label{th5}
We have the following result:
$$
\frac{1}{(a-(m-2))}\floor*{\frac{1}{n^j}}T_{n,m}(a)=\floor*{\frac{1}{n^j}}Q_{n,m}, \forall 1\leq j\leq m, a\geq m-1
$$
In particular,
\begin{equation}
\floor*{\frac{1}{n^j}}T_{n,m}(a)=\floor*{\frac{1}{n^j}}Q_{n,m}=0,\forall 1\leq j\leq m-1, a\geq m-1
\label{p1}
\end{equation}

\begin{equation}
\begin{aligned}
\frac{1}{(a-(m-2))}\floor*{\frac{1}{n^m}}T_{n,m}(a)
&=\floor*{\frac{1}{n^m}}Q_{n,m} \\
&=\frac{(-1)^m}{k^m}\mathbf{1}^T\left(\begin{bmatrix}
1\\-s\\s^2\\\vdots\\(-s)^{m-2}
\end{bmatrix}\odot
\begin{bmatrix}
C(m-2,0)\\C(m-2,1)\\C(m-2,2)\\\vdots\\C(m-2,m-2)
\end{bmatrix}
\odot\begin{bmatrix}
a_m\\a_{m-1}\\a_{m-2}\\\vdots\\a_2
\end{bmatrix}\right), a\geq m-1,
\end{aligned}
\label{p2}
\end{equation}
where $\odot$ denotes the Hadamard product, i.e., the element-wise multiplication.
\end{theorem}

\begin{proof}
Because $R_{n,m}(t)=\mathcal{O}(\frac{1}{n^m})$, we know $\floor*{\frac{1}{n^j}}T_{n,m}(a)=0,\floor*{\frac{1}{n^j}}Q_{n,m}=0$, $\forall 1\le j\leq m-1$, $a\ge m-1$, which means (\ref{p1}) is valid.

We will prove (\ref{p2}) using induction. When $m=2$, (\ref{p2}) is obviously fulfilled. First, we assume that $\forall 3\leq m\leq x$, (\ref{p2}) holds,  and we will proceed to compute whether it holds for the case where $m = x + 1$. For some integer $y\geq m$, we have
\begin{equation}
\floor*{\frac{1}{n^y}}T_{n,m}(a)\cdot g_{a-b,a}=\sum_{i=0}^{y-m}\floor*{\frac{1}{n^{y-i}}}T_{n,m}(a)\cdot \floor*{\frac{1}{n^{i}}}g_{a-b,a}
\label{(3)}
\end{equation}

According to the formulation of $T_{n,m}$, we write the following formula with $s=1-1/k$:
\begin{align}
 \floor*{\frac{1}{n^x}}T_{n,m}(a)&=  \frac{1}{k}\floor*{\frac{1}{n^x}}Q_{n,m-1}- \frac{1}{(a-m+3)k}\floor*{\frac{1}{n^x}}g_{a-m+2,a}T_{n,m-1}(a)\nonumber\\
 &=\frac{1}{k}\floor*{\frac{1}{n^x}}T_{n,m-1}(m-2)- \frac{1}{(a-m+3)k}\floor*{\frac{1}{n^x}}g_{a-m+2,a}T_{n,m-1}(a)\nonumber\\
 &=\frac{\floor*{\frac{1}{n^x}}((a-m+3)T_{n,m-1}(m-2)-T_{n,m-1}(a))-\Sigma_{i=1}^{n-m+1}\floor*{\frac{1}{n^i}}g_{a-m+2,a}\cdot\floor*{\frac{1}{n^{x-i}}}T_{n,m-1}(a)}{(a-m+3)k}\nonumber\\
 &=\frac{1}{k}(\floor*{\frac{1}{n^x}}(T_{n,m-1}(m-2)-\frac{T_{n,m-1}(a)}{(a-m+3)})\nonumber\\
 &-\frac{\Sigma_{i=1}^{x-m+1}\frac{(-1)^is}{k^{i-1}}(a-m+2)(a(k-1)+m-2)^{i-1}\floor*{\frac{1}{n^{x-i}}}T_{n,m-1}(a)}{(a-m+3)})
 \label{(4)}
\end{align}

It is easy to observe that after expanding, the expression $\floor*{\frac{1}{n^x}}T_{n,x}(a)$ is composed of terms with lower indices: $\floor*{\frac{1}{n^j}}T_{n,x-1}(a), j\leq x$. At the same time, $\floor*{\frac{1}{n^x}}T_{n,x}(a)$ is expressed in terms of $a_i,2\leq i\leq x$, we will look at the formulation of $\floor*{\frac{1}{n^x}}T_{n,x}(a)$ by focusing on coefficients in front of each $a_i$. Next, we will illustrate the construction of $\floor*{\frac{1}{n^j}}T_{n,j}(a)$ starting from the base case $T_{n,3}$, and investigate how the coefficient in front of each $a_i$ accumulates.\\

By equation (\ref{(3)}) and Lemma \ref{3.12}, we write the Laurent expansion of $T_{n,3}(a)$:

$$
T_{n,3}(a)=\frac{1}{k^3}[f_1-\frac{1}{a}g_{a-1,a}f_a]
$$

\begin{center}
\begin{tabular}{ c c  }
 $1$ & $0$\\[6pt]
 $\dfrac{1}{n}$ & $0$ \\[6pt]
 $\dfrac{1}{n^2}$ & $0$ \\[6pt]
 $\dfrac{1}{n^3}$ &
 $\displaystyle
 -\frac{1}{k^3}(a-1)
 \begin{bmatrix}
 1\\[4pt]
 -s
 \end{bmatrix}
 \cdot
 \begin{bmatrix}
 a_3\\[4pt]
 a_2
 \end{bmatrix}
 $\\[12pt]
 $\dfrac{1}{n^4}$ &
 $\displaystyle
 -\frac{1}{k^3}(a-1)
 \begin{bmatrix}
 a+1\\[4pt]
 -as\\[4pt]
 \dfrac{s}{k}\bigl(a(k-1)+1\bigr)
 \end{bmatrix}
 \cdot
 \begin{bmatrix}
 a_4\\[4pt]
 a_3\\[4pt]
 a_2
 \end{bmatrix}
 $\\[12pt]
 $\dfrac{1}{n^5}$ &
 $\displaystyle
 -\frac{1}{k^3}(a-1)
 \begin{bmatrix}
 a^2+a+1\\[4pt]
 -a^2 s\\[4pt]
 a\displaystyle\frac{s}{k}\bigl(a(k-1)+1\bigr)\\[4pt]
 -\displaystyle\frac{s}{k^2}\bigl(a(k-1)+1\bigr)^2
 \end{bmatrix}
 \cdot
 \begin{bmatrix}
 a_5\\[4pt]
 a_4\\[4pt]
 a_3\\[4pt]
 a_2
 \end{bmatrix}
 $\\[12pt]
 $\vdots$ & $\vdots$ \\[6pt]
 $\dfrac{1}{n^b}$ &
 $\displaystyle
 -\frac{1}{k^3}(a-1)
 \begin{bmatrix}
 \dfrac{1-a^{\,b-2}}{1-a}\\[6pt]
 -a^{\,b-3}s\\[6pt]
 a^{\,b-4}\dfrac{s}{k}\bigl(a(k-1)+1\bigr)\\[6pt]
 -a^{\,b-5}\dfrac{s}{k^2}\bigl(a(k-1)+1\bigr)^2\\[6pt]
 \vdots\\[4pt]
 \dfrac{(-1)^b s}{k^{\,b-3}}\bigl(a(k-1)+1\bigr)^{b-3}
 \end{bmatrix}
 \cdot
 \begin{bmatrix}
 a_b\\[4pt]
 a_{b-1}\\[4pt]
 a_{b-2}\\[4pt]
 \vdots\\[4pt]
 a_2
 \end{bmatrix}
 $\\
\end{tabular}
\end{center}

We can see that for $b\ge3$,
\[
\begingroup
\renewcommand{\arraystretch}{1.25} 
\left\lfloor \displaystyle\frac{1}{n^b} \right\rfloor
\!\left(\displaystyle\frac{T_{n,3}(2)}{2-1} - \displaystyle\frac{T_{n,3}(a)}{a-1}\right)
= \displaystyle\frac{1}{k^3}
\begin{bmatrix}
\displaystyle \frac{1-a^{b-2}}{1-a}-\frac{1-2^{b-2}}{1-2} \\[6pt]
\displaystyle (-a^{b-3}-(-2^{b-3}))\,s \\[6pt]
\displaystyle a^{b-4}\frac{s}{k}\bigl(a(k-1)+1\bigr)-2^{b-4}\frac{s}{k}\bigl(2(k-1)+1\bigr) \\[6pt]
\displaystyle -a^{b-5}\frac{s}{k^2}\bigl(a(k-1)+1\bigr)^2+2^{b-5}\frac{s}{k^2}\bigl(2(k-1)+1\bigr)^2 \\[6pt]
\displaystyle \vdots \\[6pt]
\displaystyle \frac{(-1)^b s}{k^{b-3}}\bigl(a(k-1)+1\bigr)^{b-3} - \frac{(-1)^b s}{k^{b-3}}\bigl(2(k-1)+1\bigr)^{b-3}
\end{bmatrix}
\cdot
\begin{bmatrix}
\displaystyle a_b \\[6pt]
\displaystyle a_{b-1} \\[6pt]
\displaystyle a_{b-2} \\[6pt]
\displaystyle \vdots \\[6pt]
\displaystyle a_2
\end{bmatrix}
\endgroup
\]

By equation (\ref{(4)}) we can see that for $j\le 4$,

\begin{align*}
    &\left\lfloor \dfrac{1}{n^j} \right\rfloor T_{n,4}(a)
    = \dfrac{1}{k}\left\lfloor \dfrac{1}{n^j} \right\rfloor \left( \dfrac{T_{n,3}(2)}{2-1} - \dfrac{T_{n,3}(a)}{a-1} \right) \\
     & \quad  + s\dfrac{1}{k}(a-2) \dfrac{ \left\lfloor \dfrac{1}{n^{j-1}} \right\rfloor T_{n,3}(a) + \cdots + \dfrac{(-1)^{j-4}}{k^{j-4}} (ak-a+2)^{j-4} \left\lfloor \dfrac{1}{n^{3}} \right\rfloor T_{n,3}(a) }{a-1} \\
    &= \dfrac{1}{k}\left\lfloor \dfrac{1}{n^j} \right\rfloor \left( \dfrac{T_{n,3}(2)}{2-1} - \dfrac{T_{n,3}(a)}{a-1} \right)
        - s\dfrac{1}{k}(a-2) \left( \left\lfloor \dfrac{1}{n^{j-1}} \right\rfloor \dfrac{T_{n,3}(a)}{1-a} + \cdots + \dfrac{(-1)^{j-4}}{k^{j-4}} (ak-a+2)^{j-4} \left\lfloor \dfrac{1}{n^{3}} \right\rfloor \dfrac{T_{n,3}(a)}{1-a} \right) \\
    &= \dfrac{a-2}{k^4}
        \begin{bmatrix} a_j \\ a_{j-1} \\ a_{j-2} \\ \vdots \\a_3\\ a_2 \end{bmatrix}
        \cdot \left(
            \begin{bmatrix}
                \dfrac{\dfrac{1-a^{j-2}}{1-a}-\dfrac{1-2^{j-2}}{1-2}}{a-2} \\
                \dfrac{(-a^{j-3}-(-2^{j-3}))s}{a-2} \\
                \dfrac{a^{j-4}\dfrac{s}{k}(a(k-1)+1)-2^{j-4}\dfrac{s}{k}(2(k-1)+1)}{a-2} \\
                \vdots \\\dfrac{\dfrac{(-1)^{j-1}s}{k^{j-4}}a(a(k-1)+1)^{j-4}-\dfrac{(-1)^{j-1}s}{k^{j-4}}2(2(k-1)+1)^{j-4}}{a-2}\\
                \dfrac{\dfrac{(-1)^js}{k^{j-3}}(a(k-1)+1)^{j-3}-\dfrac{(-1)^js}{k^{j-3}}(2(k-1)+1)^{j-3}}{a-2}
            \end{bmatrix} \right.\\
            &\quad+
                \begin{bmatrix} 0 \\ -s\dfrac{1-a^{j-3}}{1-a} \\ a^{j-4}s^2 \\ \vdots \\\dfrac{(-1)^{j-1} s^2}{k^{j-5}}a\bigl(a(k-1)+1\bigr)^{j-5} \\\dfrac{(-1)^{j} s^2}{k^{j-4}}\bigl(a(k-1)+1\bigr)^{j-4}        \end{bmatrix}
                +
                \begin{bmatrix}
                0 \\ 0 \\ s\dfrac{a(k-1)+2}{k}\dfrac{1-a^{j-4}}{1-a} \\ \vdots \\s\dfrac{a(k-1)+2}{k}\dfrac{(-1)^{j-1} s}{k^{j-6}}a\bigl(a(k-1)+1\bigr)^{j-6}\\ s\dfrac{a(k-1)+2}{k}\dfrac{(-1)^{j} s}{k^{j-5}}\bigl(a(k-1)+1\bigr)^{j-5} \end{bmatrix}
                +\dots\\
               &\quad +\left.\begin{bmatrix} 0 \\ 0 \\ 0 \\ \vdots \\\dfrac{(-1)^{j-1}}{k^{j-4}}(a(k-1)+2)^{j-4}s\\ \dfrac{(-1)^j}{k^{j-4}}(a(k-1)+2)^{j-4}s^2\end{bmatrix}\right)\\
             &=(a-2)\cdot\begin{bmatrix}
                a_j\\a_{j-1}\\a_{j-2}\\\vdots\\a_2
             \end{bmatrix}\cdot\begin{bmatrix}
                a^{j-4}+R_{4,1,j-5}(a)\\-2s a^{j-4}+R_{4,2,j-5}(a)\\3s^2 a^{j-4}+R_{4,3,j-5}(a)\\\vdots\\(j-2)(-s)^{j-3}a^{j-4}+R_{4,j-2,j-5}(a)\\(j-2)(-s)^{j-2}a^{j-4}+R_{4,j-1,j-5}(a)
             \end{bmatrix},
            \end{align*}
where $\deg R_{m,i,j-m-1}(a)\le j-m-1$ and $R_{m,i,j-m-1}(a)$ denotes the terms of coefficient of $a_{j-i+1}$ in $\floor*{\frac{1}{n^{j}}}T_{n,m}(a)$ whose degrees of $a$ are less than $j-m-1$. And $R_{m,i,j-m-1}(a)=0$ when $m=j$. According to lemma \ref{3.15}, we only need to focus on the highest-degree coefficient of the polynomial. When we calculate $\floor*{\frac{1}{n^{4}}}T_{n,4}(a)$, the highest power of $\floor*{\frac{1}{n^{4}}}T_{n,3}(a)$ is $1$, and $\floor*{\frac{1}{n^j}} \left( \frac{T_{n,3}(2)}{2-1} - \frac{T_{n,3}(a)}{a-1}\right)$ leads to the cancellation of constant terms. Therefore, we need only focus on the linear term. When considering a more general $m$, after the polynomial in $\floor*{\frac{1}{n^{m}}}T_{n,3}(a)$ undergoes $m-3$ applications of the "differencing operation" (analogous to that in Lemma \ref{3.15}), only the coefficient of the term of highest degree is preserved and propagates to influence $\floor*{\frac{1}{n^{m}}}T_{n,m}(a)$. Next, we will give the general formula of the Laurent expansion of $T_{n,x}(a)$ and show how to compute $\floor*{\frac{1}{n^{j}}} T_{n,x+1}(a)$ from $\floor*{\frac{1}{n^j}} T_{n,x}(a),j\le x$.

\newpage
We obtain the Laurent expansion of $T_{n,x}(a)$:
$$
T_{n,x}(a)
$$

\begin{center}
\begin{tabular}{ c c  }
 $1$ & $0$\\[6pt]
 $\dfrac{1}{n}$ & $0$ \\[6pt]
 $\dfrac{1}{n^2}$ & $0$ \\[6pt]
 $\vdots $ & $\vdots$\\[6pt]
 $\dfrac{1}{n^x}$ &
 $\displaystyle
 (-1)^x\frac{a-x+2}{k^x}
 \begin{bmatrix}
 C(x-2,0)\\[4pt]
 C(x-2,1)(-s)\\[4pt]
 C(x-2,2)(-s)^2\\[4pt]
 \vdots\\[4pt]
 C(x-2,x-3)(-s)^{x-3}\\[4pt]
 C(x-2,x-2)(-s)^{x-2}\\[4pt]
 \end{bmatrix}
 \cdot
 \begin{bmatrix}
 a_x\\[4pt]
 a_{x-1}\\[4pt]
 a_{x-2}\\[4pt]
 \vdots\\[4pt]
 a_3\\[4pt]
 a_2\\[4pt]
 \end{bmatrix}
 $\\[12pt]
 $\dfrac{1}{n^{x+1}}$ &
 $\displaystyle
 (-1)^x\frac{a-x+2}{k^x}
 \begin{bmatrix}
 C_3(x-3,0)a+R_{x,1,0}(a)\\[4pt]
 C_3(x-3,1)(-s)a+R_{x,2,0}(a)\\[4pt]
 C_3(x-3,2)(-s)^2a+R_{x,3,0}(a)\\[4pt]
 \vdots\\[4pt]
 C_3(x-3,x-2)(-s)^{x-2}a+R_{x,x-1,0}(a)\\[4pt]
 C_3(x-3,x-1)(-s)^{x-1}a+R_{x,x,0}(a)\\[4pt]
 \end{bmatrix}
 \cdot
 \begin{bmatrix}
 a_{x+1}\\[4pt]
 a_{x}\\[4pt]
 a_{x-1}\\[4pt]
 \vdots\\[4pt]
 a_3\\[4pt]
 a_2\\[4pt]
 \end{bmatrix}
 $\\[12pt]
 $\vdots$ & $\vdots$ \\[6pt]
 $\dfrac{1}{n^j}$ &
 $\displaystyle
 (-1)^x\frac{a-x+2}{k^x}
 \begin{bmatrix}
 C_{j-x+2}(x-3,0)a^{j-x}+R_{x,1,j-x-1}(a)\\[6pt]
 C_{j-x+2}(x-3,1)(-s)a^{j-x}+R_{x,2,j-x-1}(a)\\[6pt]
 C_{j-x+2}(x-3,2)(-s)^2a^{j-x}+R_{x,3,j-x-1}(a)\\[6pt]
 \vdots\\[4pt]
 C_{j-x+2}(x-3,j-3)(-s)^{j-3}a^{j-x}+R_{x,j-2,j-x-1}(a)\\[6pt]
 C_{j-x+2}(x-3,j-2)(-s)^{j-2}a^{j-x}+R_{x,j-1,j-x-1}(a)
 \end{bmatrix}
 \cdot
 \begin{bmatrix}
 a_j\\[4pt]
 a_{j-1}\\[4pt]
 a_{j-2}\\[4pt]
 \vdots\\[4pt]
 a_3\\[4pt]
 a_2
 \end{bmatrix}
 $\\
\end{tabular}
\end{center}

From (\ref{(4)}) we can know that
\begin{align*}
    &\floor*{\frac{1}{n^{j}}}T_{n,x+1}(a)=\frac{1}{k}(\floor*{\frac{1}{n^j}}(T_{n,x}(x-1)-\frac{T_{n,x}(a)}{(a-x+2)})\nonumber\\
    &-\frac{\Sigma_{i=1}^{j-x}\frac{(-1)^is}{k^{i-1}}(a-x+1)(a(k-1)+x-1)^{i-1}\floor*{\frac{1}{n^{j-i}}}T_{n,m-1}(a)}{(a-x+2)})\\
    =&(-1)^{x+1}\frac{a-x+1}{k^{x+1}}
        \begin{bmatrix} a_j \\ a_{j-1} \\ a_{j-2} \\ \vdots \\a_3\\ a_2 \end{bmatrix}
        \cdot\left(
            \begin{bmatrix}
                C_{j-x+2}(x-3,0)a^{j-x-1}\\
                C_{j-x+2}(x-3,1)(-s)a^{j-x-1}\\
                C_{j-x+2}(x-3,2)(-s)^2a^{j-x-1}\\
                \vdots\\
                C_{j-x+2}(x-3,j-3)(-s)^{j-3}a^{j-x-1}\\
                C_{j-x+2}(x-3,j-2)(-s)^{j-2}a^{j-x-1}
            \end{bmatrix} \right.
            +
                \begin{bmatrix} 0\\
                C_{j-x+1}(x-3,0)(-s)a^{j-x-1}\\
                C_{j-x+1}(x-3,1)(-s)^2a^{j-x-1}\\
                \vdots\\
                C_{j-x+1}(x-3,j-4)(-s)^{j-3}a^{j-x-1}\\
                C_{j-x+1}(x-3,j-3)(-s)^{j-2}a^{j-x-1}       \end{bmatrix}\\
                &+
                \begin{bmatrix} 0\\
                0\\
                C_{j-x}(x-3,0)(-s)^2a^{j-x-1}\\
                \vdots\\
                C_{j-x}(x-3,j-6)(-s)^{j-3}a^{j-x-1}\\
                C_{j-x}(x-3,j-5)(-s)^{j-2}a^{j-x-1}
                 \end{bmatrix}
                +\dots+\left.
                \begin{bmatrix} 0 \\ 0 \\ 0 \\ \vdots \\C(x-2,x-3)(-s)^{j-3}a^{j-x-1}\\C(x-2,x-2)(-s)^{j-2}a^{j-x-1}\end{bmatrix}+
                \begin{bmatrix} R_{x+1,1,j-x-2}(a) \\ R_{x+1,2,j-x-2}(a) \\ R_{x+1,3,j-x-2}(a) \\ \vdots \\R_{x+1,x-2,j-x-2}(a)\\R_{x+1,x-1,j-x-2}(a)\end{bmatrix}
                \right)\\
                =&(-1)^{x+1}\frac{a-x+1}{k^{x+1}}
                \begin{bmatrix} a_j \\ a_{j-1} \\ a_{j-2} \\ \vdots \\a_3\\ a_2 \end{bmatrix}
        \cdot
            \begin{bmatrix}
                C_{j-x+1}(x-2,0)a^{j-x-1}+R_{x+1,1,j-x-2}(a)\\
                C_{j-x+1}(x-2,1)(-s)a^{j-x-1}+R_{x+1,2,j-x-2}(a)\\
                C_{j-x+1}(x-2,2)(-s)^2a^{j-x-1}+R_{x+1,3,j-x-2}(a)\\
                \vdots\\
                C_{j-x+1}(x-2,j-3)(-s)^{j-3}a^{j-x-1}+R_{x+1,x-2,j-x-2}(a)\\
                C_{j-x+1}(x-2,j-2)(-s)^{j-2}a^{j-x-1}+R_{x+1,x-1,j-x-2}(a)
            \end{bmatrix}
\end{align*}

The last equality is justified by Lemma \ref{3.18}.  Therefore, $\floor*{\tfrac{1}{n^j}}T_{n,x+1}(a)$ can be derived from $\floor*{\tfrac{1}{n^j}}T_{n,x}(a)$, which in turn allows us to determine all values of $\floor*{\tfrac{1}{n^m}}T_{n,m}(a)$. In particular, when $j=x+1$, we can get
$$
\floor*{\frac{1}{n^{x+1}}}T_{n,x+1}(a)=(-1)^{x+1}\frac{a-x+1}{k^{x+1}}\mathbf{1}^T\left(\begin{bmatrix}
    a_{x+1}\\a_{x}\\a_{x-1}\\...\\a_2
 \end{bmatrix}\odot
 \begin{bmatrix}
    1\\-s\\s^2\\...\\(-1)^{x-2}s^{x-2}\\(-1)^{x-1}s^{x-1}
 \end{bmatrix}\odot
 \begin{bmatrix}
    C(x-1,0)\\C(x-1,1)\\C(x-1,2)\\...\\C(x-1,x-2)\\C(x-1,x-1)
 \end{bmatrix}\right)
$$
So (\ref{p2}) is proved to be true when $m=x+1$.

\end{proof}

\begin{theorem}[Formula of $R_m(t)$]
$$
R_{m}(t)=\frac{\sum_{j=0}^{m-1}(-1)^j(B_{m-1,j}\cdot k^{m-j-1})}{k^{2m-1}}, k=1+1/t, t\geq 0,
$$
where
$
B_{m,j}=\frac{\binom{2m+2}{m-j}\binom{m+j}{m}}{m+1}
$ and $B_{m,j}$ denotes the Borel's Triangle.
\end{theorem}
\begin{proof}
This follows from expanding the expression of the Theorem \ref{th5} by recalling $a_x=(-1)^x(s^{x-1}+s^{x-2}+\cdots+1)$, where $s=1-1/k$.
{\footnotesize
\begin{align*}
\floor*{\frac{1}{n^m}}Q_{n,m}=&\frac{(-1)^m}{k^m} \mathbf{1}^T\left(\begin{bmatrix}
    a_{m}\\a_{m-1}\\a_{m-2}\\...\\a_2
 \end{bmatrix}\odot
 \begin{bmatrix}
    1\\-s\\s^2\\...\\(-s)^{m-2}
 \end{bmatrix}\odot
 \begin{bmatrix}
    C(m-2,0)\\C(m-2,1)\\C(m-2,2)\\...\\C(m-2,m-2)
 \end{bmatrix}\right) \\
 =&\frac{1}{k^m}\left(\Sigma_{i=0}^{m-1}s^iC(m-2,0)+\Sigma_{i=1}^{m-1}s^iC(m-2,1)+\cdots+\Sigma_{i=m-2}^{m-1}s^iC(m-2,m-2)\right)\\
 =&\frac{1}{k^m}\left(s^0\Sigma_{i=0}^0C(m-1,i)+s^1\Sigma_{i=0}^1C(m-2,i)+\cdots+s^{m-2}\Sigma_{i=0}^{m-2}C(m-2,i)+s^{m-1}\Sigma_{i=0}^{m-2}C(m-2,i)\right)\\
 =&\frac{1}{k^m}\left(s^0C(m-1,0)+s^1C(m-1,1)+\cdots s^{m-1}C(m-1,m-1)\right)\\
 =&\frac{1}{k^m}\left(\binom{0}{0}C(m-1,0)+\left(\binom{1}{0}-\binom{1}{1}\frac{1}{k}\right)C(m-1,1)+\cdots +\Sigma_{i=0}^{m-1}\left(\binom{m-1}{i}(-\frac{1}{k})^i\right)C(m-1,m-1)\right)\\
 =&\frac{1}{k^m}\left(\Sigma_{s=0}^{m-1}\binom{s}{0}C(m-1,s)+\cdots+(-\frac{1}{k})^{m-1}\Sigma_{s=m-1}^{m-1}\binom{s}{m-1}C(m-1,s)\right)
\end{align*}
}
Because $B_{n,k}=\Sigma_{s=k}^n\binom{s}{k}C(n,s)$, we can get $\floor*{\frac{1}{n^m}}Q_{n,m}=\frac{1}{k^m}\Sigma_{j=0}^{m-1}\left((-\frac{1}{k})^jB_{m-1,j}\right)=\frac{\sum_{j=0}^{m-1}(-1)^j(B_{m-1,j}\cdot k^{m-j-1})}{k^{2m-1}}$.
\end{proof}

\begin{ex}[Illustration of $R_m(t)$ and Borel Triangle]
The Borel triangle in row-wise form is given as follows.
\[
\begin{bmatrix}
  1&&&\\
  2&1&&\\
  5&6&2&\\
  14&28&20&5\\
\end{bmatrix}
\]
We can get $R_m(t)$ from $m=1$ to $m=5$ ($k=1+1/t$):
$$
R_{1}(t)=\frac{1}{k}
$$
$$
R_{2}(t)=\frac{2k-1}{k^3}
$$
$$
R_{3}(t)=\frac{5k^2-6k+2}{k^5}
$$
$$
R_{4}(t)=\frac{14k^3-28k^2+20k-5}{k^7}
$$

\end{ex}

\subsection{Corollary \ref{cor:expdelta}}\label{A5}
Here, we consider the case where $\{X_{(i)},i=1,\ldots,n\}$ are sampled from exponential distribution with arbitrary rate parameter, while $\{Y_{(i)},i=1,\ldots,n\}$ are still sampled from unit exponential distribution. The following result is a special case of Theorem \ref{thm:notexp}.
\begin{corollary}
	\label{cor:expdelta}
	Suppose $\{X_{(i)}, i=1,\ldots,n\}$ are the order statistics of $n$ independent exponentially distributed random variables with rate parameter $1+\delta$ with $\delta>-1$. $\{Y_{(i)},i=1,\ldots,n\}$ are the order statistics of $n$ independent unit exponential random variables. Then, assume the index set $\Lambda_m$ satisfy Condition \ref{cond:1}, we have
	\begin{equation}
		\label{qmid2}
		q_{n,\Lambda_m}=\max_{i\in\Lambda_m}\{X_{(i)}/Y_{(i)}\}\overset{P}{\to} \frac{1}{1+\delta},\text{ as }n\to\infty.
	\end{equation}
	For any arbitrary positive integer $k$,
	\begin{equation}
		\label{qright2}
		q_{n,\Lambda_r}\triangleq\max_{i\in\Lambda_r}\{X_{(i)}/Y_{(i)}\}\overset{P}{\to} \frac{1}{1+\delta},\text{ as }n\to\infty,
	\end{equation}
	with $\Lambda_r=\{i\in\mathbb{N}:n-k+1\leq i\leq n\}.$
	
	For any arbitrary positive integer $k$, let $\Lambda_l=\{i\in\mathbb{N}: 1\leq i\leq k\}$, then $$q_{n,\Lambda_l}\triangleq\max_{i\in\Lambda_l}\{X_{(i)}/Y_{(i)}\}$$ converges to a non-degenerate distribution.
\end{corollary}
This corollary directly follows from first three theorems in Section \ref{subsec:limit} due to the fact that if $X$ is a unit exponential random variable, $X/(1+\delta)$ is an exponentially distributed random variable with rate parameter $1+\delta$. Thus, the proof is omitted.

\subsection{Proof of Theorem \ref{thm:notexp}}
\begin{proof}
	For two different distributions with monotone increasing, continuous cumulative distribution functions $F$ and $G$, there exists $y^\ast\in\R$ such that $F^{-1}(y^\ast)\neq G^{-1}(y^\ast)$. Without loss of generality, assume $F^{-1}(y^\ast)< G^{-1}(y^\ast)$. Since $F^{-1}$ and $G^{-1}$ are continuous, there exists $y_0<y^\ast<y_1$ such that $$F^{-1}(y_0)<F^{-1}(y_1)<G^{-1}(y_0)<G^{-1}(y_1).$$
	Denote by $A=\{X_{(i)}\not\in(F^{-1}(y_0),F^{-1}(y_1)),i\in\Lambda_m\}$. Let $$\varepsilon=\frac12\left(\frac{G^{-1}(y_0)}{F^{-1}(y_1)}-1\right),$$
	we have
	\begin{equation*}
		\begin{aligned}
			P\left(\abs{q^{(1)}_{n,\Lambda_m}-1}<\varepsilon,\abs{q^{(2)}_{n,\Lambda_m}-1}<\varepsilon\right)\leq P\left(A\right)+P\left(\abs{q^{(1)}_{n,\Lambda_m}-1}<\varepsilon,\abs{q^{(2)}_{n,\Lambda_m}-1}<\varepsilon,A^c\right).
		\end{aligned}
	\end{equation*}
	
	First we show that the first term converges to $0$. Fix some small $\delta\in(0,\min\{y_0,1-y_1\})$. Then for sufficiently large $n$ such that $\alpha_n+n^{-1}<y_0-\delta<y_1+\delta<\beta_n-n^{-1}$, we have
	\begin{equation*}
		\begin{aligned}
			P(A)=& P(X_{(\floor{\beta_nn})}<F^{-1}(y_1),A) + P(X_{(\ceil{\alpha_nn})}>F^{-1}(y_0),A) \\
			&+ P(X_{(\ceil{\alpha_nn})}\leq F^{-1}(y_0)<F^{-1}(y_1)\leq X_{(\floor{\beta_nn})},A) \\
			\leq& P\left(F_n(X_{(\floor{\beta_nn})})-F(X_{(\floor{\beta_nn})})>\beta_n-n^{-1}-y_1\right) \\
			&+ P\left(F_n(X_{(\floor{\alpha_nn})})-F(X_{(\floor{\alpha_nn})})<\alpha_n+n^{-1}-y_0\right) \\
			&+ P(X_{(i)}\not\in(F^{-1}(y_0),F^{-1}(y_1)), 1\leq i\leq n), \\
		\end{aligned}
	\end{equation*}
	where $F_n(\cdot)$ is the empirical distribution function.
	By the Glivenko-Cantelli theorem \citep{durrett2010probability}, we have $\sup_{x\in\R} \abs{F_n(x)-F(x)}\xrightarrow{a.s.}0$. So
	$$P\left(F_n(X_{(\floor{\beta_nn})})-F(X_{(\floor{\beta_nn})})>\beta_n-n^{-1}-y_1\right)\leq P\left(\sup_{x\in X}\abs{F_n(x)-F(x)}>\delta\right)\xrightarrow{n\to\infty}0,$$
	$$P\left(F_n(X_{(\floor{\alpha_nn})})-F(X_{(\floor{\alpha_nn})})<\alpha_n+n^{-1}-y_0\right)\leq P\left(\sup_{x\in X}\abs{F_n(x)-F(x)}>\delta\right)\xrightarrow{n\to\infty}0.$$
	And notice that since $y_0<y_1$,
	$$P(X_{(i)}\not\in(F^{-1}(y_0),F^{-1}(y_1)),1\leq i\leq n)=(1-y_1+y_0)^n\xrightarrow{n\to\infty}0.$$
	Therefore, we have
	\begin{equation}
		\label{eq:firstconverge}
		P(A)\xrightarrow{n\to\infty}0.
	\end{equation}
	
	Next we show that the second term converges to $0$. When $A^c$ happens, there exists some $i\in\Lambda_m$ such that $X_{(i)}\in(F^{-1}(y_0),F^{-1}(y_1))$. 
	Thus
	\begin{equation*}
		\begin{aligned}
			&\quad \ P\left(\abs{q^{(1)}_{n,\Lambda_m}-1}<\varepsilon,\abs{q^{(2)}_{n,\Lambda_m}-1}<\varepsilon,A^c\right)\\
			&\leq P\left(\frac{Y_{(i)}}{X_{(i)}}<1+\varepsilon,F^{-1}(y_0)<X_{(i)}<F^{-1}(y_1)\right)\\
			&\leq P\left(Y_{(i)}<\frac12(F^{-1}(y_1)+G^{-1}(y_0)),F^{-1}(y_0)<X_{(i)}\right) \\
			&\leq P\left(F(X_{(i)})-G(Y_{(i)})>y_0-G\left[\frac12(F^{-1}(y_1)+G^{-1}(y_0))\right]\right).
		\end{aligned}
	\end{equation*}
	By the Glivenko-Cantelli theorem,
	\begin{equation*}
		\begin{aligned}
			\sup_{1\leq i\leq n}\abs{F(X_{(i)})-G(Y_{(i)})}&\leq\sup_{1\leq i\leq n}\abs{F(X_{(i)})-F_n(X_{(i)})}+\sup_{1\leq i\leq n}\abs{G(Y_{(i)})-G_n(Y_{(i)})}\\
			&\leq \sup_{x\in\R}\abs{F(x)-F_n(x)}+\sup_{x\in\R}\abs{G(x)-G_n(x)}\xrightarrow{a.s.} 0.
		\end{aligned}
	\end{equation*}
	So we have
	\begin{equation}
		\label{eq:secondconverge}
		P\left(\abs{q^{(1)}_{n,\Lambda_m}-1}<\varepsilon,\abs{q^{(2)}_{n,\Lambda_m}-1}<\varepsilon,A^c\right)\to 0.
	\end{equation}
	
	Combine \eqref{eq:firstconverge} and \eqref{eq:secondconverge}, we can conclude that
	\begin{equation}
		\label{qnoconverge}
		P\left(\abs{q^{(1)}_{n,\Lambda_m}-1}<\varepsilon,\abs{q^{(2)}_{n,\Lambda_m}-1}<\varepsilon\right)\xrightarrow{n\to\infty}0.
	\end{equation}
\end{proof}

\subsection{Proof of Theorem \ref{procedure}}
Conditional on $\{Y_i, i=1,\cdots,n\}$, the estimates $\{\widehat{\theta}^j, j=1,\cdots,m\}$ are independent. Furthermore, for each time, $\{X_{(k)}, k=1,\cdots,n\}$ are generated independently from an identical distribution, hence $\{\widehat{\theta}^j, j=1,\cdots,m\}$ are identically distributed. Finally, by the law of large numbers \citep{van2000asymptotic}, we have the order statistics $\widehat{\theta}_{i,(\alpha m/2)}$ and $\widehat{\theta}_{i,((1-\alpha/2)m)}$ converge to the $\alpha/2$ and $1-\alpha/2$ percentiles of the distribution of $\widehat{\theta}_i$ respectively.

\end{appendix}

\end{document}